\documentclass[11pt]{amsart}
\usepackage{}
\usepackage{amssymb}

\usepackage{amsmath}
\usepackage{txfonts}

\input xy
\xyoption{all}

\usepackage{amsfonts}
\usepackage{mathrsfs}

\usepackage{latexsym}
\usepackage{graphicx}
\usepackage{amscd,amssymb,amsmath,amsbsy,amsthm}
\usepackage[all]{xy}
\usepackage[colorlinks,plainpages,backref,urlcolor=blue]{hyperref}
\usepackage{verbatim}

\usepackage{tikz}
\usetikzlibrary{arrows,calc}
\tikzset{
>=stealth',
help lines/.style={dashed, thick},
axis/.style={<->},
important line/.style={thick},
connection/.style={thick, dotted},
}

\usepackage{enumitem, hyperref}
\makeatletter
\def\namedlabel#1#2{\begingroup
    #2%
    \def\@currentlabel{#2}%
    \phantomsection\label{#1}\endgroup
}
\makeatother

\newcommand{\nc}{\newcommand}
\nc{\rnc}{\renewcommand}
\nc{\bb}[1]{{\mathbb #1}}
\nc{\bbA}{\bb{A}}\nc{\bbB}{\bb{B}}\nc{\bbC}{\bb{C}}\nc{\bbD}{\bb{D}}
\nc{\bbE}{\bb{E}}\nc{\bbF}{\bb{F}}\nc{\bbG}{\bb{G}}\nc{\bbH}{\bb{H}}
\nc{\bbI}{\bb{I}}\nc{\bbJ}{\bb{J}}\nc{\bbK}{\bb{K}}\nc{\bbL}{\bb{L}}
\nc{\bbM}{\bb{M}}\nc{\bbN}{\bb{N}}\nc{\bbO}{\bb{O}}\nc{\bbP}{\bb{P}}
\nc{\bbQ}{\bb{Q}}\nc{\bbR}{\bb{R}}\nc{\bbS}{\bb{S}}\nc{\bbT}{\bb{T}}
\nc{\bbU}{\bb{U}}\nc{\bbV}{\bb{V}}\nc{\bbW}{\bb{W}}\nc{\bbX}{\bb{X}}
\nc{\bbY}{\bb{Y}}\nc{\bbZ}{\bb{Z}}
\nc{\mbf}[1]{{\mathbf #1}}
\nc{\bfA}{\mbf{A}}\nc{\bfB}{\mbf{B}}\nc{\bfC}{\mbf{C}}\nc{\bfD}{\mbf{D}}
\nc{\bfE}{\mbf{E}}\nc{\bfF}{\mbf{F}}\nc{\bfG}{\mbf{G}}\nc{\bfH}{\mbf{H}}
\nc{\bfI}{\mbf{I}}\nc{\bfJ}{\mbf{J}}\nc{\bfK}{\mbf{K}}\nc{\bfL}{\mbf{L}}
\nc{\bfM}{\mbf{M}}\nc{\bfN}{\mbf{N}}\nc{\bfO}{\mbf{O}}\nc{\bfP}{\mbf{P}}
\nc{\bfQ}{\mbf{Q}}\nc{\bfR}{\mbf{R}}\nc{\bfS}{\mbf{S}}\nc{\bfT}{\mbf{T}}
\nc{\bfU}{\mbf{U}}\nc{\bfV}{\mbf{V}}\nc{\bfW}{\mbf{W}}\nc{\bfX}{\mbf{X}}
\nc{\bfY}{\mbf{Y}}\nc{\bfZ}{\mbf{Z}}
\nc{\bfa}{\mbf{a}}\nc{\bfb}{\mbf{b}}\nc{\bfc}{\mbf{c}}\nc{\bfd}{\mbf{d}}
\nc{\bfe}{\mbf{e}}\nc{\bff}{\mbf{f}}\nc{\bfg}{\mbf{g}}\nc{\bfh}{\mbf{h}}
\nc{\bfi}{\mbf{i}}\nc{\bfj}{\mbf{j}}\nc{\bfk}{\mbf{k}}\nc{\bfl}{\mbf{l}}
\nc{\bfm}{\mbf{m}}\nc{\bfn}{\mbf{n}}\nc{\bfo}{\mbf{o}}\nc{\bfp}{\mbf{p}}
\nc{\bfq}{\mbf{q}}\nc{\bfr}{\mbf{r}}\nc{\bfs}{\mbf{s}}\nc{\bft}{\mbf{t}}
\nc{\bfu}{\mbf{u}}\nc{\bfv}{\mbf{v}}\nc{\bfw}{\mbf{w}}\nc{\bfx}{\mbf{x}}
\nc{\bfy}{\mbf{y}}\nc{\bfz}{\mbf{z}}

\newcommand{\G}{\mathbb{G}}
\newcommand{\op}{\text{op}}

\nc{\mcal}[1]{{\mathcal #1}}
\nc{\calA}{\mcal{A}}\nc{\calB}{\mcal{B}}\nc{\calC}{\mcal{C}}\nc{\calD}{\mcal{D}}
\nc{\calE}{\mcal{E}} \nc{\calF}{\mcal{F}}\nc{\calG}{\mcal{G}}\nc{\calH}{\mcal{H}}
\nc{\calI}{\mcal{I}}\nc{\calJ}{\mcal{J}}\nc{\calK}{\mcal{K}}\nc{\calL}{\mcal{L}}
\nc{\calM}{\mcal{M}}\nc{\calN}{\mcal{N}}\nc{\calO}{\mcal{O}}\nc{\calP}{\mcal{P}}
\nc{\calQ}{\mcal{Q}}\nc{\calR}{\mcal{R}}\nc{\calS}{\mcal{S}}\nc{\calT}{\mcal{T}}
\nc{\calU}{\mcal{U}}\nc{\calV}{\mcal{V}}\nc{\calW}{\mcal{W}}\nc{\calX}{\mcal{X}}
\nc{\calY}{\mcal{Y}}\nc{\calZ}{\mcal{Z}}
\nc{\fA}{\frak{A}}\nc{\fB}{\frak{B}}\nc{\fC}{\frak{C}} \nc{\fD}{\frak{D}}
\nc{\fE}{\frak{E}}\nc{\fF}{\frak{F}}\nc{\fG}{\frak{G}}\nc{\fH}{\frak{H}}
\nc{\fI}{\frak{I}}\nc{\fJ}{\frak{J}}\nc{\fK}{\frak{K}}\nc{\fL}{\frak{L}}
\nc{\fM}{\frak{M}}\nc{\fN}{\frak{N}}\nc{\fO}{\frak{O}}\nc{\fP}{\frak{P}}
\nc{\fQ}{\frak{Q}}\nc{\fR}{\frak{R}}\nc{\fS}{\frak{S}}\nc{\fT}{\frak{T}}
\nc{\fU}{\frak{U}}\nc{\fV}{\frak{V}}\nc{\fW}{\frak{W}}\nc{\fX}{\frak{X}}
\nc{\fY}{\frak{Y}}\nc{\fZ}{\frak{Z}}
\nc{\fa}{\frak{a}}\nc{\fb}{\frak{b}}\nc{\fc}{\frak{c}} \nc{\fd}{\frak{d}}
\nc{\fe}{\frak{e}}\nc{\fFf}{\frak{f}}\nc{\fg}{\frak{g}}\nc{\fh}{\frak{h}}
\nc{\fri}{\frak{i}}\nc{\fj}{\frak{j}}\nc{\fk}{\frak{k}}\nc{\fl}{\frak{l}}
\nc{\fm}{\frak{m}}\nc{\fn}{\frak{n}}\nc{\fo}{\frak{o}}\nc{\fp}{\frak{p}}
\nc{\fq}{\frak{q}}\nc{\fr}{\frak{r}}\nc{\fs}{\frak{s}}\nc{\ft}{\frak{t}}
\nc{\fu}{\frak{u}}\nc{\fv}{\frak{v}}\nc{\fw}{\frak{w}}\nc{\fx}{\frak{x}}
\nc{\fy}{\frak{y}}\nc{\fz}{\frak{z}}

\newtheorem{theorem}{Theorem}[section]
\newtheorem{lemma}[theorem]{Lemma}
\newtheorem{corollary}[theorem]{Corollary}
\newtheorem{prop}[theorem]{Proposition}

\theoremstyle{definition}
\newtheorem{definition}[theorem]{Definition}
\newtheorem{example}[theorem]{Example}
\newtheorem{remark}[theorem]{Remark}

\newtheorem{assumption}[theorem]{Assumption}

\newtheorem{thm}{Theorem}

 \DeclareMathOperator{\id}{id}

 \DeclareMathOperator{\GL}{GL}
\DeclareMathOperator{\Hom}{{Hom}}

 \DeclareMathOperator{\Lie}{Lie}
\DeclareMathOperator{\Spec}{{Spec}} 
\DeclareMathOperator{\Aut}{Aut}
 \DeclareMathOperator{\End}{End}
\DeclareMathOperator{\sEnd}{{\mathcal{E}nd}}

\DeclareMathOperator{\Coh}{Coh}

\DeclareMathOperator{\Mod}{Mod\hbox{-}}
\DeclareMathOperator{\Gm}{\bbG_m}

\DeclareMathOperator{\rat}{rat}

\DeclareMathOperator{\FM}{FM}

\DeclareMathOperator{\Ell}{Ell}
\DeclareMathOperator{\Pic}{Pic}
\DeclareMathOperator{\Th}{Th}
\DeclareMathOperator{\sph}{sph}
\DeclareMathOperator{\ext}{ext}
\DeclareMathOperator{\loc}{loc}

\DeclareMathOperator{\Gr}{Gr}

\DeclareMathOperator{\SL}{SL}
\DeclareMathOperator{\coop}{coop}
\DeclareMathOperator{\Skl}{Skl}

\DeclareMathOperator{\pr}{pr}
\newcommand {\Omit}[1]{}

\newcommand{\sO}{\mathcal{O}}

\newcommand{\h}{\mathfrak{h}}
\newcommand{\g}{\mathfrak{g}}

\newcommand{\catA}{\mathfrak{A}}
\DeclareMathOperator{\Rep}{Rep}
\DeclareMathOperator{\Res}{Res}

\newcommand{\inj}{\hookrightarrow}

\def\angl#1{{\langle #1\rangle}}

\newcommand{\pt}{\text{pt}}

\newcommand{\PP}{\bbP}
\newcommand{\ZZ}{\bbZ}

\newcommand{\Z}{\bbZ}
\newcommand{\C}{\bbC}

\newcommand{\N}{\bbN}

\DeclareMathOperator{\fac}{fac}
\DeclareMathOperator{\Sh}{Sh}
\DeclareMathOperator{\inc}{in}
\DeclareMathOperator{\out}{out}
\newcommand{\SH}{\calS\calH}

\newcount\cols
{\catcode`,=\active\catcode`|=\active
 \gdef\Young(#1){\hbox{$\vcenter
 {\mathcode`,="8000\mathcode`|="8000
  \def,{\global\advance\cols by 1 &}%
  \def|{\cr
        \multispan{\the\cols}\hrulefill\cr
        &\global\cols=2 }%
  \offinterlineskip\everycr{}\tabskip=0pt
  \dimen0=\ht\strutbox \advance\dimen0 by \dp\strutbox
  \halign
   {\vrule height \ht\strutbox depth \dp\strutbox##
    &&\hbox to \dimen0{\hss$##$\hss}\vrule\cr
    \noalign{\hrule}&\global\cols=2 #1\crcr
    \multispan{\the\cols}\hrulefill\cr%
   }
 }$}}
}

\title[Quiver varieties and elliptic quantum groups]
{Quiver varieties and elliptic quantum groups}
\date{\today}

\author[Y.~Yang]{Yaping~Yang}
\address{School of Mathematics and Statistics, The University of Melbourne, 813 Swanston Street, Parkville VIC 3010, Australia}\email{yaping@math.umass.edu}

\author[G.~Zhao]{Gufang~Zhao}
\address{University of Massachusetts, Department of Mathematics and Statistics, 710 N. Pleasant St., Amherst MA 01003}\email{zhao@math.umass.edu}

\begin{document}
\begin{abstract}
We define a sheafified elliptic quantum group for any symmetric Kac-Moody Lie algebra. This definition is naturally obtained from the elliptic cohomological Hall algebra of a preprojective algebra.
The sheafified elliptic quantum group is an algebra object in a certain monoidal category of coherent sheaves on the colored Hilbert scheme of an elliptic curve. 

We show that the elliptic quantum group acts on the equivariant elliptic cohomology of Nakajima quiver varieties.
This action is compatible with the action induced by Hecke correspondence, a construction similar to that of Nakajima. The elliptic Drinfeld currents are obtained as generating series of certain rational sections of the sheafified elliptic quantum group. We show that the Drinfeld currents satisfy the commutation relations of the dynamical elliptic quantum group studied by Felder and Gautam-Toledano Laredo.

\end{abstract}
\maketitle
\tableofcontents

\section{Introduction}
In \cite{F,F1}, Felder constructed the elliptic $R$-matrix,
 which is an elliptic solution to the dynamical Yang-Baxter equations. This $R$-matrix
is related to moduli of bundles on an elliptic curve, WZW conformal field theory on a torus, and
the elliptic integrable systems. 
Representation of elliptic quantum group is  an interesting
subject which recently gained more attentions. Gautam-Toledano Laredo \cite{GTL15} defined the
category of finite dimensional representations of the elliptic quantum groups and studied this category
using $q$-difference equations, while in the $\fs\fl_2$-case the representations have been studied by
Felder-Varchenko \cite{FV}. Aganagic-Okounkov \cite{AO} constructed an action of certain elliptic $R$-matrix
on elliptic cohomology of quiver varieties using stable envelope construction in the elliptic setting, and more concretely using Felder's $R$-matrix in \cite{FRV} in type-$A$.

However, in studying representations of the elliptic quantum group, an algebra containing the
currents on the elliptic curve, referred to as the Drinfeld realization of the elliptic quantum group,
is needed. So far there has been no intrinsic definition of an elliptic quantum group, whose currents
have the desired convergence property, nor a construction of the action of any form of the Drinfeld
realization on the equivariant elliptic cohomology of quiver varieties. Various results towards this direction have been achieved in  \cite{ES, AO, Konn16, GTL15}. A more detailed, but still far from being complete discussion of the historical developments of elliptic quantum group is summarized in \S~\ref{subsec:intr_survey}. 
Therefore, in the present paper, we introduce and initiate the study of a sheafified elliptic
quantum group and achieve the aforementioned goals.

In the first part of the present paper, we give the definition of the sheafified elliptic quantum group for any symmetric Kac-Moody Lie algebra $\fg$, as an algebra object in a certain monoidal category of coherent sheaves on colored Hilbert scheme of an elliptic curve. The space of meromorphic sections of this sheafified elliptic quantum group is an associative algebra. We find explicit generating series in this algebra, referred to as the {\it Drinfeld currents}, which deform the classical elliptic currents in the Manin pair of Drinfeld \cite{Dr}.  We also compute explicitly the commutation relations of the Drinfeld currents. These relations are also found in \cite{GTL15} by Gautam-Toledano Laredo via an entirely different approach.

The main tool we use in constructing the sheafified elliptic quantum group is the preprojective  cohomological Hall algebra,  developed by the authors in \cite{YZ, YZ2}, inspired by earlier work of Kontsevich-Soibelman \cite{KS} and Schiffmann-Vasserot \cite{SV2}. This construction will give, for any quiver $Q$ and any 1-dimensional affine algebraic group $\bbG$ (or formal group), a quantum affine algebra associated to the Kac-Moody Lie algebra of $Q$. In the present paper, we extend this construction to the case when $\bbG$ is an elliptic curve. In this special case, there naturally appears a non-standard monoidal structure on the category of coherent sheaves on the colored Hilbert scheme of $\bbG$. The classical limit of this monoidal structure also shows up in the study of the global semi-infinite loop Grassmannian over $\bbG$. 

In the second part of this paper, we provide a geometric interpretation of the sheafified elliptic quantum group, as the equivariant elliptic cohomology of the moduli of representations of the preprojective algebra of $Q$, where $Q$ is the Dynkin quiver of $\fg$.  The equivariant elliptic cohomology theory we use was introduced by Grojnowski \cite{Gr} and Ginzburg-Kapranov-Vasserot \cite{GKV95}, and was investigated by many others later on, including  \cite{And00, An,Chen10,Gep,GH, Gan,Lur}. 
It had been an open question since the construction of Nakajima \cite{Nak01} and Varagnolo \cite{Va00}, that a Drinfeld realization of the elliptic quantum group should act on the equivariant elliptic cohomology of quiver varieties. In the present paper, we deduce this from the geometric interpretation of the sheafified elliptic quantum group. This also provides a geometric construction of highest weight representations of the elliptic quantum group.

\subsection{The colored Hilbert scheme}
One of the unexpected features in the study of elliptic quantum groups via the present approach is the occurrence of a non-standard monoidal structure on the category of coherent sheaves on colored Hilbert scheme.  Although also present in the rational and trigonometric case, this monoidal structure is not visible in those cases due to the lack of non-trivial line bundles. The classical limit of this moniodal structure already shows up implicitly in the study of semi-infinite loop Grassmannian and locality structures in \cite{Mirk}. 

Let $I$ be the set of simple roots of $\fg$, and $v=(v^i)_{i\in I}\in \N^I$. Let $E$ be an elliptic curve. 
Recall that the Hilbert scheme of $I$-colored points on $E$ is 
\[
\calH_{E\times I}=\coprod_{v\in\bbN^I}E^{(v)},\,\ \text{ where $E^{(v)}:=\prod_{i\in I}E^{v^i}/\fS_{v^i}$}.
\]
One can think an element of $E^{(v)}$ as $\sum_{i\in I}v^i$ points on $E$, where $v^i$ points have color $i\in I$. Let $\calC$ be the abelian category of quasi-coherent sheaves on $\calH_{E\times I}$. An object of $\calC$ consists of tuples $(\mathcal{F}_{v})_{v\in \bbN^I}$,  each $\mathcal{F}_{v}$ is a quasi-coherent sheaf on $E^{(v)}$. A morphism of $\calC$ is naturally defined as a morphism of sheaves $\mathcal{F}_{v}\to \mathcal{G}_{v}$ on $E^{(v)}$, for each $v\in \bbN^I$.

We introduce a two-parameter family of monoidal structures on $\calC$, with the two deformation parameters $(t_1,t_2)\in E^2$. For two objects $\mathcal{F}=\{\calF_{v}\}_{v\in\bbN^I}, \mathcal{G}=\{\calG_v\}_{v\in\bbN^I}$ in $\mathcal{C}$, the tensor $\mathcal{F}\otimes_{t_1, t_2} \mathcal{G}$ is defined as
\begin{equation}\label{equ:tensor on C}
(\mathcal{F}\otimes_{t_1, t_2} \mathcal{G})_{v}:=\bigoplus_{v_1+v_2=v}
(\mathbb{S}_{v_1, v_2})_{*}\Big( (\mathcal{F}_{v_1} \boxtimes 
 \mathcal{G}_{v_1})\otimes \mathcal L_{v_1, v_2}\Big),
\end{equation}
where $\mathbb{S}_{v_1, v_2}$ is the symmetrization map 
$\mathbb{S}_{v_1, v_2}: E^{(v_1)}\times E^{(v_2)}
\to E^{(v_1+v_2)}$, and $\mathcal L_{v_1, v_2}$ is some line bundle on $E^{(v_1)}\times E^{(v_2)}$  depending on the parameters $t_1, t_2\in E^2$, described in detail in \S\ref{sec:category C}.

\begin{thm}[Theorems~\ref{thm:monoidal} and \ref{thm:braiding}]\label{thm:ellCoHA}
\Omit{\begin{enumerate}
\item }The abelian category $\calC$, endowed with $\otimes_{t_1,t_2}$, is a monoidal category with two parameters $t_1, t_2$. There is a meromorphic braiding $\gamma$, making $(\calC,\otimes_{t_1,t_2},\gamma)$ a symmetric monoidal category.
\end{thm}

\subsection{The elliptic cohomological Hall algebra}
For any compact Lie group $G$, let $\mathfrak{A}_G$ be the moduli scheme of semistable principal $G^{alg}$-bundles over an elliptic curve where $G^{alg}$ is the associated split algebraic group.
For a $G$-variety $X$, the $G$-equivariant elliptic cohomology $\Ell_{G}(X)$ is a quasi-coherent sheaf of $\mathcal{O}_{\mathfrak{A}_G}$-module, satisfying certain axioms, see \cite{GKV95}. In particular, when $v=(v^i)_{i\in I}$, and $G=U_{v}:=\prod_{i\in I} U_{v^i}$, we have $\mathfrak{A}_{U_v}=E^{(v)}$, the colored Hilbert scheme of $v$-points on $E$. Without raising confusion,  for simplicity we use the notations $\mathfrak{A}_{\GL_v}$ and $\Ell_{\GL_v}$ instead of $U_n$.

Let $Q$ be the Dynkin quiver of $\fg$,  with  the set of vertices $I$ and the set of arrows $H$. Let $Q\cup Q^{op}$ be the double quiver (see \S~\ref{sec:CoHA}). 
The preprojective algebra, denoted by $\Pi_Q$, is the quotient of the path algebra $\C( Q\cup Q^{op})$ by the ideal generated by the relation $[x, x^{op}]=0$. 
Let $\Rep(\Pi_Q, v)$ be the representation space of $\Pi_Q$ with dimension vector $v=(v^i)_{i\in I}$.  
It is an affine variety endowed with a natural action of $G_{v}=\prod_{i\in I}\GL_{v^i}$.  The elliptic cohomological Hall algebra (CoHA), denoted by $\calP_{\Ell}(Q)$, is the collection of sheaves $\calP_{\Ell, v}:=\Ell_{G_{v}\times \C^*}(\Rep(\Pi_Q, v))$ on $E^{(v)}$, for $v\in \N^I$ as an object in $\calC$. It is endowed with a multiplication, which is
 a morphism of sheaves
$\calP_{\Ell, v_1} \otimes_{t_1, t_2} 
\calP_{\Ell, v_2}
\to \calP_{\Ell, v_1+v_2}$ on $E^{(v_1+v_2)}$ for any $v_1, v_2$. 
The above multiplication,  referred to as the Hall multiplication, makes $\calP_{\Ell}(Q)$ into an algebra object in $\calC$. 

We also construct a coproduct $\Delta: \calP_{\Ell} \to (\calP_{\Ell}\otimes_{t_1,t_2} \calP_{\Ell})_{\loc}$ on a suitable localization of $\calP_{\Ell}$ (see \S~\ref{sec:CoHA}). 

For each $k\in I$, let $e_k$ be the dimension vector valued $1$ at vertex $k$ of the quiver $Q$ and zero otherwise. The \textit{spherical subsheaf}, denoted by $\calP^{\sph}$, is the subsheaf of $\calP_{\Ell}$ generated by $\calP_{e_k}$ as $k$ varies in $I$. 
When restricting on $\calP^{\sph}$, the coproduct $\Delta$ is well-defined without taking localization. Therefore, we have the following. 
\begin{thm}[Corollary~\ref{cor:bialg}]
The object $(\calP_{\Ell}^{\sph}(Q), \star, \Delta)$,  endowed with the Hall multiplication $\star$, and coproduct $\Delta$, is a bialgebra object in $\mathcal{C}$.  
\end{thm}
\subsection{The relation with loop Grassmannians}
When $t_1=t_2=0$, the classical limit of the elliptic CoHA is related to the global loop Grassmannian on the elliptic curve $E$. 

Mirkovi\'c recently gave a construction of loop Grassmannians in the framework of \textit{local spaces} \cite{Mirk}.  A notion of {\it locality} structure is developed, as a refined version of the factorization structure of Beilinson-Drinfeld. 
In particular, line bundles on $\calH_{E\times I}$ with locality structures are in correspondence with quadratic forms on $\bbZ^I$. When the quadratic form is the adjacency matrix of $Q$, the locality structure on the corresponding line bundle gives the classical limit of the monoidal structure $\otimes_{t_1=t_2=0}$ (see \S~\ref{subsec:classical limit} for the details).

Using the results of \cite{Mirk}, the adjacency matrix (or equivalently, the quadratic form) gives a loop Grassmannian $\mathcal{G}r$ over $\mathcal{H}_{E\times I}$, which becomes the Beilinson-Drinfeld Grassmannian when $Q$ is of type $A,D,E$.
The tautological line bundle $\mathcal{O}_{\mathcal{G}r}(1)$ is endows with a natural local structure.
There is a \textit{Zastava space} $\mathcal{Z} \subset \mathcal{G}r$, which is a local subspace over $\mathcal{H}_{E\times I}$. 
Taking certain components in a torus-fixed loci of $\mathcal{Z}$ gives a section $\mathcal{H}_{E\times I}\inj \calZ$. In particular, the restriction of $\calO_{\mathcal{G}r}(1)$ to $\mathcal{H}_{E\times I}$ recovers the local line bundle. 
\begin{thm}[Corollary~\ref{cor:Zastava}, Proposition~\ref{prop:local_algebra}]
\begin{enumerate}
\item
The classical limit $\calP^{\sph}|_{t_1=t_2=0}$ of the spherical subalgebra $\calP^{\sph}$ 
is the local line bundle $\calO_{\mathcal{G}r}(1)|_{\calH_{E\times I}}$. 
\item The algebra structure on $\calP^{\sph}|_{t_1=t_2=0}$ is equivalent to the locality structure on $\mathcal{O}_{\mathcal{G}r}(1)|_{ \mathcal{H}_{E\times I}}$. 
\end{enumerate}
\end{thm}

\subsection{The elliptic quantum group}
Let $E$ be the elliptic curve over $\calM_{1, 2}$. We can consider $E$  as a family of complex elliptic curves naturally endowed with the Poincar\'e line bundle $\mathbb{L}$. Let $\calP_{\Ell}(Q)$ be the elliptic CoHA associated to this elliptic curve $E$ and an arbitrary quiver $Q$. 
For the dimensional vector $e_k$, $k\in I$, $\calP_{e_k}$ is the same as the Poincar\'e line bundle $\mathbb{L}$ over $E^{(e_k)}=E$. 
We consider certain meromorphic sections of $\bbL$, namely,  meromorphic functions $f(z)$ on $\C$, holomorphic on $\C\backslash (\bbZ+\tau\bbZ)$, such that:
\[
f(z+1)=f(z), \,\ f(z+\tau )=e^{2\pi i \lambda } f(z).
\]
A basis of these meromorphic sections can be chosen as $\left\{ g^{(i)}_\lambda(z):=\frac{1}{i!} \frac{\partial^i}{\partial z^i}\left( \frac{\vartheta(z+\lambda)}{\vartheta(z)\vartheta(\lambda) }\right)\right\}_{i\in \N}$, where $\vartheta(z)$ is the Jacobi theta function in Example \ref{example:theta}. 

We have a functor $\Gamma_{\rat}$ of taking certain rational sections from the category $\calC$ to the category of vector spaces \S\ref{subsec:rational sec}. For simplicity, we denote $\Gamma_{\rat}(\calP)$ by $\mathbf{P}$. Let $\lambda=(\lambda_{k})_{k\in I}\in E^I$. 
Consider the following generating series
\begin{equation*}
\mathfrak{X}_{k}^+(u, \lambda):=
\sum_{i=0}^{\infty} g^{(i)}_{\lambda_k}(z_k) u^{i}, 
\end{equation*}
 of certain meromorphic sections of $\calP_{\Ell}^{\sph}$, let $\mathfrak{X}_{k}^-(u, \lambda)$ be the corresponding series in the opposite algebra $\calP_{\Ell}^{\text{coop}}$, and $\Phi_{k}(u)$ a generating series of meromorphic sections of the Cartan subalgebra (see \S \ref{subsec:rational sec} for the details).  
Therefore, $\mathfrak{X}_{k}^{\pm}(u, \lambda), \Phi_{k}(u)$ are elements in $D(\mathbf{P}^{\sph})[\![u]\!]$, where $D(\mathbf{P}^{\sph}):=\mathbf{P}^{\sph}\otimes \mathbf{P}^{\sph, \coop}$ is the Drinfeld double of $\mathbf{P}^{\sph}$. 

\begin{thm}[Theorem~\ref{thm:generating series relation}]\label{thm:rat_sec_ellCoHA}
The Drinfeld double $D(\mathbf{P}^{\sph})$ satisfies relations of the elliptic quantum group. 
In other words, the series  $\mathfrak{X}_{k}^{\pm}(u, \lambda)$, and $\Phi_{k}(u)$ of $D(\mathbf{P}^{\sph})$ satisfy the commutation relations \eqref{EQ1}-\eqref{EQ5}.
\end{thm}
Gautam-Toledano Laredo in \cite{GTL15} defined a category of representations of the elliptic Drinfeld currents imposing the same  commutation relations \eqref{EQ1}-\eqref{EQ5}.
The relation between the category of elliptic Drinfeld currents studied in {\it loc. cit.} and category of representations of Felder's elliptic $R$-matrices is partially clarified in \cite{Gau}, via an explicit calculation of Gaussian decomposition of Felder's elliptic $R$-matrix.

Motivated by Theorem~\ref{thm:rat_sec_ellCoHA}, we define the sheafified elliptic quantum group to be the Drinfeld double of the elliptic CoHA $\calP_{\Ell}^{\sph}$. The details and technicality in the construction of the Drinfeld double are explained in \S~\ref{subsec:DrinfeldDouble}.
This definition of sheafified elliptic quantum group not only gives a conceptual understanding of the kind of algebra object the elliptic quantum group is, but also makes various features of this quantum group more transparent. In particular, we show that the braiding in the category $\calC$ naturally gives the conjugation action of the Cartan subalgebra of the elliptic quantum group on its positive part. Moreover, the dynamical parameters also naturally show up in the reconstruction of the Cartan, as is explained in detail in \S~\ref{subsec:dyn_cartan}.

\subsection{Representations of the elliptic quantum group}\label{subsec:FelderEll}
Let $\mathfrak{M}(w)$ be the Nakajima quiver varieties associated to the quiver $Q$ with framing $w\in \N^I$. We show that the equivariant elliptic cohomology $\Ell^*_{G_w \times\Gm}( \mathfrak{M}(w))$ of $\mathfrak{M}(w)$ is a Drinfeld-Yetter module of $\calP_{\Ell}(Q)$, for any $w\in \N^I$. 
In particular, the sheafified elliptic quantum group acts on $\Ell^*_{G_w \times\Gm}( \mathfrak{M}(w))$. 

More precisely, we introduce a category $\calD_w$ as the module category of $\calC$ with highest weight no more than $\sum_{i\in I}w_i\overline{\omega_i}$. 
The equivariant elliptic cohomology $\Ell^*_{G_w \times\Gm}( \mathfrak{M}(w))$ lies in $\calD_w$. We prove the following.
\begin{thm}[Proposition~\ref{lem:M_PMod} and Theorem~\ref{thm:DoubAct}]
 \begin{enumerate}
\item
For each $w\in \bbN^I$, the elliptic cohomology of quiver varieties $\Ell_{G_w}(\mathfrak{M}(w))$ is a module object over the Drinfeld double  $D(\calP_{\Ell}^{\sph})$.
\item The action of $D(\calP_{\Ell}^{\sph})$ on $\Ell_{G_w}(\mathfrak{M}(w))$ is compatible with the 
Hecke relation by Nakajima. In other words,  the action of elements in $\calP_{e_k}$
are given by convolution with certain characteristic classes of the tautological line bundle on the Hecke correspondence $C_{k}^+\subset \mathfrak{M}(v, w)\times \mathfrak{M}(v+e_k, w)$.
\end{enumerate}
\end{thm}

In \S~\ref{subsec:DrinfeldDouble}, the category $\Mod_fD(\calP_{\Ell}^{\sph})$ of finite dimensional representations  of the elliptic quantum group is defined.
We also introduce the notion of the elliptic Drinfeld polynomials for highest weight modules. 

Similar to the algebra object itself, the module category also have an equivalent description involving dynamical parameters (Theorem~\ref{thm:DynParam}), in which the braiding gives the action of the Cartan subalgebra. This is a conceptual explanation that the representation category of the elliptic quantum group, as an abelian category, does not depend on the dynamical parameters.

\subsection{A quick literature survey}\label{subsec:intr_survey}
After the formula of the dynamical elliptic $R$-matrix has been found by Felder,  study of the elliptic Drinfeld currents and representations have been initiated. In the $\fs\fl_2$-case, an algebra containing the elliptic Drinfeld currents has been found in \cite{ER}. The relation with Felder's $R$-matrix is in \cite{FE}.
Representation theory of this algebra is studied in \cite{FV}. 
Etingof-Schiffmann \cite{ES,ES2} also studied representations of the elliptic $R$-matrices via a dynamical twist procedure of Babelon-Bernard-Billey \cite{BBB}.

We note that when study representations of the $R$-matrices, an algebra or an algebra object $E(\fg)$ can only be defined through an extrinsic embedding $E(\fg)\inj \End(\prod V_i)$ for some family of vector spaces $V_i$ with $\prod V_i$ possibly infinite dimensional. When doing explicit calculations where a closed formula of $R$-matrix is needed,  an embedding $\fg\inj \fg\fl_N$ is required for some finite number $N$. This is less canonical unless $\fg$ is of type-$A $. Hence, a Drinfeld-type realization is more handy is studying representations. 

A systematic construction of Drinfeld realizations for more general class of Lie algebras has recently been achieved by Konno  \cite{Konn16} and Jimbo-Konno-Odake-Shiraishi \cite{33}. The Drinfeld currents of these algebra are constructed in Farghly-Konno-Oshima \cite{22}.
However, these Drinfeld currents coming from this presentation do not know to have desired
convergence properties on representations. 

Gautam-Toledano Laredo recently studied representations of the elliptic quantum group in \cite{GTL15}, partially extending the results of \cite{FV} to other types of Lie algebras, yet via an entirely different approach.  Instead of defining an algebra, they defined a representation of the elliptic quantum group as a vector space endowed with certain meromorphic operators, which turn out to be the Drinfeld currents defined in the present paper. In \cite{Gau}, it is shown that the commutation relations of these meromorphic operators are satisfied by the $RTT$-relation of Felder.

Geometrically, it has been expected that the elliptic quantum group acts on the equivariant elliptic cohomology of Nakajima  quiver varieties. The later has been realized recently by Aganagic-Okounkov as a representation of the Felder's elliptic $R$-matrix, via the elliptic stable envelope construction. This stable envelope construction is an extension of the construction of Yangian of Maulik-Okounkov. Note that in the Yangian case the comparison of the stable envelope construction and the Kac-Moody Yangian is intricate outside of $A,D,E$ types. Therefore, we leave the precise relation between out construction of the elliptic quantum group and the elliptic stable envelope construction to  future investigations.

Another even earlier geometric study of elliptic quantum group was carried out by Feigin-Odesskii \cite{FO}. This elliptic quantum group has a lot of geometric applications, including quantization of moduli space of bundles on elliptic curves.  It also has a description in terms of a shuffle algebra, which does have the dynamical parameter. However, it is not clear to us at represent the precise relation between the Feigin-Odesskii algebra and the one studied in the present paper, where the later naturally acts on elliptic cohomology of quiver variety, and eventually is related to the Felder's $R$-matrix. More precisely, in the simply-laced type, the shuffle factor occurred in \cite{FO} resembles the factor in \cite[\S~2]{YZ}, which is the CoHA of the path algebra in lieu of the preprojective algebra. Also, the structure similar to the Sklyanin algebra already shows up in the positive part of the algebra in \cite{FO}, but only occurs in the Cartan part in the algebra studied in the present paper.

\subsection*{Acknowledgement}
During the preparation of this paper, the authors received helps from many people, an incomplete list includes Sachin Gautam, Marc Levine, Alina Marian, Ivan Mirkovi\'c, Valerio Toledano Laredo, Eric Vasserot, and Changlong Zhong.  The authors are grateful to Sachin Gautam  for access to work in progress, and for sharing some calculations related to this paper, which significantly encouraged the authors at the initial stage of their investigation. Part of the work was done when both authors were visiting Universit\"at Duisburg-Essen, May-August, 2016, and the Max-Planck Institut f\"ur Mathematik in Bonn, June-July, 2017.

\section{Reminder on Theta functions}
We start by fixing some notations and terminologies about line bundles on abelian varieties. 
\subsection{Theta functions}
Let $E$ be an elliptic curve over $S$, for some scheme  $S$ of finite type over a field of characteristic zero. Let $0:S\to E$ be the zero section. Denote by $\{0\}$ the image of $0$, which is a codimension one subvariety of $E$. Let $\calO(-\{0\})$ be the ideal sheaf of $\{0\}$, which is a line bundle. Its dual $\calO(\{0\})$ has a natural section, denoted by $\vartheta$.
Let $i:E\to E$ be the involution sending any point to its additive inverse. Then $i^*\calO(\{0\})\cong\calO(\{0\})$, and the natural section $\vartheta$ is sent to $-\vartheta$ under $i$. In this sense, we say that $\vartheta$ is an odd function.

\begin{example} 
\label{example:theta}
Fix $\tau\in \C \backslash \mathbb{R}$, let $\mathfrak{H}$ be the upper half plane, i.e. $\mathfrak{H}:=\{z \in \C \mid \text{Im}(z)> 0\}$.
When $E$ is a complex elliptic curve, i.e., $\bbC/(\bbZ+\tau\bbZ)$, up to normalization, 
$\vartheta(z| \tau)$ is the function uniquely characterized by the following properties:
\begin{enumerate}
\item
$\vartheta(z| \tau )$ is a holomorphic function $\C\times \mathfrak{H} \to \C$, such that $\{z\mid\vartheta(z| \tau )=0\}=\bbZ+\tau\bbZ$.
\item
$\frac{\partial \vartheta}{\partial z}(0| \tau)=1$.
\item
$\vartheta(z+1| \tau )=-\vartheta(z| \tau )=\vartheta(-z| \tau )$, and\,\  $\vartheta(z+\tau| \tau )=-e^{-\pi i \tau}e^{-2 \pi i z}
\vartheta(z| \tau )$.
\item
$\vartheta(z| \tau+1)=\vartheta(z| \tau )$, while $\vartheta(-z/\tau|-1/\tau)=-(1/\tau)e^{(\pi i/\tau)z^2}\vartheta(z| \tau )$.
\item Let $q:=e^{2\pi i \tau}$ and $\eta(\tau):=q^{1/24}\prod_{n\geq 1}(1-q^n)$. If we set $\theta(z|\tau):=\eta(\tau)^3 \vartheta(z|\tau)$, then $\theta(z|\tau)$ satisfies the differential equation:
$
\frac{\partial \theta(z, \tau)}{\partial \tau}
=\frac{1}{4\pi i}\frac{\partial^2 \theta(z, \tau)}{\partial z^2}.$
\end{enumerate}
\end{example}

\subsection{Line bundles on abelian varieties}
Let $T$ be a compact torus of rank $n$, i.e., non-canonically $T\cong (S^1)^n$. Let $\Lambda=\bbX^*(T)=\Hom(T, \Gm)$ be the character lattice of $T$, and $\bbX_*(T)$ its dual. Denote by $\catA_T$ the $R$-scheme that classifies maps from $\bbX^*(T)$ to $E$ as abelian groups. Hence, when $T$ is connected, $\catA_T$ is canonically isomorphic to the abelian variety $E\otimes \bbX_*(T)$. 

For any character $\xi\in \bbX^*(T)$, let $\chi_\xi: \catA_T\to E$ be the map induced by $\xi$. 
The subvariety $\ker\chi_\xi \subset \catA_T$ is a divisor, whose ideal sheaf $\calO(-\ker\chi_\xi)$  is a line bundle on $\catA_T$. Clearly, we have  $\calO(-\ker\chi_\xi)\cong \chi_\xi^*\calO(-\{0\})$.
The natural section of  $\calO(\ker\chi_\xi)$,  denoted by $\vartheta(\chi_\xi)$, is equal to $\chi_\xi^{*}\vartheta$. 

In terms of coordinates, we fix an isomorphism $T\cong (S^1)^n$, then $\catA_T\cong E^n$. Let $\xi_1,\cdots \xi_n$ be a basis of $\bbX^*(T)$. For any $\xi=\sum_{i=1}^nn_i\xi_i \in \bbX^*(T)$, the morphism $\chi_\xi:\catA_T \cong E^n\to E$ is given by  $z=(z_1,\dots,z_n)\mapsto \sum_{i=1}^nn_iz_i\in E$, and we have $\vartheta(\chi_\xi)=\vartheta(\sum _{i=1}^nn_iz_i)$.

Let $\ZZ[\bbX^*(T)]$ be the group ring of $\bbX^*(T)$, whose elements are virtual $T$-characters. It is a standard fact (see, e.g., \cite[Proposition 2.2]{Bea}) that the set map $\chi:\bbX^*(T)\to \Pic(\catA_T), \xi\mapsto \calL_{\xi}:=\calO(-\ker\chi_\xi)$ induces a homomorphism of abelian groups $\chi:\ZZ[\bbX^*(T)]\to \Pic(\catA_T)$. 
For any $\xi \in \ZZ[\bbX^*(T)]$, we also denote $\calO(-\ker\chi(\xi))$ by $\calL_\xi$.
The theta function $\vartheta(\chi_\xi)$ is a section of $\calL_{\xi}^\vee$.

More generally, let $U_r$ be the unitary group of degree $r\in \bbN$. On the symmetric product $E^{(r)}:=E^r/\fS_r$, there is a tautological line bundle $\calL_{U_r}$, whose dual has a natural section $\vartheta^{U_r}$ (see, e.g., \cite[\S~1.2]{ZZ15}). 
\Omit{\textcolor{blue}{Say one word what is it, and why it has to do with $U_r$?}} For any representation $\rho:T\to U_r$ of $T$, we have a map $\chi_\rho:\catA_T\to E^{(r)}$ induced by $\rho$. Indeed, $E^{(r)}$ is the moduli scheme of semistable $U_r$-bundles on $E$. The map $\chi_\rho$ sends a $T$-bundle $V$ to the associated $U_r$-bundle $V\times_{T} U_r$. Denote $\calL_\rho$ the line bundle $\chi_\rho^*\calL_{U_r}$ on $\catA_T$, whose dual has a natural section $\chi_\rho^*(\vartheta^{U_r} )$. This section is denoted by $\vartheta(\chi_{\rho})$.
We have the following standard lemma (see, e.g., \cite[\S~1.8]{GKV95}).

\begin{lemma}\label{lem: theta_add_tensor}
Let $\rho_i:T\to  U_{r_i}$, $i=1, 2$, be  two representations of $T$.
Let $\bbS:E^{(r_1)}\times E^{(r_2)}\to E^{(r_1+r_2)}$ be the symmetrization map.
Then the following diagram commutes.
\[\xymatrix{
E^{r_1+r_2}\ar[rr]^-{\chi_{\rho_1}\times\chi_{\rho_2}}\ar[drr]_-{\chi_{\rho_1\oplus\rho_2}}& & E^{(r_1)}\times E^{(r_2)}\ar[d]^{\bbS}\\
& & E^{(r_1+r_2)}
}\]
Moreover, $\vartheta(\chi_{\rho_1\oplus\rho_2})=\vartheta(\chi_{\rho_1})\otimes\vartheta(\chi_{\rho_2})$ as sections of $\calL_{\rho_1\oplus\rho_2}^\vee\cong \calL_{\rho_1}^\vee\otimes\calL_{\rho_2}^\vee$.
\end{lemma}

\Omit{
Let $\Lambda_n$ be a lattice of rank-$n$ with natural coordinates $(x_1,\dots,x_n)$. We have a group homomorphism $ \xi_{ij}: \Lambda_n\to \bbZ$, defined as $(x_1,\cdots,x_n) \mapsto x_i-x_j$. 
Then $\calL_{\xi_{ij}}^\vee$ is a line bundle on $E^n\cong \Lambda^n\otimes_{\bbZ}E$ with a natural section $\vartheta(x_i-x_j)$  associated to this group homomorphism $\xi_{ij}$. We will also consider 2 deformation parameters $\bbG\otimes_\bbZ\Lambda_2$, where $(t_1,t_2)$ are the coordinates of $\Lambda_2$.  Similarly, the function $\vartheta(x_i-x_j-t_1-t_2)$ is the theta-function associated to the group homomorphism $\xi_{ij, t_1, t_2}: \Lambda_n\oplus\bbZ^2\to \bbZ$ given by $(x_1,\cdots,x_n, t_1, t_2)\mapsto  x_i-x_j-t_1-t_2$, which is a section of $\calL_{ij, t_1, t_2}^\vee$ on $E^n\times_SE^2$. 
}

\begin{remark}\label{ex:multipliers}
Lemma~\ref{lem: theta_add_tensor} is true for general abelian varieties. For complex abelian varieties, it has another reformulation. 
Let $\fh$ be a complex vector space, and $\Gamma$ be a lattice in $\fh$ in the sense of \cite{Bea}. Then $\fh/\Gamma$ is an abelian variety.  An element $\gamma$ of $\Gamma$ acts linearly on $\fh\times \C$ by 
$
\gamma \cdot (z, t)=(z+\gamma, e_{\gamma}(z) t),\,\  \text{for $z\in \fh$, $t\in \C$,}
$
where $e_\gamma$ is a holomorphic invertible function on $\fh$. This formula defines a group action of $\Gamma$ on $\fh\times \C$ if and only if the functions $e_\gamma$ satisfy the cocycle condition
\[
e_{\gamma+\delta}(z)=e_{\gamma}(z+\delta) e_{\delta}(z). 
\]
A theta function for the system $(e_\gamma)_{\gamma\in \Gamma}$ is a holomorphic function $\fh\to \C$ satisfying 
\[
\vartheta(z+\gamma)=e_{\gamma}(z) \vartheta(z), \,\ \text{for all $\gamma\in \Gamma, z\in \fh$.}
\]
Conversely, for any system of multipliers $(e_\gamma)_{\gamma\in \Gamma}$ satisfying certain cocycle condition, there is an associated line bundle $L=\fh\times_{\Gamma} \C$, whose space $H^0(\fh/\Gamma, L)$ is canonically identified with the space of theta functions for $(e_\gamma)_{\gamma\in \Gamma}$ with the above multipliers. 

Let $(e_\gamma)_{\gamma\in \Gamma}$  and $(e_\gamma')_{\gamma\in \Gamma}$ be two systems of multipliers, defining line bundles $L$ and $L'$. The line bundle $L\otimes L'$ has the multiplier $(e_\gamma e_{\gamma}')_{\gamma\in \Gamma}$. 

In the present paper, we are interested in the special case when $\fh=\bbC\otimes_{\bbZ} \Lambda^{\vee}$, and 
$\Gamma=\Lambda^{\vee}\oplus\tau\Lambda^{\vee}$. Then, $\fh/\Gamma=\mathfrak{A}_T$. For $\xi=(\xi_i)_{i=1}^n\in \Lambda$, the line bundle $\calL_\xi^{\vee}= \chi_{\xi}^*(\calO(\{0\}))$ has multipliers given by
 \[
 f(z+e_i)=-f(z), \,\  f(z+\tau e_i)=-e^{-\pi i\tau \xi_i}e^{-2 \pi i z_i} f(z).
 \]
The group homomorphism in Lemma~\ref{lem: theta_add_tensor} is $\Z[\Lambda]\to \Pic(\mathfrak{A}_T)$, given by $ e^\xi\mapsto \calL_\xi^{\vee}$. 
\Omit{That is, it satisfies the following
\[
e^{\xi_1+\xi_2}\mapsto L_{\xi_1}\otimes  L_{\xi_2}, \,\ 
\text{and $e^{-\xi}\mapsto L_{\xi}^{-1}.$}
\] }
\end{remark}
More generally, let $W_P\subseteq \fS_n$ be a subgroup of $\fS_n$. Naturally $W_P$ acts on $\Z[\Lambda]$. 
\begin{lemma}\label{lem:1.4}
If $\xi\in \Z[\Lambda]^{W_P}$, then we have $\calL_\xi\cong \pi^*(\calL_\xi^P )$, for some line bundle $\calL_\xi^P$ on $\mathfrak{A}_T/W_P$, where $\pi: \mathfrak{A}_T\to \mathfrak{A}_T/W_P$ is the natural projection. 
\end{lemma}

\section{Quantization of a monoidal structure on the colored Hilbert scheme}
\label{sec:category C}
In this section and \S\ref{sec:quantum group}, $E\to S$ is an elliptic curve, where $S$ is a scheme of finite type over a field of characteristic zero. In what follows, without loss of generality, we may assume $S=\calM_{1,1}$ over the base field and $E$ to be the universal elliptic curve on $\calM_{1,1}$. The coordinate of $S$ will be denoted by $\tau$.
Unless otherwise specified, when taking product of elliptic curves we mean fibered product over $S$.

Let $Q=(I, H)$ be a quiver with the set of vertices $I$ and arrows $H$.  
For each arrow $h\in H$, we denote by $\inc(h)$ (resp. $\out(h)$) the incoming (resp. outgoing) vertex of $h$. Let $h^*$ be the corresponding  reversed arrows in the opposite quiver $Q^{\op}$. 
Let $m:H\coprod H^{\op}\to \bbZ$ be a function, which for each $h\in H$ provides two integers $m_h$ and $m_{h^*}$. We will consider specializations of $(t_1,t_2)\in E^2$ which are compatible with the function $m$ in the following sense. 
\begin{assumption} \label{Assu:WeghtsGeneral}
We consider specializations of $t_1$ and $t_2$ which are compatible with the integers $m_{h}, m_{h^*}$ for any $h\in H$, in the sense that $t_1^{m_h}t_2^{m_{h^*}}$ is a constant, i.e., does not depend on $h\in H$.
\end{assumption}

\begin{remark}\label{rmk:weights}Two examples of the integers $m_h,m_{h^*}$ satisfying Assumption~\ref{Assu:WeghtsGeneral} are the following. 
\begin{enumerate}
\item Let $t_1$ and $t_2$ are independent variables, but $m_h=m_{h^*}=1$ for any $h\in H$.
\item Specialize $t_1=t_2=\hbar/2$. For any pair of vertices $i$ and $j$ with arrows  $h_1, \dots, h_a$ from $i$ to $j$, the pairs of integers are $m_{h_p}=a+2-2p$ and $m_{h_p^*}=-a+2p$. 
\end{enumerate}
 \end{remark}

For any dimension vector $v=(v^i)_{i\in I}\in \bbN^I$ of the quiver $Q$, let
$E^{(v)}:=\left(\prod_{i\in I}\right)_S E^{(v^i)},$
where the product $\left(\prod_{i\in I}\right)_S$ is understood as fibered product over $S$. An element of $E^{(v)}$ is $\sum_{i}v^i$ points on $E$, where $v^i$ points have color $i\in I$. 
The disjoint union $\calH_{E\times I}:=\coprod_{v\in\bbN^I}E^{(v)}$ is the Hilbert scheme of $I$-colored points in $E$, relative to $S$.

Let $G_v=\prod_{i\in I}\GL_{v^i}$ be the product of general linear groups. 
Let $\Lambda_{v^i}$ be the character lattice of $\GL_{v^i}$, and $\Lambda_v:=\bigoplus_{i\in I}\Lambda_{v^i}$ of $G_v$. The natural coordinate of $\Lambda_v$ is denoted by $(z^i_r)_{i\in I,r\in[1,v^i]}$. 
Let $\fS_{v}:=\prod_{i\in I} \fS_{v^i}$ be the Weyl group of $G_v$. 
For any pair $(p,q)$ of positive integers, let $\Sh(p,q)$ be the subset of $\fS_{p+q}$ consisting of $(p,q)$-shuffles (permutations of $\{1,\cdots, p+q\}$ that preserve the relative order of $\{1,\cdots,p\}$ and $\{p+1,\cdots,p+q\}$). For any dimension vector $v\in \bbN^I$, with $v=v_1+v_2$, we denote $\Sh(v_1,v_2)\subset \fS_v$ to be the product $\prod_{i\in I}\Sh(v_1^i,v_2^i)$.

\subsection{Two special line bundles}
\label{subsec: line bundle}
Consider $\calH_{E\times I}\times_S E^2$ as a scheme over $E^2$.
Now we describe some special line bundles on $(\calH_{E\times I}\times_S E^2)\times_{E^2}(\calH_{E\times I}\times_S E^2)$.
We introduce some notations. For any dimension vector $v\in \bbN^I$, a partition of $v$ is a pair of collections $A=(A^i)_{i\in I}$ and $B=(B^i)_{i\in I}$, where  $A^{i}, B^{i} \subset [1, v^{i}]$ for any $i\in I$, satisfying the following conditions:
$A^i\cap B^i=\emptyset$, and  
 $A^{i} \cup B^{i}=[1, v^{i}]$. We use the notation $(A,B)\vdash v$ to mean $(A, B)$ is a partition of $v$. 
  
We also write $|A|=v$ if $|A^i|=v^i$ for each $i\in I$.  
For any two dimension vectors $v_1, v_2\in \bbN^I$ such that $v_1+v_2=v$, we introduce the notation  \[\bfP(v_1,v_2):=\{(A,B)\vdash v\mid |A|=v_1,|B|=v_2\}. \]
 There is a standard element $(A_o,B_o)$ in $\bfP(v_1,v_2)$ with 
$A_o^{i}:=[1, v_1^i]$, and 
$B_o^{i}:=[v_1^i+1, v_{1}^i+v_2^i]$ for any $i\in I$. This standard element will also be denoted by $([1,v_1],[v_1+1,v])$ for short. 

For any $(A,B)\in \bfP(v_1,v_2)$, consider the following function on $E^{(A)}\times_S E^{(B)}\times_S E^2$ 
\begin{equation}\label{equ:fac1}
\fac_1(z_A|z_B):=\prod_{\alpha \in I}
\prod_{s\in A^{\alpha}}
\prod_{t\in B^\alpha}
\frac{\vartheta(z^\alpha_s-z^\alpha_t+t_1+t_2)}{\vartheta(z^\alpha_t-z^\alpha_s)}. 
\end{equation}
By Lemma~\ref{lem: theta_add_tensor} and \ref{lem:1.4}, there is a line bundle $\calL_{A,B}^{\fac_1}$ on $E^{(A)}\times_S E^{(B)}\times_S E^2$, such that $\fac_1$ is a rational section of the dual line bundle $(\calL_{A,B}^{\fac_1})^\vee$. Similarly, consider the function 
\begin{small}\begin{equation}
\label{equ:fac2}
\fac_2(z_A|z_B):=\prod_{h\in H}\Big(
\prod_{s\in A^{\out(h)}}
\prod_{t\in B^{\inc(h)}}
\vartheta(z_t^{ \inc(h)}-z_s^{\out(h)}+ m_h t_1)
\prod_{s\in A^{\inc(h)}}
\prod_{t\in B^{\out(h)}}
\vartheta(z_t^{\out(h)}-z_s^{\inc(h)}+m_{h^*}t_2)
\Big).
\end{equation}\end{small}
There is a line bundle $\calL_{A,B}^{\fac_2}$ on $E^{(A)}\times_S E^{(B)}\times_S E^2$, such that $\fac_2$ is a rational section of the dual line bundle $(\calL_{A,B}^{\fac_2})^\vee$. 

Define $\calL_{A,B}:=\calL^{\fac_1}_{A,B}\otimes\calL_{A,B}^{\fac_2}$. Then, by Lemma \ref{lem: theta_add_tensor}, the dual $\calL_{A,B}^{\vee}$ has a rational section $\fac(z_A|z_B):=\fac_1\fac_2$. As we will see in \S\ref{sec:CoHA}, the function $\fac(z_A|z_B)$ naturally shows up in the study of cohomological Hall algebras.  
When $(A,B)=(A_o,B_o)$, the standard element, we also denote $\calL_{A,B}$ by $\calL_{v_1,v_2}$. 
\begin{lemma}\label{lem:line_fac_lim}
\begin{enumerate}
\item
We have the isomorphism on $E^{(v_1)}\times E^{(v_2)}$:
$\calL_{v_1, v_2}^{\fac_1} |_{t_1=t_2=0}=\calO. $
\item
If the quiver $Q$ has no arrows, then $\calL_{v_1, v_2}^{\fac_2}=\calO$.
\item If either $v_1=0$ or $v_2=0$, then $\calL_{v_1, v_2}^{\fac}=\calO_{E^{(v_1+v_2)}}$, the structure sheaf on $E^{(v_1+v_2)}$.  \end{enumerate}
\end{lemma}
\begin{proof}
The claim follows from the formulas of $\fac_1$ \eqref{equ:fac1} and $\fac_2$ \eqref{equ:fac2}. 
\end{proof}
\begin{lemma}
\label{lem:relation of L}
For dimension vectors $v_1, v_2, v_3 \in \N^I$, we have the isomorphism
\[
\calL_{v_1, v_2}^{\fac}\otimes \mathbb{S}_{12}^*(\calL_{v_1+v_2, v_3}^{\fac}) \cong 
\calL_{v_2, v_3}^{\fac}\otimes \mathbb{S}_{23}^*(\calL_{v_1, v_2+ v_3}^{\fac}), 
\]
where the maps $\mathbb{S}_{ij}$ are symmetrization maps as in the following diagram.
\begin{equation}
\label{dia: S_{123}}
\xymatrix@C=.01em @R=1.5em {
&E^{(v_1)}\times_S E^{(v_2)} \times_S E^{(v_3)}\times_S E^2 \ar[dd]^{\mathbb{S}_{123}}\ar[ld]_{\mathbb{S}_{12}}\ar[rd]^{\mathbb{S}_{23}}&\\
E^{(v_1+v_2)} \times_S E^{(v_3)}\times_S E^2\ar[rd]_{\mathbb{S}_{12, 3}}&& 
E^{(v_1)} \times_S E^{(v_2+v_3)}\times_S E^2\ar[ld]^{\mathbb{S}_{1, 23}}\\
&E^{(v_1+v_2+v_3)}\times_S E^2 &
}\end{equation}
\end{lemma}
\begin{proof}
This isomorphism can be checked at the level of the corresponding elements in $\Z[\Lambda]^{\fS_{v_1}\times \fS_{v_2}\times \fS_{v_3}}$. Those elements are obviously associative with respect to addition.  
\end{proof}
\begin{lemma}
Let $\sigma: E^{(v_1)} \times E^{(v_2)}\times E^2 \to E^{(v_2)} \times E^{(v_1)}\times E^2$ be the map that permutes the first two factors, and sends $(t_1,t_2)\in E^2$ to $(-t_1,-t_2)$. Then $\sigma^*(\calL_{v_1, v_2})\cong \calL_{v_2, v_1}$. 
\end{lemma}

\subsection{A monoidal category of sheaves}
\label{subsec:cat C}
Notations as before, let $\calH_{E\times I}$ be the $I$-colored Hilbert scheme of points on $E$. 
Let $\calC$ be the abelian category of quasi-coherent sheaves on $\calH_{E\times I}$. 
More explicitly, an object of $\calC$ is a tuple $(\mathcal{F}_{v})_{v\in \bbN^I}$, where $\mathcal{F}_{v}$ is a quasi-coherent sheaf on $E^{(v)}$. A morphism between two objects $(\mathcal{F}_{v})_{v\in \bbN^I}$ and $(\mathcal{G}_{v})_{v\in \bbN^I}$ is a collection of morphisms of sheaves $\mathcal{F}_{v}\to \mathcal{G}_{v}$ on $E^{(v)}$, for $v\in \bbN^I$. 

We define a tensor product $\otimes_{t_1,t_2}$ on $\calC$ parameterized by $(t_1,t_2)\in E^2$.
For any $\mathcal{F}, \mathcal{G}\in \calC$, we define $\mathcal{F}\otimes_{t_1,t_2} \mathcal{G}=\{(\mathcal{F}\otimes_{t_1,t_2} \mathcal{G})_v\}_{v \in \N^I}$ as
\begin{equation}\label{equ:def tensor}
(\mathcal{F}\otimes_{t_1,t_2} \mathcal{G})_{v}:=
\sum_{v_1+v_2=v}(\mathbb{S}_{v_1, v_2}\times \id)_{*}\big(p_1^*\mathcal{F}_{v_1}\otimes p_2^*\mathcal{G}_{v_2} \otimes \mathcal L_{v_1, v_2}\big),
\end{equation}
where $\mathcal L_{v_1, v_2}$ is the line bundle on $E^{(v_1)}\times_S E^{(v_2)} \times_S E^2$ described in \S \ref{subsec: line bundle}, and the maps are given in the  following diagram
\[
\xymatrix@R=1.5em @C=.1em {
&E^{(v_1)}\times_S E^{(v_2)}\times_S E^2  \ar[ld]_{p_1}\ar[rd]^{p_2}\ar[rr]^{\mathbb{S}_{v_1, v_2} \times \id}&& E^{(v_1+v_2)} \times_S E^2\\
E^{(v_1)}\times_S E^2 && E^{(v_1)}\times_S E^2&
}.
\]
Here $\mathbb{S}_{v_1, v_2}: E^{(v_1)}\times_S E^{(v_2)} \to E^{(v_1+v_2)}$ is the symmetrization map.

Consider the following object $\epsilon$ of $\calC$. When $v=0$, $\epsilon_{0}=\calO_{E^{(0)} }$ the structure sheaf of $E^{(0)}$; when $v\neq 0$, $\epsilon_{v}=0$ the zero sheaf on $E^{(v)}$. 

\begin{prop}\label{prop:C_monoidal}
\begin{enumerate}
\item
The operator $\otimes_{t_1,t_2}$ is associative. That is, we have
\[
(\mathcal{F}\otimes_{t_1,t_2} \mathcal{G})\otimes_{t_1,t_2} \mathcal{H}
\cong \mathcal{F}\otimes_{t_1,t_2} (\mathcal{G}\otimes_{t_1,t_2} \mathcal{H}).
\]
\item
The object $\epsilon$ is the identity object. That is, we have
$\epsilon\otimes_{t_1,t_2} \mathcal{G}=\mathcal{G}$, for any $\mathcal{G}\in \calC$.  
\end{enumerate}
\end{prop}
\begin{proof}
Notations as in diagram \eqref{dia: S_{123}}. 
By definition, we have
\begin{align*}
((\mathcal{F}\otimes_{t_1,t_2} \mathcal{G})\otimes_{t_1,t_2} \mathcal{H})_{v}
=&\sum_{v_1+v_2+v_3=v} (\mathbb{S}_{12, 3})_* 
\Big( (\mathbb{S}_{12})_*( \mathcal{F}_{v_1} \boxtimes \mathcal{G}_{v_2}\otimes \mathcal{L}_{v_1, v_2}) \boxtimes \mathcal{H}_{v_3}\otimes \mathcal{L}_{v_1+v_2, v_3}
\Big)\\
=&\sum_{v_1+v_2+v_3=v} (\mathbb{S}_{12, 3})_* (\mathbb{S}_{12})_*
\Big(  \mathcal{F}_{v_1} \boxtimes \mathcal{G}_{v_2}\otimes \mathcal{L}_{v_1, v_2} \boxtimes \mathcal{H}_{v_3}\otimes \mathbb{S}_{12}^* \mathcal{L}_{v_1+v_2, v_3}
\Big)\\
=&\sum_{v_1+v_2+v_3=v} (\mathbb{S}_{123})_*
\Big(  \mathcal{F}_{v_1} \boxtimes \mathcal{G}_{v_2} \boxtimes \mathcal{H}_{v_3} \otimes \mathcal{L}_{v_1, v_2} \otimes \mathbb{S}_{12}^* \mathcal{L}_{v_1+v_2, v_3}
\Big).
\end{align*}
Similar argument shows 
\begin{align*}
(\mathcal{F}\otimes_{t_1,t_2} (\mathcal{G}\otimes_{t_1,t_2} \mathcal{H}))_{v}
=\sum_{v_1+v_2+v_3=v} (\mathbb{S}_{123})_*
\Big(  \mathcal{F}_{v_1} \boxtimes \mathcal{G}_{v_2} \boxtimes \mathcal{H}_{v_3} \otimes \mathcal{L}_{v_2, v_3} \otimes \mathbb{S}_{23}^* \mathcal{L}_{v_1, v_2+v_3}\Big).
\end{align*}
The associativity follows from Lemma \ref{lem:relation of L}. 

For (2): By definition we have
\[
(\epsilon \otimes_{t_1,t_2} \mathcal{G})_{v}=
\sum_{v_1+v_2=v}\mathbb{S}_{v_1, v_2*}
\big(p_1^*\epsilon \otimes p_2^*\mathcal{G} \otimes \mathcal L_{v_1, v_2}\big)
=\id_{v*}\big(\calO_{v} \otimes_{E^{(v)}} \mathcal{G} \otimes_{E^{(v)}} \mathcal L_{0, v}\big)
=\mathcal{G}_v. 
\] 
This completes the proof. \end{proof}
Therefore, we proved the following 
\begin{theorem}\label{thm:monoidal}
The category $(\calC, \otimes_{t_1,t_2})$ is a family of monoidal categories over $E^2$.
\end{theorem}

\section{The sheafified elliptic quantum group}
\label{sec:quantum group}
\subsection{Sheafified elliptic shuffle algebra}
Recall that $H$ is the set of arrows of the quiver $Q$. Let $\overline{A}=(a_{kl})_{k, l\in I}$ be the adjacency matrix of $Q$, whose $(k, l)$-th entry is
$a_{kl}:=\#\{h\in H \mid \out(h)=k, \inc(h)=l\}$ , and let $\overline C:=I-\overline{A}$.

We consider the object $\SH:=\calO_{\calH_{E\times I}}$ of the category $\calC$. Precisely, $\SH_v=\calO_{E^{(v)}}$ for any $v\in\bbN^I$. We construct an algebra structure on $\SH$. 
The rational section $\fac(z_{A_o}|z_{B_o})=\fac_1(z_{A_o}|z_{B_o})\fac_2(z_{A_o}|z_{B_o})$ (See formulas \eqref{equ:fac1} \eqref{equ:fac2}) of $\calL_{v_1, v_2}^\vee$ induces a rational map $\calL_{v_1, v_2}\to \calO_{E^{(v_1)}\times E^{(v_2)}}$, and hence induces the following rational map \begin{equation}\label{eqn:(1)}
\SH_{v_1}\otimes_{t_1,t_2}\SH_{v_2}= \bbS_*( \calL_{v_1, v_2})\to \bbS_*\calO_{E^{(v_1)}\times_SE^{(v_2)}}.
\end{equation}

Note that $\calO_{E^{(v_1+v_2)}}$ is the subsheaf $ (\bbS_*\calO_{E^{(v_1)}\times E^{(v_2)}})^{\fS_v}$ of $\fS_v$-invariants in $\bbS_*\calO_{E^{(v_1)}\times_SE^{(v_2)}}$. We define 
\begin{equation}\label{eqn:(2)}
\bbS_*\calO_{E^{(v_1)}\times_SE^{(v_2)}}\to \calO_{E^{(v_1+v_2)}}=\SH_{v_1+v_2}, \,\  \,\ 
f\mapsto (-1)^{(v_2, \overline{C} v_1)}\sum_{\sigma\in\Sh(v_1,v_2)}\sigma(f)
\end{equation}
to be the symmetrizer modified by a sign. 
The multiplication $\star:\SH_{v_1}\otimes_{t_1,t_2}\SH_{v_2}\to \SH_{v_1+v_2}$
 is defined to be the composition of \eqref{eqn:(1)} and \eqref{eqn:(2)}. In other words, for $f_i\in \calS\calH_{v_i}, i=1, 2$, the multiplication of $f_1$ and $f_2$ has the following formula
\begin{equation}\label{shuffle formula}
f_1\star f_2=\sum_{\sigma \in\Sh(v_1,v_2)} (-1)^{(v_2, \overline{C} v_1)} \sigma \big(f_1\cdot f_2\cdot \fac(z_{[1, v_1]}|z_{[v_1+1, v_2]})\big)\in \calS\calH_{v_1+v_2}.
\end{equation}
Note that although $\fac(z_A|z_B)$ has simple poles, the multiplication $\star$ is a well-defined regular map (see, e.g., \cite[Proposition~5.29(1)]{Vish}). 

\begin{theorem}
$(\SH, \star)$ is an algebra object in $(\calC,\otimes_{t_1,t_2})$.
\end{theorem}
The theorem follows from a standard calculation. However, we will provide a topological proof of this fact in Proposition~\ref{prop:SH_ellCoh}.

\subsection{Coproduct}
In this section, we construct a coproduct $\Delta: \SH \to (\SH\otimes_{t_1,t_2} \SH)_{\loc}$ on a suitable localization of $\SH$. By definition of $\otimes_{t_1, t_2}$, it suffices to construct a rational map\[
\Delta_{v}: \SH_{v}\to \bigoplus_{\{(v_1, v_2) \mid v_1+v_2=v\}}\mathbb{S}_{*}(\SH_{v_1}\boxtimes \SH_{v_2}\otimes \calL_{v_1, v_2})
\] on each component $v\in \N^I$. 

The rational section $\frac{1}{\fac(z_{[v_1+1,v]}| z_{[1,v_1]})}$ induces a rational map $\calO_{E^{(v_1)}\times E^{(v_2)}} \to \calL_{v_1, v_2}$. This gives a rational map 
$\mathbb{S}^{*} \SH_{v}=\calO_{E^{(v_1)}\times E^{(v_2)}}  \to \SH_{v_1}\boxtimes \SH_{v_2}\otimes \calL_{v_1, v_2}= \calL_{v_1, v_2}$. The coproduct $\Delta_{v}$ is obtained by adjunction. More precisely, 
for any local section $P$ of $\SH_{v}$, the coproduct $\Delta(P)$  is given by the formula
\begin{equation}\label{eq:coprod}
\Delta (P(z))
=\sum_{\{v_1+v_2=v\}} (-1)^{(v_2, \overline{C} v_1)}  \frac{  P(z_{[1,v_1]}\otimes z_{[v_1+1,v]})}{\fac(z_{[v_1+1,v]}| z_{[1,v_1]})},
\end{equation}
where $\Delta (P(z))$ is a well-defined local rational section of $\bigoplus_{v_1+v_2=v} \SH_{v_1}\otimes_{t_1,t_2}\SH_{v_2}. $

\begin{prop}
The operator $\Delta$ is coassociative.
\end{prop}
\begin{proof}
This follows from a standard verification, similar to \cite[Proposition~2.1(1)]{YZ2}.
\end{proof}

The compatibility of $\star$ and $\Delta$ relies on a braiding $\gamma$ of the category $(\calC,\otimes_{t_1,t_2})$.
For any $v_1,v_2\in \bbN^I$, and each pair of objects $\{\calF_{v}\}_{v\in \bbN^I}$ and $\{\calG_{v}\}_{v\in \bbN^I}$, we describe a rational isomorphism 
\begin{equation}\label{eq:gamma}
\gamma_{\calF, \calG}: \calF_{v_1}\otimes_{t_1,t_2}\calG_{v_2}=\bbS_*(\calF_{v_1}\boxtimes\calG_{v_2}\otimes\calL_{v_1,v_2})
\to \calG_{v_2}\otimes_{t_1,t_2}\calF_{v_1}=\bbS_*(\calG_{v_2}\boxtimes\calF_{v_1}\otimes\calL_{v_2,v_1}). 
\end{equation}
To define $\gamma$, we only need to specify a rational section of $\calL_{v_1,v_2}^{\vee}\otimes_{E^{v_1}\times E^{(v_2)}}\calL_{v_2,v_1}$. There is a natural rational section given by 
$\widehat{\Phi}(z_{B_o}|z_{A_o}): =\frac{\fac(z_{B_o}|z_{A_o})}{\fac(z_{A_o}|z_{B_o})}$, where $A_o=[1, v_1]$, and $B_o=[v_1+1, v_{1}+v_2]$. 

\begin{theorem}
\label{thm:braiding}
The rational morphism $\gamma$ is a braiding, making $(\calC, \otimes_{t_1, t_2}, \gamma)$ a symmetric monoidal category. 
\end{theorem}
\begin{proof}
As the associative constraint is the identify map of sheaves, we only need to show the equality $\gamma_{\calF,\calG}\gamma_{\calG,\calF}=\id$, and the commutativity of the diagram
\[\xymatrix@R=1.5em @C=.1em{
\calF\otimes_{t_1, t_2} \calG\otimes_{t_1, t_2}  \calH\ar[dr]_{\gamma_{\calF,\calG}\otimes  1}
\ar[rr]^{\gamma_{\calF,\calG\otimes \calH}}&&\calG\otimes_{t_1, t_2}  \calH\otimes_{t_1, t_2}  \calF\\
&\calG\otimes_{t_1, t_2}  \calF\otimes_{t_1, t_2}  \calH\ar[ur]_{1\otimes\gamma_{\calF,\calH}}
} \]
These follow from the facts that $ \widehat{\Phi}(z_{A}|z_{B}) \widehat{\Phi}(z_{A}|z_{C})= \widehat{\Phi}(z_{A}|z_{B\sqcup C})$, and that $\widehat{\Phi}(z_{A}|z_{B})\widehat{\Phi}(z_{B}|z_{A}) =1$. 
\end{proof}

\begin{remark}\label{rmk:H_gamma}
\begin{enumerate}
\item Compared to the formula \eqref{eq:coprod}, there is an extra term $H_{A}(z_B)$ in the formula of $\Delta (P(z))$ in \cite[\S 2.1]{YZ2}. This extra factor $H_{A}(z_B)$ is absorbed in the braiding $\gamma$ of $\calC$. In \S~\ref{Cartan subalgebra}, we will construct an action of a commutative algebra on $\SH$, where $H_{A}(z_B)$ lies in the commutative algebra and has a natural meaning. 
\item The fact that $\gamma$ is only a meromorphic isomorphism has to do with the meromorphic tensor structure of the module category of $\SH$. 
\Omit{\textcolor{blue}{Add a reference here.}}
\end{enumerate}
\end{remark}
\begin{theorem}
The object $\SH=\calO_{\calH_{E\times I}}$ is a bialgebra object in $\calC_{\loc}$.
\end{theorem}

\begin{proof}
The proof goes the same way as in \cite[Theorem~2.1]{YZ2}. Nevertheless, to illustrate how the role of the $H_{A}(z_B)$-factor in \cite{YZ2} is replaced by $\gamma$, we demonstrate the proof that $\Delta$ is an algebra homomorphism. 
Note that we have the same sign $(-1)^{(v_2, \overline{C} v_1)}$ in both multiplication $\star$ \eqref{shuffle formula} and comultiplication $\Delta$ in \eqref{eq:coprod} of $\SH$. To show the claim, it is suffices to drop the sign $(-1)^{(v_2, \overline{C} v_1)}$ in both $\star$ and $\Delta$. 

We introduce the following notations for simplicity. For a pair of dimension vectors $(v_1', v_2')$, with $v_1'+v_2'=v$, and for $(A, B)\in\bfP(v_1, v_2)$, we write
\[
A_1:=A\cap [1, v_1'], \,\ A_2:=A\cap [v_1'+1, v_1'+v_2'], \,\
B_1:=B\cap [1, v_1'], \,\ B_2:=B\cap [v_1'+1, v_1'+v_2']. 
\]
Note that $\SH\otimes_{t_1, t_2} \SH$ is an algebra object in $(\calC, \otimes_{t_1, t_2}, \gamma)$. The multiplication $m_{\SH\otimes \SH}$ is given by  
\[(m_{\SH} \otimes m_{\SH})\circ (\id_{\SH}\otimes\gamma \otimes \id_{\SH}): (\SH\otimes_{t_1, t_2} \SH)\otimes_{t_1, t_2}(\SH\otimes_{t_1, t_2} \SH) \to \SH\otimes_{t_1, t_2} \SH.
\] 
We now check the identity $\Delta m=(m\otimes m)(\id\otimes\gamma\otimes\id)( \Delta\otimes\Delta )$. We have
\begin{align*}
 \Delta(P\star Q)
 =&
\sum_{(A, B)\in\bfP(v_1,v_2)} 
\Delta\big(P(z_A)\cdot Q(z_B)\cdot \fac(z_A|z_B)\big)\\
=&\sum_{(A, B)\in\bfP(v_1,v_2)} 
\sum_{ v_1'+v_2'=v}
\frac{ 
\gamma_{A_2,B_1}P(z_{A_1} \otimes z_{A_2})
Q(z_{B_1} \otimes z_{B_2})
\fac(z_{A}|z_{B})}
{ \fac( z_{[v_1'+1, v_1'+v_2']}|z_{[1, v_1']})}
\\
=&\sum_{ v_1'+v_2'=v}
\sum_{(A, B)\in\bfP(v_1,v_2)}
\frac{ 
P(z_{A_1} \otimes z_{A_2})
Q(z_{B_1} \otimes z_{B_2})
}
{ \fac(z_{A_2\sqcup B_2}|z_{A_1\sqcup B_1})}
\widehat{\Phi}(z_{B_2}| z_{A_1}) \fac(z_{A_1\sqcup A_2}|z_{B_1\sqcup B_2}).
\end{align*}
The last equality is obtained from the formula of $\gamma$.
Recall that by definition $\widehat{\Phi}(z_{B_2}| z_{A_1})=
\frac{\fac(z_{B_2}|z_{A_1})}{\fac(z_{A_1}|z_{B_2})}.$ Plugging the equality 
\[
\frac{\widehat{\Phi}(z_{B_2}| z_{A_1}) \fac(z_{A_1\sqcup A_2}|z_{B_1\sqcup B_2})}{ \fac(z_{A_2\sqcup B_2}|z_{A_1\sqcup B_1})}
= \frac{
\fac(z_{A_1}|z_{B_1})
\fac(z_{A_2}|z_{B_2})
}{ 
\fac(z_{A_2}|z_{A_1})
\fac(z_{B_2}|z_{B_1})
}\]
into the formula of $\Delta(P\star Q)$, we have
\begin{align*}
\Delta(P\star Q)
&=\sum_{ v_1'+v_2'=v}
\sum_{(A, B)\in\bfP(v_1, v_2)} 
&\frac{
P(z_{A_1} \otimes z_{A_2})  }{\fac(z_{A_2}|z_{A_1})}
\frac{
Q(z_{B_1} \otimes z_{B_2}) }
{\fac(z_{B_2}|z_{B_1})}
\fac(z_{A_1}|z_{B_1})
\fac(z_{A_2}|z_{B_2})=\Delta(P)\star\Delta(Q).
\end{align*}
This completes the proof. 
\end{proof}

For each $k\in I$, let $e_k$ be the dimension vector valued $1$ at vertex $k$ of the quiver $Q$ and zero otherwise. The \textit{spherical subsheaf}, denoted by $\SH^{\sph}$, is the subsheaf of $\SH$ generated by $\SH_{e_k}=\calO_{E^{(e_k)}}$ as $k$ varies in $I$. 
\begin{corollary}\label{cor:bialg}
The spherical subsheaf $\SH^{\sph}$ is a bialgebra object (without localization) in $(\calC, \otimes_{t_1, t_2}, \gamma)$.
\end{corollary}

In Theorem \ref{thm:pairing}, we will construct a non-degenerate bialgebra pairing on $\SH^{\sph}$, so that we can take the Drinfeld double $D(\SH^{\sph})$ of $\SH^{\sph}$. The sheafified elliptic quantum group is defined to be the Drinfeld double $D(\SH^{\sph})$.  In \S\ref{subsec:rational sec}, we verify that the meromorphic sections of $D(\SH^{\sph})$ satisfy the relations of elliptic quantum group in \cite{GTL15}. 

\section{A digression to loop Grassmannians}
\label{sec:Zastava}
In this section, we take a digression to discuss the relation between the sheafified elliptic quantum group and the global loop Grassmannian on the elliptic curve $E$. 
In particular, we show that the monoidal structure we defined is related to the factorization structure on the loop Grassmannian.

\subsection{Local spaces and local line bundles}
\label{subsec:LocalLineBundle}

Let $C$ be a smooth curve and $\mathcal{H}_{C\times I}$ be the $I$-colored Hilbert scheme of points on $C$. 
Very recently, Mirkovi\'c in \cite{Mirk} gave a new construction of loop Grassmannian $\mathcal{G}r$ over $\mathcal{H}_{C\times I}$ in the framework of \textit{local spaces}. 
A \textit{local space} introduced by Mirkovi\'c is a space over $\mathcal{H}_{C\times I}$ satisfying a factorization property for disjoint unions similar to that of  Beilinson-Drinfeld, referred to as the {\it locality structure} in {\it loc. cit.}. 
Similarly, there are notions of {\it local vector bundles, local line bundles, local projectivization of a local vector bundle}, etc. 
This locality structure has many applications, including a new construction of the semi-infinite orbits in the loop Grassmannians of reductive groups and affine Kac-Moody groups. It also has applications in geometric Langlands in higher dimensions. 
We briefly recall the relevant notions and results, and refer the readers to {\it loc. cit.} for the details.

Let $C$ be an arbitrary smooth complex curve. Similarly to \S~\ref{sec:category C}, we have the Hilbert scheme of $I$-colored points in $C$
\[\calH_{C\times I}:=\coprod_{v\in\bbN^I}C^{(v)}.\]
For a dimension vector $v=(v^i)\in \N^I$, we have  a natural projection $C^{|v|} \to C^{(v)}$. If we write points in $C^{|v|}$ using the coordinate $(z^{k}_{a})_{k\in I, a\in [1, v^k]}$ introduced in \S~\ref{sec:category C}, then have have the divisor $\Delta_{ij}$ which  is the image of the divisor 
$\bigcup_{a\in [1, v^i], b\in [1, v^j]} \{  z^{i}_a=z^{j}_{b}\}.$ It only depends on $i,j\in I$.
In the special case when $I$ is a point, the image of the divisors $\Delta_{ij}$ under the projection gives the diagonal divisor $\Delta$ of $C^{(N)}$.

We first recall the classification of local line bundles on $\mathcal{H}_{C\times I}$ given in \cite{Mirk}.
\begin{prop}[Mirkovi\'c]
\label{prop:classification local line}
A local line bundle $L$ on $\mathcal{H}_{C\times I}$ is equivalent to the data of line bundles 
$L_i$ on $C$, $i \in I$, and one quadratic form  on $\bbZ[I]$.
\end{prop}

Let $\{L_i\}_{i\in I}$ be line bundles on $C$, and $(d_{ij})$ be a quadratic form on $\bbZ[I]$. 
Note that any line bundle $L$ on $C$ defines a line bundle $L^{(n)}$ on $C^{(n)}$, $n\in \bbN$. 
The pullback to $C^n$ is $L^{\boxtimes n}$. Let $L_i^{(v^i)}$ be the line bundle on $C^{(v^i)}$. 
The corresponding local line bundle in Proposition \ref{prop:classification local line} is given by $L|_{C^{(v)}}=(\boxtimes_{i\in I} L_i^{(v^i)}) \otimes_{E^{(v)}} \calO(\sum_{i, j\in I} d_{ij}\Delta_{ij} )$.

\subsection{Loop Grassmannians from local line bundles}

The construction of loop Grassmannians from \cite{Mirk} is summarized as follows. 
For the data of a finite set $I$, a quadratic form on the lattice $\Z[I]$, Mirkovic constructs a local space, called \textit{the Zastava space}, over $\mathcal{H}_{C\times I}$. Gluing these Zastava spaces using the locality structure together gives the loop Grassmannian $\mathcal{G}r$. This procedure is called the semi-infinite construction in {\it loc. cit.}. 

Let $\mathcal{T}$ be the universal family  over $\mathcal{H}_{C\times I}$, and $\Gr( \mathcal{T})$ be the  Grassmannian parameterizing subschemes of $\mathcal{T}$. In other words, for any  $D \in  \mathcal{H}_{C\times I}$, the fibers are  $\mathcal{T}_D=D$, and  $ \Gr( \mathcal{T})_D=\{ D'\mid D'\subset D\}$ the moduli of all subschemes of $D$. 
Let $p: \Gr( \mathcal{T}) \to \mathcal{H}_{C\times I}$ be the bundle map. Let $q:  \Gr( \mathcal{T}) \to \mathcal{H}_{C\times I}$ be the map $(D', D)\mapsto D'$. 
\[
\xymatrix@R=0.5em{
& \Gr( \mathcal{T}) \ar[ld]_q\ar[rd]^p& \\
 \mathcal{H}_{C\times I} && \mathcal{H}_{C\times I}
}
\]
Let $\FM$ be the  Fourier-Mukai transform $\FM=p_* \circ q^*$, and let  $\bbP_{\loc}(V)$ be the local projectivization of a vector bundle $V$ introduced by  Mirkovi\'c in {\it loc. cit.}.
We write the line bundle $\{\mathcal{O}_{E^{(v)}} (\sum_{i, j\in I} d_{ij}\Delta_{ij})\}_{v\in \N^I}$ simply as $\calL_{ (d_{ij})_{i, j\in I}}$.
\begin{definition}
The Zastava space $\mathcal{Z}$ is isomorphic to $\bbP_{\loc}\big(\FM (\calL_{ (d_{ij})_{i, j\in I}})\big)$. 
\end{definition}

By construction, the Zastava space is a family  over $\mathcal{H}_{C\times I}$, whose generic fiber is the products $\mathbb{P}^1 \times \mathbb{P}^1 \dots \times \mathbb{P}^1$.

Assume $G$ is a reductive group, $Q=(I, H)$ the corresponding quiver of $G$, and the quadratic form from Proposition~\ref{prop:classification local line} being the adjacency matrix. 
On $\calH_{C\times I}$ there is a Beilinson-Drinfeld loop Grassmannian $\mathcal{G}r(G)$. For a point $D\in \mathcal{H}_{C\times I}$, the fiber $\mathcal{G}r_D$ of loop Grassmannian associated to $G$ parametrizes the $G$-bundles on $C$ trivialized outside $D$.  
Then, $\mathcal{Z}\subset \mathcal{G}r(G)$ is the Zastava space in the sense of \cite{FM}, which is a local space on $\calH_{E\times I}$ with special fiber the semi-infinite orbit (Mirkovi\'c-Vilonen cycle) in  $\mathcal{G}r(G)$.

The torus $T:=\Gm^I$ acts on $\mathcal{Z}$. 

\begin{theorem}[Mirkovi\'c]
\label{thm: Mirkovic}
\begin{enumerate}
\item The local space $\mathcal{Z}^{T}$ over $\mathcal{H}_{C\times I}$ is isomorphic to  $\Gr( \mathcal{T})$, the Grassmannian of the tautological scheme $\mathcal{T}$. 
\item The line bundle $\mathcal{O}_{\mathcal{G}r}(1) \mid_{\mathcal{Z}^{T} }$ is isomorphic to the pullback of the local line bundle $\calL_{ (d_{ij})_{i, j\in I}}:=\{\mathcal{O}_{E^{(v)}} (\sum_{i, j\in I} d_{ij}\Delta_{ij})\}_{v\in \N^I}$ under the map $q$. 
\end{enumerate}
\end{theorem}

The semi-infinite construction of {\it loc. cit.} then gives a loop Grassmannian $\mathcal{G}r$, which is a local space over $\mathcal{H}_{C\times I}$.
When $G$ is a reductive group, and the quadratic form is the adjacency matrix, this $\mathcal{G}r$ is the Beilinson-Drinfeld Grassmannian.
Fix a point $c$ on $E$, the fiber $\mathcal{G}r(G)_{[c]}$ at $[c]=\cup_{n\in \N} nc \in \calH_{C\times I}$ is isomorphic to the ind-scheme $G((z))/G[[z]]$, where $z$ is a local coordinate of the formal neighborhood of $c\in C$,
and the fiber 
$\mathcal{Z}_{[c]}$ is isomorphic to the Mirkovi\'c-Vilonen cycle of $G((z))/G[[z]]$.

Note that  $\mathcal{Z}(G)^{T}$ has a component $\mathcal{H}_{C\times I}$ when the subscheme $D'$ is the same as $D$. The map $q$ is identity when restricted to the component $\mathcal{H}_{C\times I}$.

\subsection{Classical limit of the monoidal structure}
\label{subsec:classical limit}
We describe the classical limit of the monoidal structure $\otimes_{t_1, t_2}$ of $\calC$, and identify it with structures occurred in the local space construction of the loop Grassmannian.

If $t_1=t_2=0$, by Lemma \ref{lem:line_fac_lim}, $\calL_{v_1,v_2}^{\fac_1}=\calO$. 
In this case, the $\fac_2$ in \eqref{equ:fac2} can be simplified as
\begin{equation}\label{eq:fac2 simplified}
\fac_2=\prod_{i, j\in I}
\prod_{s=1}^{v_1^i}
\prod_{t=1}^{v_2^j}  \vartheta(z^j_{t+v_1^j}-z^i_s)^{a_{ij}+a_{ji}}, 
\end{equation} where $a_{ij}$ is the number of arrows from $i$ to $j$ of quiver $Q$.  Let $D_{v_1, v_2}$ be  the divisor
$\sum_{i.j\in I}(a_{ij}+a_{ji}) \Delta_{ij}$ in $E^{(v_1)}\times E^{(v_2)}$. Therefore, $\calL_{v_1,v_2}^{\fac}|_{t_1=t_2=0}$ is the line bundle $\calO(D_{v_1, v_2})$ associated to the divisor $D_{v_1, v_2}$. 

We now describe the monoidal structure $\otimes_{0}$ of $\calC$. A special case is when $Q$ has no arrows. Lemma~\ref{lem:line_fac_lim} shows that in this case the tensor structure on the category $\calC$ is the same as the convolution product $*$, defined to be $\calF*\calG=\bbS_*(\calF\boxtimes\calG)$ with $\bbS:\calH_{C\times I}\times \calH_{C\times I}\to \calH_{C\times I}$ be the union map, for $\calF,\calG$ coherent sheaves on $\calH_{C\times I}$. 
That is,  
\[\{\calF\}_{v}*\{\calG\}_{v}:=\{\mathbb{S}_{v_1, v_2, *}(\calF_{v_1}\boxtimes  \calG_{v_2})\}_{v_1+v_2},\] where 
$
\mathbb{S}_{v_1,v_2}: E^{(v_1)}\times E^{(v_2)} \to E^{(v_1+v_2)}
$ is the symmetrization map. 

In general, the tensor structure $\otimes_{0}$ is closely related to the locality structure from 
\S~\ref{subsec:LocalLineBundle}. We take the curve $C$ to be the elliptic curve $E$, and the quadratic form from Proposition~\ref{prop:classification local line} to be the adjacency matrix of the quiver $Q$. 

\begin{prop}\label{prop:local}
\begin{enumerate}
\item The classical limit $\SH^{\sph}|_{t_1=t_2=0}$ of the spherical subalgebra $\SH^{\sph}$ 
is the local line bundle $\calL_{ (d_{ij})_{i, j\in I}}$ on  $\mathcal{H}_{E\times I}$. 
\item Let $L$ be the local line bundle on $\mathcal{H}_{E\times I}$ from (1), then $\calF\otimes_{t_1=t_2=0}\calG=\bbS_*((\calF\boxtimes\calG)\otimes (L\boxtimes L))\otimes L^{\vee}$.
\end{enumerate}
\end{prop}
\begin{proof}
For dimension vector $v\in \N^I$, we identify $\mathcal{SH}^{\sph}_{v}|_{t_1=t_2=0}$ with the local line bundle $\mathcal{O}_{E^{(v)}} (D^{(v)})$, where $D^{(v)}:=\sum_{i, j\in I} d_{ij}\Delta_{ij}$, and $d_{ij}=a_{ij}+a_{ji}$ is the number of arrows between vertices $i$ and $j$. 
Note that $D^{(v)}$ is the zeros of the function $\prod_{i< j\in I}
\prod_{s=1}^{v^i}
\prod_{t=1}^{v^j} \vartheta(z^j_t-z^i_s)^{d_{ij}}.$

We now prove the claim by induction on $|v|=\sum_{i} v^i$.
When $|v|=1$, we have $v=e_k$, for some $k\in I$. In this case, $\mathcal{SH}_{e_k}=\mathcal{O}_{E^{(e_k)}}$, which is the desired local bundle on $E$.  

By induction hypothesis, assume if $|v|<n$, we have $\mathcal{SH}^{\sph}_{v}|_{t_1=t_2=0}=\mathcal{O}_{E}^{\boxtimes v}(D^{(v)})$. 
For any $v\in \N^I$, we write $v=v_1+v_2$, for $v_i\in \N^I$, $i=1, 2$. We assume that $|v_i|<n$, $i=1, 2$. 
When $t_1=t_2=0$, by \eqref{eq:fac2 simplified}, the multiplication $\mathcal{SH}^{\sph}_{v_1}\otimes \mathcal{SH}^{\sph}_{v_2} \to \mathcal{SH}^{\sph}_{v}$  is
given by \eqref{shuffle formula}:
 \[f_1\star f_2= 
\sum_{\sigma \in \Sh(v_1,v_2)} (-1)^{(v_2, \overline{C} v_1)} 
\sigma(f_1\cdot f_2 \cdot \fac|_{t_1=t_2=0}),\,\ \text{ where}\,\ 
\fac|_{t_1=t_2=0}=\prod_{i, j\in I}
\prod_{s=1}^{v_1^i}
\prod_{t=1}^{v_2^j}  \vartheta(z^j_{t+v_1^j}-z^i_s)^{d_{ij}}. 
\]
Therefore, $\mathcal{SH}^{\sph}_{v}|_{t_1=t_2=0}=\bigoplus_{v_1+v_2=v}
\mathcal{O}_{E^{(v_1)}}(D^{(v_1)})\otimes 
\mathcal{O}_{E^{(v_2)}}(D^{(v_2)})  \otimes \mathcal{O}(-\fac). 
$
We have the equality
\begin{align*}
&\prod_{i< j\in I}
\prod_{s=1}^{v_1^i}
\prod_{t=1}^{v_1^j} \vartheta(z^j_t-z^i_s)^{d_{ij}}
\cdot
\prod_{i< j\in I}
\prod_{s=1}^{v_2^i}
\prod_{t=1}^{v_2^j} \vartheta(z^j_t-z^i_s)^{d_{ij}}
\cdot
\prod_{i, j\in I}
\prod_{s=1}^{v_1^i}
\prod_{t=1}^{v_2^j}  \vartheta(z^j_{t+v_1^j}-z^i_s)^{d_{ij}}\\
&=\prod_{i< j\in I}
\prod_{s=1}^{v^i}
\prod_{t=1}^{v^j} \vartheta(z^j_t-z^i_s)^{d_{ij}}, 
\end{align*}
which is the defining function of $D^{(v)}$. 
This completes the proof of (1).

Statement (2) is a direct consequence of (1). 
\end{proof}

In other words, $\SH^{\sph}$ is a 2-parameter deformation of the local line bundle $\{\mathcal{O}_{E^{(v)}} (\sum_{i, j\in I} d_{ij}\Delta_{ij})\}_{v\in \N^I}$ on $\mathcal{H}_{E\times I}$, where $d_{ij}$ is the number of arrows between vertices $i$ and $j$ of the quiver $Q$.

\subsection{Loop Grassmannians and quantum groups}
By Proposition \ref{prop:local}(1) and Theorem \ref{thm: Mirkovic}, we have
\begin{corollary}
\label{cor:Zastava}
\begin{enumerate}
\item
 $\SH^{\sph}$ is a two parameter deformation of the local line bundle $\mathcal{O}_{\mathcal{G}r}(1)\mid_{ \mathcal{H}_{C\times I} \subset \mathcal{Z}^T}$. 
 \item
We have the isomorphism $\bbP_{\loc}(\FM(\SH^{\sph}|_{t_1=t_2=0}))\cong \mathcal{Z}(G)$.
 \end{enumerate}
\end{corollary}

Recall that $(\calC, \otimes_{t_1, t_2})$ in \S\ref{subsec:cat C} is the monoidal category of coherent sheaves on $\mathcal{H}_{C\times I}$.  Let $(d_{ij})_{i, j\in I}$ be the quadratic form of $\bbZ[I]$, whose entry $d_{ij}$ is the number of arrows between $i$, $j$ of the quiver $Q$. 

Proposition \ref{prop:local}(1) identifies the sheaf $\SH^{\sph}|_{t_1=t_2=0}$ with the local line bundle 
\[
\calL_{ (d_{ij})_{i, j\in I}}=
\{\mathcal{O}_{E^{(v)}} (\sum_{i, j\in I} d_{ij}\Delta_{ij})\}_{v\in \N^I}.
\] 
By Proposition~\ref{prop:local}(2) the multiplication $\star: \calL_{ (d_{ij})_{i, j\in I}}\otimes_{0} \calL_{ (d_{ij})_{i, j\in I}}\to \calL_{ (d_{ij})_{i, j\in I}}$ is the identity map. In other words, using this observation and Corollary \ref{cor:Zastava}, we have the following. 
\begin{prop}\label{prop:local_algebra}
The algebra structure on $\SH^{\sph}|_{t_1=t_2=0}$ is equivalent to the locality structure on $\mathcal{O}_{\mathcal{G}r}(1)\mid_{ \mathcal{H}_{C\times I} \subset \mathcal{Z}^T}$. 
\end{prop}

\begin{remark}
Although the above results are stated for the case when $C=E$, the counterparts for the case when $C=\bbG_a, \bbG_m$ also hold. As there are no non-trivial line bundles on the Hilbert scheme of points in $C$ in these cases, most statements become trivial.
\end{remark}

\section{Reconstruction of the Cartan subalgebra}
\label{Cartan subalgebra}
In this section, we give another description of the symmetric monoidal category 
$(\calC, \otimes_{t_1, t_2}, \gamma)$ in Theorem~\ref{thm:braiding}. From this new description, it is clear how the Cartan subalgebra of the sheafified elliptic quantum group, denoted by $\SH^0$, acts on $\SH$. This is similar to the reconstruction of Cartan in \cite{Maj}. However, in the affine setting the structure of the monoidal category is richer, and the braiding is only a meromorphic section of a line bundle.  The meromorphic nature of the braiding is related to the existence of the Drinfeld polynomials (see \S~\ref{subsec:CartanModule}).

We introduce a category $\calC '$ which is equivalent to $(\calC, \otimes_{t_1, t_2}, \gamma)$. 
Roughly speaking, an object of  $\calC '$ have dynamical parameters. Adding of the dynamical parameters simplifies the action of the Cartan subalgebra. 

In this section, we work under Assumption~\ref{Assu:WeghtsGeneral}(2).
\subsection{Recover the Cartan-action}
Recall that for any object $\calF=(\calF_{v})_{v\in \N^I}$ in $\calC$, the braiding $\gamma$ \eqref{eq:gamma} gives a rational isomorphism $\gamma_{\calO,  \calF}: \calO_{\calH_{E\times I}}\otimes_{t_1,t_2}\calF\cong  \calF\otimes_{t_1,t_2}\calO_{\calH_{E\times I}}$. In particular,  for any $k\in I$ and $v\in \bbN^I$, on $E^{e_k}\times E^{(v)}$, $\gamma$ induces a rational isomorphism 
\begin{equation}\label{eq:gamma_k}
\gamma_{e_k,v}: \calO_{E^{e_k}}\boxtimes\calF_v\to (\calO_{E^{e_k}}\boxtimes\calF_v) \otimes \calL_{e_k,v}^\vee\otimes \calL_{v,e_k}.
\end{equation}
The map $\gamma_{e_k,v}$ is defined by the rational section 
$\widehat{\Phi}(z_{[1,v]}|z_{e_k}): =
\frac{\fac(z_{[1,v]}|z_{e_k})}{\fac(z_{e_k}|z_{[1, v]})}$ of $\calL_{e_k,v}^\vee\otimes \calL_{v,e_k}$.
Same formula appears in \cite[\S1.4]{YZ2}, which was motivated by the formulas in \cite[\S 10.1]{Nak01} and \cite[\S 4]{Va00}. 

Note that for any local section $f_k$ of $\calO_{E^{e_k}}$, $\gamma_{e_k,v}$ gives a map from 
$\calF_v$ to $\calF_v\otimes \calL_{e_k,v}^\vee\otimes \calL_{v,e_k}$. 
In the next two subsections, we show that this twist $\calL_{e_k,v}^\vee\otimes \calL_{v,e_k}$ can be removed using a structure similar to the Sklyanin algebra \S~\ref{subsec:dyn_cartan}. It will give an action of $\SH^0$ on any object of $\calC$. 
\subsection{Recover the dynamical parameters}
\label{subsec:universal elliptic curve}
Let $E$ be the universal elliptic curve over $\calM_{1, 2}$. 
Let $\mathfrak{E}$ be the elliptic curve over $\calM_{1, 1}$. Then $\mathfrak{E}$ can be realized as a quotient of $\bbC \times \mathfrak{H}$ by the action of $\bbZ^2$ and $\SL_2(\bbZ)$, where $\mathfrak{H}$ is the upper half plane. 
More explicitly, for $(z, \tau)\in \bbC \times \mathfrak{H}$, and $(n, m) \in \bbZ^2$, the action is given $(n, m)*(z, \tau):=
(z+n+\tau m, \tau)$. For $\left(\begin{smallmatrix}
a & b \\
c & d
\end{smallmatrix}\right)\in \SL_2(\Z)$, the action is given by
$\left(\begin{smallmatrix}
a & b \\
c & d
\end{smallmatrix}\right)*(z, \tau):=(\frac{z}{c\tau+d}, \frac{a\tau+b}{c\tau+d})$. Then $\mathfrak{E}$ is the 
the orbifold quotient $\bbC \times \mathfrak{H}/(\SL_2(\bbZ)\ltimes \Z^2)$, and $\mathcal{M}_{1,1}$ is the orbifold quotient $\mathfrak{H}/\SL_2(\bbZ)$.

 A closed point $(z, \lambda, \tau)$ of $E$ over $\calM_{1, 2}$ consists of a point $(z, \tau)$ on the curve $\mathfrak{E}$, together with a degree-zero line bundle on $\mathfrak{E}_\tau$ given by $\lambda \in \mathfrak{E}_\tau^{\vee}$. Here $(\lambda, \tau)$ are the coordinates of $\calM_{1, 2}$. 

Let $v=(v^i)_{i\in I}\in \N^I$ be a dimension vector. For $i\in I$, recall that 
$E^{(v^i)}= E^{|v^i|}_{\calM_{1, 2}}/\mathfrak{S}_{v^i}$, where the product $E^{|v^i|}_{\calM_{1, 2}}$ is over $\calM_{1, 2}$. Define 
\[
E^{(v)'}:=\prod_{i\in I, \calM_{1, 1}} E^{(v^i)}= \prod_{i\in I} (\mathfrak{E}^{(v^i)}\times_{\calM_{1, 1}} \calM_{1, 2}). 
\] 
The coordinates of each $\calM_{1,2}$ are denoted by $\lambda_i \in \mathfrak{E}_\tau$ for each $i\in I$.
Hence, the coordinates of $E^{(v)'}$ can be taken as $(z^i_{s}, \lambda_i, \tau)_{i\in I, s\in[1,v^i]}$.
Let $(t_1,t_2)$ be the coordinates of $\mathfrak{E}^2$. 
Let $E_{\lambda}:=\prod_{i\in I,\calM_{1,1}} \calM_{1,2}$, whose coordinates are 
$(\lambda_i, \tau)_{i\in I}$, and $E_{\lambda, t_1, t_2}:=\prod_{i\in I,\calM_{1,1}} \calM_{1,2}\times_{\calM_{1,1}}\mathfrak{E}^2$, whose coordinates are $(\lambda_i, \tau, t_1, t_2)_{i\in I}$. 
The scheme $E^{(v)'}\times_{\calM_{1,1}}\mathfrak{E}^2$ with coordinates $(z^i_{s}, \lambda_i, \tau, t_1, t_2)_{i\in I, s\in[1,v^i]}$ is a scheme over $E_{\lambda, t_1, t_2}$. 

There is a natural line bundle $\bbL$ on $E$, which has rational section 
$g_{\lambda}(z):=\frac{\vartheta(z+\lambda)}{\vartheta(z)\vartheta(\lambda)}$. We have an induced line bundle $\bbL^{(v)}$ on $E^{(v)'}$, which has rational section $\prod_{i\in I} \prod_{j=1}^{v^i} \frac{\vartheta(z^{(i)}_j+\lambda_i)}{\vartheta(z^{(i)}_j)\vartheta(\lambda_i)}$. 

For simplicity, we specialize $t_1=t_2=\frac{\hbar}{2}$. 
 Let $c_{ki}$ be the $(k, i)$-entry of the Cartan matrix of the quiver $Q$. 
 For each $k\in I$, we consider the map
\[
g_k: E^{(v)'}\times_{\calM_{1,1}} \mathfrak{E} \to E^{(v)'}\times_{\calM_{1,1}}\mathfrak{E}, \,\ \text{by} \,\ 
(z^i_{s}, \lambda_i, \tau, \hbar)_{i\in I, s\in[1,v^i]} \mapsto (z^i_{s}, c_{i,k}\hbar, \tau, \hbar)_{i\in I, s\in[1,v^i]}.
\] 
Let $\bbL_{k, \hbar}^{(v)}:=g_k^*\bbL^{(v)}$  be the line bundle on $\calH_{E\times I}'\times_{\calM_{1,1}}\mathfrak{E}$ obtained from pulling-back via $g_k^*$. Then a rational section of $\bbL_{k, \hbar}^{(v)}$ can be taken as 
$\prod_{i\in I} \prod_{j=1}^{v^i} \frac{\vartheta(z^{(i)}_j+c_{i,k}\hbar )}{\vartheta(z^{(i)}_j)\vartheta(\frac{c_{ki}}{2}\hbar)}$ or $\prod_{i\in I} \prod_{j=1}^{v^i}\frac{\vartheta(z^{(i)}_j+\frac{c_{ki}}{2}\hbar)}{\vartheta(z^{(i)}_j-\frac{c_{ki}}{2}\hbar)}$. Denote $\calH_{E\times I}':=\coprod_{v\in \N^I} E^{(v)'}$. Let $\bbL_{k, \hbar}=\{\bbL_{k, \hbar}^{(v)}\}_{v\in \N^I}$ be the line bundle on $\calH_{E\times I}'\times_{\calM_{1, 1}}\mathfrak{E}$. 
 
For each $k\in I$, define the shifting operators
 \begin{align*}
 &\rho_k: E_{\lambda,\hbar} \to E_{\lambda,\hbar}, \,\ \text{by}\,\ (\lambda_i, \hbar)_{i\in I} \mapsto (\lambda_i+c_{ki} \hbar, \hbar)_{i\in I} ,\\
 &\rho_k^{(v)}: E^{(v)'}\times_{\calM_{1,1}}\mathfrak{E} \to E^{(v)'}\times_{\calM_{1,1}}\mathfrak{E} , \,\ (z^i_s, \lambda_i,\hbar)_{i\in I,s\in [1,v^i]}\mapsto (z^i_s, \lambda_i+c_{ki}\hbar,  \hbar)_{i\in I,s\in [1,v^i]}.
 \end{align*}
Note that for $k_{1}, k_2\in I$, $\rho_{k_1}^{(v)}$ and
$\rho_{k_2}^{(v)}$ commute with each other. 
The map $\rho_k^{(v)}$ is not a morphism of $E_{\lambda,\hbar}$-schemes.
Instead, we have the following Cartesian diagram.
\begin{equation}\label{eq:shifting map}
\xymatrix@R=1em {
E^{(v)'}\times_{\calM_{1,1}}\mathfrak{E} \ar[r]^{\rho_k^{(v)}} \ar[d] & E^{(v)'}\times_{\calM_{1,1}}\mathfrak{E} \ar[d]\\
E_{\lambda,\hbar} \ar[r]^{\rho_k}  & E_{\lambda,\hbar}.
}
\end{equation}

\begin{lemma}\label{lem:trans_twist}
We have the isomorphism $(\rho_k^{(v)})^* \bbL^{(v)}\cong \bbL^{(v)}\otimes \bbL_{k, \hbar}^{(v)}$ of sheaves on $E^{(v)'}\times_{\calM_{1,1}}\mathfrak{E}$.
\end{lemma}
\begin{proof}
The claim follows from a straightforward calculation of the multipliers. Indeed, a rational section of $\bbL^{(v)}\otimes \bbL_{k, \hbar}^{(v)}$ is 
\[
\prod_{i\in I} \prod_{j=1}^{v^i} \frac{\vartheta(z^{(i)}_j+\lambda_i)}{\vartheta(z^{(i)}_j)\vartheta(\lambda_i)}
\cdot \prod_{i\in I} \prod_{j=1}^{v^i}
\frac{\vartheta(z^{(i)}_j+\frac{c_{ki}}{2}\hbar)}{\vartheta(z^{(i)}_j-\frac{c_{ki}}{2}\hbar)}.
\] 
The periodicity of $z^{(i)}_j$ is trivial along $\Lambda$-direction and  $e^{\lambda_i}\cdot e^{c_{ki}\hbar}$ along $\tau\Lambda$-direction. 
A section of $(\rho_k^{(v)})^* \bbL^{(v)}$ is
\[
\prod_{i\in I} \prod_{j=1}^{v^i} \frac{\vartheta(z^{(i)}_j+\lambda_i+c_{ki}\hbar)}{\vartheta(z^{(i)}_j)\vartheta(\lambda_i+c_{ki}\hbar)}, 
\]
whose $z^{(i)}_j$ have the same periodicity. The two line bundles have the same system of multipliers, hence they must coincide. 
\end{proof}

\subsection{Cartan-action and the dynamical parameters}\label{subsec:dyn_cartan}
We now define a category $\calC'$. 
An object of $\calC'$ consists of a pair $(\mathcal{F},  \varphi_{\mathcal{F}})$, 
where 
\begin{itemize}
\item $\mathcal{F}=(\mathcal{F}_{v})_{v\in \N^I}$ is a quasi-coherent sheaf on $\calH_{E\times I}'$. That is, each $\mathcal{F}_{v}$ is a quasi-coherent sheaf on $E^{(v)'}$, for $v\in \N^I$; 
\item $\varphi_{\calF}=(\varphi_{\calF,k})_{k\in I}$ is a collection of isomorphisms 
$\varphi_{\calF,k}:\calF\otimes_
{(\calH'_{E\times I}\times_{\calM_{1,1}}\mathfrak{E})}\bbL_{k, \hbar}\cong \rho_k^*\calF$. That is, for each $k\in I$, and $v\in \N^{I}$,  $\varphi_{\calF,k}^{(v)}: 
\mathcal{F}_v \otimes_{(E^{(v)'}\times_{\calM_{1,1}}\mathfrak{E})} \bbL_{k, \hbar}^{(v)} \cong (\rho_k^{(v)})^*\mathcal{F}_v$ is an isomorphism. 
For $k_{1}, k_2\in I$, we impose the condition that  $[\varphi_{k_1}, \varphi_{k_2}]=0$. 
\end{itemize}

A morphism from $(\mathcal{F},  \varphi_{\mathcal{F}})$ to $(\mathcal{G},  \varphi_{\mathcal{G}})$ is a morphism $f: \mathcal{F} \to \mathcal{G}$ of sheaves on $\calH_{E\times I}'$, such that, 
the following diagram commutes for any $k\in I$. 
\[
\xymatrix@R=1.5em{
\calF\otimes_{(\calH'_{E\times I}\times_{\calM_{1,1}} \mathfrak{E})}\bbL_{k, \hbar} \ar[r]^(0.7){\varphi_{\mathcal{F} }} \ar[d]_{f\otimes \id}
& (\rho_k)^*\mathcal{F} \ar[d]^{(\rho_k)^* f}\\
\calG\otimes_{(\calH'_{E\times I}\times_{\calM_{1,1}} \mathfrak{E})}\bbL_{k, \hbar} \ar[r]^(0.7){\varphi_{\mathcal{G}}}
& (\rho_k)^*\mathcal{G}
}
\]

We have the natural projection
\begin{align*}
&p_1: E^{(v)'}\to E^{(v)},  \,\ (z^i_{s}, \lambda_i, \tau)_{i\in I, s\in[1,v^i]}\mapsto (z^i_{s}, \lambda, \tau)_{i\in I, s\in[1,v^i]}.
\end{align*}
For any sheaf $\calF$ on $\mathcal{H}_{E\times I}$,  we have the sheaf
$p_1^*\calF\otimes \bbL^{(v)}=
\{p_1^*( \calF_{v})\otimes_{E^{(v)'}} \bbL^{(v)}\}_{v\in \N^I}$ on $\mathcal{H}_{E\times I}'$. 
\begin{prop}\label{prop:dyn_param}
Notations as above, the assignment $\calF \mapsto (p_1^*\calF\otimes \bbL^{(v)},  \varphi_{\calF})$ induces an equivalence of abelian categories $\calC \cong \calC'$, where the isomorphism $\varphi_{\calF}$ is given in Lemma \ref{lem:trans_twist}. 
\end{prop}
\begin{proof}
Note that $E^{(v)'}$ is a principle $E_\hbar^I$-bundle over the base $E^{(v)}$, where $E_\hbar^I$ acts by translation. It is well-known that the category of quasi-coherent sheaves on $E^{(v)}$ is equivalent to the category of quasi-coherent sheaves  on $E^{(v)'}$ which are equivariant with respect to translation. Taking into consideration of the isomorphism from Lemma~\ref{lem:trans_twist}, we are done.
\end{proof}

We now construct the action of $\SH^0$. Abstractly, $\SH^0$ is  isomorphic to $\calO_{\calH_{E\times I}}$, however, it is not an algebra object in $(\calC, \otimes_{t_1, t_2})$. 
The action of $\SH^0$ on $\calF$ will preserve each component (the root space) $\calF_{v}$, for each root $v\in \N^I$. 

Let $q: E^{e_k}\times_{\calM_{1, 2}} E^{(v)'}\times_{\calM_{1,1}} \mathfrak{E}\to E^{(v)'}\times_{\calM_{1,1}}\mathfrak{E}$ be the subtraction map sending $(z,(z_j^{(i)})_{i,j})$ to $(z_j^{(i)}-z)_{i,j}$. 
\begin{lemma} \label{lem:5.3}
We have $q^*\bbL_{k, \hbar}^{(v)}\cong \calL_{e_k,v}\calL_{v,e_k}^\vee$.
\end{lemma}
\begin{proof} 
This follows from comparing the sections on both sides.  A section of $\bbL_{k, \hbar}^{(v)}$ is $\prod_{i\in I}\prod_{j=1}^{v_i}\frac{\vartheta(z_j^{(1)}+\frac{c_{ki}}{2}\hbar)}{\vartheta(z_j^{(1)}-\frac{c_{ki}}{2}\hbar)}$, hence its pullback under $q$ is 
$\frac{\vartheta(z_j^{(1)}-z+\frac{c_{ki}}{2}\hbar)}{\vartheta(z_j^{(1)}-z-\frac{c_{ki}}{2}\hbar)}$. A rational section of the right hand side is $\frac{\fac(x_{[1,v]}|x_{e_k})}{\fac(x_{e_k}|x_{[1, v]})}=\frac{\vartheta(z_j^{(1)}-z+\frac{c_{ki}}{2}\hbar)}{\vartheta(z_j^{(1)}-z-\frac{c_{ki}}{2}\hbar)}$ (see also \cite[\S1.4]{YZ2}). 
\end{proof}
Therefore, for objects $\calF$ in $\calC'$, we have the isomorphisms
\begin{equation}\label{eq:hatPhi}
(q^*\calF_v)\otimes \calL_{e_k,v}\calL_{v,e_k}^\vee \cong
q^*(\calF_v\otimes\bbL_{k, \hbar}^{(v)})\cong 
q^*(\rho_k^{(v)*}\calF_v). \end{equation}
Define $\widehat{\Phi}_k:q^*\calF_v\to q^*(\rho_k^{(v)}\calF_v)$ to be the composition of $\gamma_{e_k,v}$ \eqref{eq:gamma_k} with the isomorphism \eqref{eq:hatPhi}.
This composition is called the action of $\calO_{E^{e_k}}$ on objects in $\calC'$. By linearity, this extends to an action of $\SH^0$. 
In particular, the action on the level rational sections of $\SH^0$ on $\SH$ in \S~\ref{subsec:rational sec} comes from this action of sheaves.

For any morphism $f:\calF\to \calG$ in $\calC$, we have the following commutative diagram
\[
\xymatrix{
\mathcal{F}_v \ar[r]^(0.4){\Phi_k} \ar[d]_{f_v}& \mathcal{F}_v\otimes \bbL_{k, \hbar}^{(v)} \ar[r]^{\varphi_{\mathcal{F} }} \ar[d]^{f_v\otimes \id}
& (\rho_k^{(v)})^*\mathcal{F}_v \ar[d]^{(\rho_k^{(v)})^* f_v}\\
\mathcal{G}_v \ar[r]^(0.4){\Phi_k}& \mathcal{G}_v\otimes \bbL_{k, \hbar}^{(v)} \ar[r]^{\varphi_{\mathcal{G}}}
& (\rho_k^{(v)})^*\mathcal{G}_v
}
\]
yielding commutativity of $f$ with the Cartan actions.

Roughly, without the dynamical parameters, sections of the Cartan given an endomorphism of an object in $\calC$, twisted by a line bundle. With the dynamical parameters, on other hand, sections of the Cartan gives an endomorphism of the same object, with a shift of the dynamical parameters. 
\begin{remark}
This structure of $\SH^0$-action is a meromorphic version of Sklyanin algebra. 
Fix an elliptic curve $\iota: E\to \mathbb{P}^2$, with corresponding line bundle $\mathcal{L}=\iota^*(\mathcal{O}_{\mathbb{P}^2}(1))$. Fix an automorphism $\sigma\in \Aut(E)$ given by translation under the group law and denote the graph of $\sigma$ by $\Gamma_{\sigma}\subset E\times E$.
Let $V:= H^0(E, \calL)$, and 
\[
\calR:=H^0( E\times E, ( \calL\boxtimes \calL)(-\Gamma_{\sigma}))
\subset H^0( E\times E, ( \calL\boxtimes \calL))=V\otimes V. 
\]
Recall that the Sklyanin algebra $\Skl(E, \mathcal{L}, \sigma)$ is by definition the algebra 
$\Skl(E, \mathcal{L}, \sigma)=T(V)/(\calR)$, where $T(V)$ denotes the tensor algebra on $V$. 

In our case, the input is $\sigma=\hbar$ and the line bundle $\calL=\calO$. Since this line bundle is not ample, we only consider  rational sections. The fact that $\calO$ has an algebra structure makes the meromorphic Sklyanin algebra commutative. 
\end{remark}

\section{The algebra of rational sections}
In this section, we take certain rational sections of the algebra object $\SH$ defined in \S\ref{Cartan subalgebra}. 
\subsection{The generating series}
\label{subsec:rational sec}
In this section, we still take $E$ to be the elliptic curve over $\calM_{1, 2}$, 
$\bbL$ the Poincar\'e line bundle on $E$, which has rational section 
$g_{\lambda}(z)=\frac{\vartheta(z+\lambda)}{\vartheta(z)\vartheta(\lambda)}$. This section is regular away from $z=0$, and has the quasi-periodicity \[
f(z+1)=f(z), \,\ f(z+\tau )=e^{2\pi i \lambda } f(z).
\] The space of all such meromorphic sections has a basis given by $\left\{ g^{(i)}_\lambda(z):=\frac{1}{i!} \frac{\partial^i}{\partial z^i}\left( \frac{\vartheta(z+\lambda)}{\vartheta(z)\vartheta(\lambda) }\right)\right\}_{i\in \N}$.

For each $v\in\bbN^I$, consider the  universal cover $\bbC^{(v)} \times \bbC^I \times \mathfrak{H}$ of $E^{(v)'}$, whose coordinates are denoted by $(z^i_t, \lambda_i, \tau)_{i\in I,t\in[1,v^i]}$.
Consider the vector space of meromorphic functions on $\bbC^{(v)} \times \bbC^I \times \mathfrak{H}$, which are regular away from the hyperplanes 
$\{z^i_{t}=n+\tau m \mid \text{for all}\,\  i\in I,  t\in [1,v^i],  \text{for some $n, m\in \bbZ$}  \}$. We have a functor $\Gamma_{\rat}$ of taking certain rational sections from the category $\calC$ to this vector space. 

Let $\mathbf{SH}_v:=\Gamma_{rat}(\SH_v \otimes \bbL^{(v)})$. 
Then $\mathbf{SH}:=\bigoplus_{v\in \N^I} \mathbf{SH}_v$ is an algebra, with multiplication defined by \eqref{shuffle formula}. 
Let $\mathbf{SH}^{\sph}\subset \mathbf{SH}$ be the subalgebra generated by $\mathbf{SH}_{e_k}$, for $k$ varies in $I$. 

Consider $\lambda=(\lambda_{k})_{k\in I}$. 
Consider the following series of $\mathbf{SH}^{\sph}$: 
\begin{equation}\label{formula of X^+}
\mathfrak{X}_{k}^+(u, \lambda):=
\sum_{i=0}^{\infty} g^{(i)}_{\lambda_k}(z_k) u^{i}=g_{\lambda_k}(u+z_k)
=\frac{\vartheta(u+z_k+\lambda_k)}{\vartheta(u+z_k)\vartheta(\lambda_k)},
\end{equation}
where the 2nd equality follows from Cauchy integral theorem. 
We have $\mathfrak{X}_{k}^+(u, \lambda)\in \mathbf{SH}^{\sph}[\![u]\!]$. 

We define a commutative algebra $\mathbf{SH}^0$, which will be the commutative subalgebra of the elliptic quantum group. 
$\mathbf{SH}^0$ is the symmetric algebra of $\bigoplus_{k\in I} \Gamma_{\rat}(\calO_{E^{e_k}})$.  
Let we take a natural basis of $\Gamma_{\rat}(\calO_{E})$ as follows. Let $\wp(z)$ be the Weierstrass $\wp$-function, and let $\wp^{(i)}(z):=\frac{1}{i!}\frac{\partial^i}{\partial z^i}\wp(z)$. Then $\{\wp^{(i)}(z)\}$ is a basis. Let \[
\Phi_{k}(u):=\frac{\vartheta(u+\hbar/2)}{\vartheta(u-\hbar/2)}\sum_{i=1}^{\infty} H_{k} \otimes \wp^{(i)}(z) u^i\]
 be the generating series of $\mathbf{SH}^0$, where $H_k\in \h$.
We have $\Phi_{k}(u)\in \mathbf{SH}^0[\![u]\!]$, for $k\in I$.  
 
We now construct an action of $\mathbf{SH}^0$ on $\mathbf{SH}$. 
Let $a_{ik}$ be the number of arrows of $Q$ from vertex $i$ to $k$, and let $c_{ik}$ be the $(i, k)$-entry of the Cartan matrix. Thus, we have $c_{ik}=-a_{ik}-a_{ki}$ if $k\neq i$, $c_{ik}=2$ if $k=i$. 
For any element $g_v \in \mathbf{SH}_{v}$, the action of $ \Phi_{k}(u)$ on $g_v$ is given by $
 \Phi_{k}(u) g_v  \Phi_{k}(u)^{-1}:=g_v \widehat{\Phi}_{k, v}(u)$ \footnote{
 The formula of $\widehat{\Phi}_{k, v}(u)$ in \cite{YZ, YZ2} is 
$\widehat{\Phi}_{k, v}(u):= \prod_{i\in I} \prod_{j=1}^{v^{i}}
\frac{\vartheta(u- z_{j}^{(i)}+(c_{ki})\frac{\hbar}{2})}{\vartheta(u- z_{j}^{(i)}-(c_{ki}) \frac{\hbar}{2})}$. If we keep the formula in  \cite{YZ, YZ2}, we only need to switch $\mathfrak{X}^+(u, \lambda)$ with  $\mathfrak{X}^-(u, \lambda)$ in the current paper. 
The formula \eqref{eq:hat Phi} is compatible with the convention in \cite{Nak01}. 
In \cite{Nak01}, $\mathfrak{X}^+(u, \lambda)$ is the lowering operator, see also \eqref{eq:x^+}. 
 },
 where 
 \begin{equation}\label{eq:hat Phi}
\widehat{\Phi}_{k, v}(u):= \prod_{i\in I} \prod_{j=1}^{v^{i}}
\frac{\vartheta(u+ z_{j}^{(i)}+(c_{ki})\frac{\hbar}{2})}{\vartheta(u+ z_{j}^{(i)}-(c_{ki}) \frac{\hbar}{2})}.
 \end{equation}

 Note that the element $\widehat{\Phi}_{k, v}(u)$ lies in 
$\mathbf{SH}_v[\![u]\!]$. 

Let $\mathfrak{X}_{k}^-(u, \lambda)$ be the corresponding series of $-\mathfrak{X}_{k}^+(-u, -\lambda)$ in the opposite algebra $\mathbf{SH}^{\coop}[\![u]\!]$. The action of $\mathbf{SH}^0$ on $\mathbf{SH}$ induces an action of $\mathbf{SH}^0$ on $\mathbf{SH}^{\coop}$. In Theorem \ref{thm:pairing}, we will construct a non-degenerate bialgebra pairing on $\SH^{\sph}$, so that we can take the Drinfeld double $D(\mathbf{SH}^{\sph}):=\mathbf{SH}^{\sph}\otimes \mathbf{SH}^0\otimes \mathbf{SH}^{\sph, \coop}$. 
The series $\mathfrak{X}_{k}^{\pm}(u, \lambda), \Phi_{k}(u)\in D(\mathbf{SH}^{\sph}) [\![u]\!]$ are the generating series of the Drinfeld double $D(\mathbf{SH}^{\sph})$.

\begin{theorem}
\label{thm:generating series relation}
The generating series  $\mathfrak{X}_{k}^{\pm}(u, \lambda)$, and $\Phi_{k}(u)$ of $D(\mathbf{SH}^{\sph})$ satisfy the following commutation  relations.
\end{theorem}
\begin{description}
\item[\namedlabel{EQ1}{EQ1}]  For each $i, j\in I$ and $h\in \mathfrak h$, we have
\[
[\Phi_i(u), \Phi_j(v)]=0 \,\ \text{and }\,\ [h, \Phi_i(u)]=0.
\]
\item[\namedlabel{EQ2}{EQ2}] For each $i\in I$ and $h\in \mathfrak h$, we have
\[
[h, \mathfrak{X}_i^{\pm}(u, \lambda)]=\pm \alpha_i(h) \mathfrak{X}_i^{\pm}(u, \lambda). 
\]
\item[\namedlabel{EQ3}{EQ3}] 
 For each $i, j\in I$,  let $a=\frac{c_{ij} }{2}\hbar$ and let $\lambda_j=(\lambda, \alpha_j)$, we have
\[
\Phi_i(u)\mathfrak{X}_j^{\pm}(v, \lambda)\Phi_i(u)^{-1}=
\frac{\vartheta(u-v\pm a)}{\vartheta(u-v\mp a)}  \mathfrak{X}_j^{\pm}(v, \lambda\pm \hbar \alpha_i)
\pm \frac{\vartheta(2a)\vartheta(u-v\mp a -\lambda_j)}{\vartheta(\lambda_j) \vartheta(u-v\mp a)} \mathfrak{X}_j^{\pm}(u\mp a, \lambda\pm \hbar \alpha_i).
\]
\item[\namedlabel{EQ4}{EQ4}] 
For each $i\neq j\in I$  and $\lambda\in \fh^*$ such that $(\lambda, \alpha_i)=(\lambda, \alpha_j)$ which denoted by $l$, let $a=\frac{c_{ij} }{2}\hbar$. Then we have
\begin{align*}
&\vartheta(2l) \vartheta(u-v\mp a) \mathfrak{X}_i^{\pm}(u, \lambda\pm \frac{\hbar}{2}\alpha_j) \mathfrak{X}_j^{\pm}(v, \lambda\mp \frac{\hbar}{2}\alpha_i) 
-\vartheta(l\pm a) \vartheta(u-v-l) \mathfrak{X}_i^{\pm}(u, \lambda\pm \frac{\hbar}{2}\alpha_j) \mathfrak{X}_j^{\pm}(u+l, \lambda\mp \frac{\hbar}{2}\alpha_i) \\
&\phantom{\vartheta(2l) \vartheta(u-v\mp a) \mathfrak{X}_i^{\pm}(u, \lambda\pm \frac{\hbar}{2}\alpha_j) \mathfrak{X}_j^{\pm}(v, \lambda\mp \frac{\hbar}{2}\alpha_i) } 
-\vartheta(l\mp a) \vartheta(u-v+l) \mathfrak{X}_i^{\pm}(v+l, \lambda\pm \frac{\hbar}{2}\alpha_j) \mathfrak{X}_j^{\pm}(v, \lambda\mp \frac{\hbar}{2}\alpha_i) \\
=&\vartheta(2l) \vartheta(u-v\pm a) \mathfrak{X}_j^{\pm}(v, \lambda\pm \frac{\hbar}{2}\alpha_i) \mathfrak{X}_i^{\pm}(u, \lambda\mp \frac{\hbar}{2}\alpha_j) 
-\vartheta(l\mp a) \vartheta(u-v-l) \mathfrak{X}_j^{\pm}(u+l, \lambda\pm \frac{\hbar}{2}\alpha_i) \mathfrak{X}_i^{\pm}(u, \lambda\mp \frac{\hbar}{2}\alpha_j) \\
&\phantom{\vartheta(2l) \vartheta(u-v\pm a) \mathfrak{X}_j^{\pm}(v, \lambda\pm \frac{\hbar}{2}\alpha_i) \mathfrak{X}_i^{\pm}(u, \lambda\mp \frac{\hbar}{2}\alpha_j) } 
-\vartheta(l\pm a) \vartheta(u-v+l) \mathfrak{X}_j^{\pm}(v, \lambda\pm \frac{\hbar}{2}\alpha_i) \mathfrak{X}_i^{\pm}(v+l, \lambda\mp \frac{\hbar}{2}\alpha_j) 
\end{align*}

\item[\namedlabel{EQ5}{EQ5}]  For each $i\neq j\in I$  and $\lambda_1, \lambda_2\in \fh^*$, we have
\[
 [\mathfrak{X}_i^{+}(u, \lambda_1) , \mathfrak{X}_j^{-}(v, \lambda_2)]=0.
\]
For $i=j\in I$, we have the following relation on a weight space $\mathbb{V}_\mu$, if $(\lambda_1+\lambda_2, \alpha_i)=\hbar(\mu, \alpha_i)$.
\[
\vartheta(\hbar) [\mathfrak{X}_i^{+}(u, \lambda_1) , 
\mathfrak{X}_i^{-}(v, \lambda_2)]=
\frac{\vartheta(u-v+\lambda_{1, i}) }{\vartheta(u-v) \vartheta(\lambda_{1, i})} \Phi_i(v)
+\frac{\vartheta(u-v-\lambda_{2, i}) }{\vartheta(u-v) \vartheta(\lambda_{2, i})} \Phi_i(u),
\]
where $\lambda_{s, i}=(\lambda_s, \alpha_i)$, for $s=1, 2$. 
\end{description}

\begin{remark}
Gautam-Toledano Laredo in \cite{GTL15} studied the category of finite dimensional representations of 
the elliptic quantum group. 
The above commutation relations was used in {\it loc. cit.} in defining a representation category of the elliptic Drinfeld currents. In a work in progress of Gautam, it is shown that the algebra defined using the elliptic $R$-matrix of Felder also satisfies that same commutation relations.
\end{remark}
Motivated by this theorem, we define $D(\mathbf{SH}^{\sph})$ to be the elliptic quantum group, and 
$D(\mathcal{SH}^{\sph})$ the sheafified elliptic quantum group. 

\subsection{Commuting relations of the Drinfeld currents}
In this section, we prove Theorem \ref{thm:generating series relation}. 
We break down the proof into Propositions~\ref{conj:ell}, \ref{prop:rel ext}, and \ref{thm: rel DSH}, which will be proven below.

\subsubsection{The relations of $\mathfrak{X}_{k}^{+}(u, \lambda)$}

\begin{prop}\label{conj:ell}
The series $\{\mathfrak{X}_{k}^+(u, \lambda)\}_{k\in I}$ satisfy the relation \eqref{EQ4} of the elliptic quantum group.  \end{prop}
\begin{proof}
Proposition \ref{conj:ell} can be proved using the product formula \eqref{shuffle formula} of the algebra $\mathbf{SH}$. For $\lambda=\{\lambda_i\}_{i\in I}$, we have $(\lambda, \alpha_i)=\lambda_i$. 
By assumption, we have $l=(\lambda, \alpha_i)=(\lambda, \alpha_j)$, and $a=\frac{c_{ij}}{2}\hbar$. 
Then, by definition, 
\[
\mathfrak{X}_i^{+}(u, \lambda+ \frac{\hbar}{2}\alpha_j)=
g_{(\lambda+ \frac{\hbar}{2}\alpha_j, \alpha_i)}(u+z_i)
= g_{l+a}(u+z_i)=\frac{\vartheta(u+z_i+l+a) }{\vartheta(u+z_i)\vartheta(l+a)}
\]
We first consider the case when $i\neq j$.
For simplicity, we write $\mathbf{a}=-c_{ij}$. Let $S$ be the set $\{\mathbf{a}, \mathbf{a}-2, \mathbf{a}-4, \dots, -\mathbf{a}+4, -\mathbf{a}+2\}$.
By the multiplication formula \eqref{shuffle formula} of $\mathbf{SH}_{e_i} \otimes \mathbf{SH}_{e_j} \to \mathbf{SH}_{e_i+e_j}$, we have
\[
\mathfrak{X}_i^+(u, \lambda) * \mathfrak{X}_j^+(v, \zeta)=
-g_{\lambda_i}(u+z_i) g_{\zeta_j}(v+z_j) 
\prod_{m\in S} \vartheta(z_{j}-z_i+m\frac{\hbar}{2}).
\] 
Therefore, the left hand side of the relation \eqref{EQ4}  becomes
\begin{align}
&\Bigg(-\vartheta(2l) \vartheta(u-v- a) 
\frac{\vartheta( u+z_i+l+a)}{ \vartheta(u+z_i ) \vartheta(l+a) }
\cdot \frac{\vartheta( v+z_j+l-a)}{ \vartheta( v+z_j ) \vartheta(l-a) } \notag\\
&+ \vartheta(u-v-l) 
\frac{\vartheta( u+z_i+l+a)}{ \vartheta(u+z_i ) }
\cdot \frac{\vartheta( u+z_j+2l-a)}{ \vartheta( u+z_j+l ) \vartheta(l-a) }
\notag\\
&
+ \vartheta(u-v+l) 
\frac{\vartheta( v+z_i+2l+a)}{ \vartheta(v+z_i+l) \vartheta(l+a) }
\cdot \frac{\vartheta( v+z_j+l-a)}{ \vartheta( v+z_j )  }\Bigg)\cdot
\prod_{m\in S} \vartheta(z_{j}-z_i+m\frac{\hbar}{2}) \label{lhs EQ3}
\end{align}

Similarly, by the multiplication formula \eqref{shuffle formula} of $\mathbf{SH}_{e_j} \otimes \mathbf{SH}_{e_i} \to \mathbf{SH}_{e_i+e_j}$, we have \begin{align*}
\mathfrak{X}_j^+(v, \zeta) * \mathfrak{X}_i^+(u, \lambda)
=&(-1)^{a+1} g_{\zeta_j}(z_j-v) g_{\lambda_i}(z_i-u) 
 \prod_{m\in S} \vartheta(z_{i}-z_j+m\frac{\hbar}{2})\\
 =&- g_{\zeta_j}(z_j-v) g_{\lambda_i}(z_i-u) 
 \prod_{m\in S} \vartheta(z_j-z_i-m\frac{\hbar}{2})
\end{align*}
Plugging the above into  \eqref{EQ4}, the right hand side of  \eqref{EQ4} becomes
\begin{align}
&\Bigg(-\vartheta(2l) \vartheta(u-v+ a) 
\frac{\vartheta( v+z_j + l+a)}{\vartheta( v+z_j )\vartheta( l+a)}
\frac{\vartheta( u+z_i + l-a)}{\vartheta( u+z_i )\vartheta( l-a)} \notag\\
&+ \vartheta(u-v-l) 
\frac{\vartheta(u+z_j+2l +a)}{\vartheta(u+z_j+l)\vartheta(l +a)}
\frac{\vartheta(u+z_i+l -a)}{\vartheta(u+z_i)}
\notag\\
&
+ \vartheta(u-v+l) 
\frac{\vartheta( v+z_j +l+a)}{ \vartheta( v+z_j )}
\frac{\vartheta( v+z_i +2l-a)}{ \vartheta( v+z_i+l)\vartheta(l-a)}\Bigg)\cdot
 \prod_{m\in S} \vartheta(z_j-z_i-m\frac{\hbar}{2}) \label{rhs EQ3}
\end{align}
In order to show \eqref{lhs EQ3} = \eqref{rhs EQ3}, we could cancel the common factor $ \prod_{m\in S\backslash \{\mathbf{a}\}} \vartheta(z_{j}-z_i+m\frac{\hbar}{2})= \prod_{m\in S\backslash \{\mathbf{a}\}} \vartheta(z_{j}-z_i-m\frac{\hbar}{2})$ and then divide both sides by $\frac{1}{\vartheta(2l)\vartheta(u-v)}$. Therefore, it suffices to show the following equality. 
 \begin{align}
&\Bigg(\frac{ \vartheta(u-v- a) }{\vartheta(u-v)}
\frac{\vartheta( u+z_i+l+a)}{ \vartheta(u+z_i ) \vartheta(l+a) }
\cdot \frac{\vartheta( v+z_j+l-a)}{ \vartheta( v+z_j ) \vartheta(l-a) }
\notag
\\
&- \frac{\vartheta(u-v-l) }{\vartheta(u-v)\vartheta(2l)}
\frac{\vartheta( u+z_i+l+a)}{ \vartheta(u+z_i ) }
\cdot \frac{\vartheta( u+z_j+2l-a)}{ \vartheta( u+z_j+l ) \vartheta(l-a) } \notag\\
&
-\frac{\vartheta(u-v+l) }{\vartheta(u-v)\vartheta(2l)}
\frac{\vartheta( v+z_i+2l+a)}{ \vartheta(v+z_i+l) \vartheta(l+a) }
\cdot \frac{\vartheta( v+z_j+l-a)}{ \vartheta( v+z_j ) }\Bigg)\cdot
\vartheta(z_{i}-z_j+a) \notag\\
=
&\Bigg( \frac{\vartheta(u-v+ a) }{\vartheta(u-v)}
\frac{\vartheta( v+z_j + l+a)}{\vartheta( v+z_j )\vartheta( l+a)}
\frac{\vartheta( u+z_i + l-a)}{\vartheta( u+z_i )\vartheta( l-a)} 
\label{eq:detailed relation of X^+}
\\
&- \frac{\vartheta(u-v-l) }{\vartheta(u-v)\vartheta(2l)}
\frac{\vartheta(u+z_j+2l +a)}{\vartheta(u+z_j+l)\vartheta(l +a)}
\frac{\vartheta(u+z_i+l -a)}{\vartheta(u+z_i)}
 \notag\\
&
- \frac{\vartheta(u-v+l) }{\vartheta(u-v)\vartheta(2l)}
\frac{\vartheta( v+z_j +l+a)}{ \vartheta( v+z_j )}
\frac{\vartheta( v+z_i +2l-a)}{ \vartheta( v+z_i+l)\vartheta(l-a)}\Bigg)\cdot
(\vartheta(z_i-z_j-a))\notag
 \end{align}
We now show the equality  \eqref{eq:detailed relation of X^+} using the following Lemmas. 
\end{proof}

The following lemma is well-known. 
\begin{lemma}
\label{gufang:Fact3}
Assume $\sum_{i=1}^4 x_i=\sum_{i=1}^4 y_i$. Then, we have the following identity of Theta function
\begin{align*}
&\frac{\vartheta(y_1-x_1)\vartheta(y_1-x_2)\vartheta(y_1-x_3)\vartheta(y_1-x_4)}
{\vartheta(y_1-y_2)\vartheta(y_1-y_3)\vartheta(y_1-y_4)}
+\frac{\vartheta(y_2-x_1)\vartheta(y_2-x_2)\vartheta(y_2-x_3)\vartheta(y_2-x_4)}
{\vartheta(y_2-y_1)\vartheta(y_2-y_3)\vartheta(y_2-y_4)}\\
&+\frac{\vartheta(y_3-x_1)\vartheta(y_3-x_2)\vartheta(y_3-x_3)\vartheta(y_3-x_4)}
{\vartheta(y_3-y_1)\vartheta(y_3-y_2)\vartheta(y_3-y_4)}
+\frac{\vartheta(y_4-x_1)\vartheta(y_4-x_2)\vartheta(y_4-x_3)\vartheta(y_4-x_4)}
{\vartheta(y_4-y_1)\vartheta(y_4-y_2)\vartheta(y_4-y_3)}=0.
\end{align*}
\end{lemma}
\begin{proof}
Define a function $f(z):=\frac{\vartheta(z-x_1)\vartheta(z-x_2)\vartheta(z-x_3)\vartheta(z-x_4)}
{\vartheta(z-y_1)\vartheta(z-y_2)\vartheta(z-y_3)\vartheta(z-y_4)}$. It is easy to check that, under the 
assumption $\sum_{i=1}^4 x_i=\sum_{i=1}^4 y_i$, $f(z)$ is an elliptic function. 
The desired identity follows from the residue theorem $\sum_{i=1}^4 \Res_{z=y_i} f(z)=0$. 
\end{proof}

Write the left hand side of \eqref{eq:detailed relation of X^+} by $I(\hbar) \vartheta(z_l-z_k+\hbar)$. 
Then, the right hand side of \eqref{eq:detailed relation of X^+} is $I(-\hbar) \vartheta(z_l-z_k-\hbar)$. 
\begin{lemma}\cite[\S 6.7]{GTL15}
We have the equality $
I(\hbar) \vartheta(z_l-z_k+\hbar)
=I(-\hbar) \vartheta(z_l-z_k-\hbar). $\end{lemma}
\begin{proof}
For the convenient of the reader, we include a proof. This identity follows essentially from Lemma \ref{gufang:Fact3}. Write $I(\hbar)$ as $I(\hbar)=T_1-T_2-T_3$, where
\begin{align*}
&T_1(\hbar)=\frac{\vartheta(u-v-\hbar)\vartheta(u-a+\lambda+\hbar)\vartheta(v-b+\lambda-\hbar)}{\vartheta(u-v)\vartheta(u-a)\vartheta(v-b)\vartheta(\lambda+\hbar)\vartheta(\lambda-\hbar)},\\
&T_2(\hbar)=
\frac{\vartheta(u-v-\lambda)
\vartheta(\lambda+\hbar)\vartheta(u-a+\lambda+\hbar)\vartheta(u-b+2\lambda-\hbar)}{\vartheta(u-v)\vartheta(2\lambda)\vartheta(u-a)\vartheta(v-b+\lambda)\vartheta(\lambda+\hbar)\vartheta(\lambda-\hbar)},\\
&T_3(\hbar)=
\frac{\vartheta(u-v+\lambda)
\vartheta(\lambda-\hbar)\vartheta(v-a+2\lambda+\hbar)\vartheta(v-b+\lambda-\hbar)}{\vartheta(u-v)\vartheta(2\lambda)\vartheta(v-a+\lambda)\vartheta(v-b)\vartheta(\lambda+\hbar)\vartheta(\lambda-\hbar)}.
\end{align*}
Applying the identity in Lemma \ref{gufang:Fact3} to the three terms, we have
\begin{align*}
&\vartheta(a-b+\hbar)T_1(\hbar)-\vartheta(a-b-\hbar) T_1(-\hbar)
=\frac{\vartheta(u-v-a+b)\vartheta(u-b+\lambda)\vartheta(v-a+\lambda)\vartheta(2\hbar)}{
\vartheta(u-v)\vartheta(u-a)\vartheta(v-b)\vartheta(\lambda+\hbar)\vartheta(\lambda-\hbar)},\\
&\vartheta(a-b+\hbar)T_2(\hbar)-\vartheta(a-b-\hbar)T_2(-\hbar)
=
\frac{\vartheta(u-v-\lambda)\vartheta(u-b+\lambda)
\vartheta(u-a+2\lambda)\vartheta(a-b+\lambda)\vartheta(2\hbar)}{
\vartheta(u-v)\vartheta(u-a)\vartheta(v-b)\vartheta(\lambda+\hbar)\vartheta(\lambda-\hbar)},\\
&\vartheta(a-b+\hbar)T_3(\hbar)-\vartheta(a-b-\hbar)T_3(-\hbar)
=-
\frac{\vartheta(u-v+\lambda)\vartheta(v-a+\lambda)
\vartheta(v-b+2\lambda)\vartheta(a-b-\lambda)\vartheta(2\hbar)}{
\vartheta(u-v)\vartheta(u-a)\vartheta(v-b)\vartheta(\lambda+\hbar)\vartheta(\lambda-\hbar)}.
\end{align*}
Hence $I(\hbar)\vartheta(a-b+\hbar)-I(-\hbar)\vartheta(a-b-\hbar)$ becomes
\begin{align*}
&\frac{\vartheta(2\hbar)}{
\vartheta(u-v)\vartheta(u-a)\vartheta(v-b)\vartheta(\lambda+\hbar)\vartheta(\lambda-\hbar)}\Bigg(
\vartheta(u-v-a+b)\vartheta(u-b+\lambda)\vartheta(v-a+\lambda)\vartheta(2\lambda)\\
&-\vartheta(u-v-\lambda)
\vartheta(u-a+2\lambda)\vartheta(v-b)\vartheta(a-b+\lambda)
+
\vartheta(u-v+\lambda)\vartheta(u-a)\vartheta(v-b+2\lambda)
\vartheta(a-b-\lambda)
\Bigg).
\end{align*}
The above expression is equal to zero by the  identity in Lemma \ref{gufang:Fact3}.
\end{proof}

\subsubsection{Cartan subalgebra}

\begin{prop}\label{prop:rel ext}
The series $\{\mathfrak{X}_{i}^+(u, \lambda)\}_{i\in I}, \Phi_{k}(u)$ satisfy the relation \eqref{EQ3} of the elliptic quantum group. 
\end{prop}
\begin{proof}
Using the fact $\mathfrak{X}_{j}^+(v, \lambda)=\frac{\vartheta(v+z_j+\lambda_j)}{\vartheta(v+z_j) \vartheta(\lambda_j)}
$, 
the left hand side of \eqref{EQ3} is 
\[
\frac{\vartheta(v+z_j+\lambda_j)}{ \vartheta(v+z_j) \vartheta(\lambda_j)}\cdot
\frac{\vartheta(u+ z_j+(c_{ij})\frac{\hbar}{2})}
{\vartheta(u+ z_j-(c_{ij}) \frac{\hbar}{2})}
=\frac{\vartheta(v+z_j+\lambda_j)}{ \vartheta(v+z_j) \vartheta(\lambda_j)}\cdot
\frac{\vartheta(u+ z_j+a)}
{\vartheta(u+ z_j-a)}, \,\ \text{where $a=\frac{c_{ij}}{2}\hbar$.} 
\]
The right hand side of \eqref{EQ3} is 
\begin{align*}
&\frac{\vartheta(u-v+ a)}{\vartheta(u-v- a)}  
\frac{\vartheta(v+z_j+\lambda_j+2a)}{ \vartheta(v+z_j) \vartheta(\lambda_j+2a)}
+ \frac{\vartheta(2a)\vartheta(u-v- a -\lambda_j)}{\vartheta(\lambda_j) \vartheta(u-v- a)}
\frac{\vartheta(u+z_j+\lambda_j+a   )}{ \vartheta(u+z_j-a )\vartheta(\lambda_j+2a  )}
\end{align*}
Therefore, it suffices to show
\begin{align*}
&\frac{\vartheta(v+z_j+\lambda_j)}{ \vartheta(v+z_j) \vartheta(\lambda_j)}\cdot
\frac{\vartheta(u+ z_j+a)}
{\vartheta(u+ z_j-a)}
-\frac{\vartheta(u-v+ a)}{\vartheta(u-v- a)}  
\frac{\vartheta(v+z_j+\lambda_j+2a)}{ \vartheta(v+z_j) \vartheta(\lambda_j+2a)}\\
&- \frac{\vartheta(2a)\vartheta(u-v- a -\lambda_j)}{\vartheta(\lambda_j) \vartheta(u-v- a)}
\frac{\vartheta(u+z_j+\lambda_j+a   )}{ \vartheta(u+z_j-a )\vartheta(\lambda_j+2a  )}
=0
\end{align*}
Multiply both sides by $\vartheta(\lambda_j+2a)\vartheta(\lambda_j)$, we need to show the vanishing of the following.
\begin{align}
\frac{ \vartheta(v+z_j+\lambda_j)\vartheta(u+ z_j+a) \vartheta(\lambda_j+2a)}
{ \vartheta(v+z_j) \vartheta(u+ z_j-a) }
&-
\frac{\vartheta(u-v+ a)\vartheta(v+z_j+\lambda_j+2a) \vartheta(\lambda_j) }{\vartheta(u-v- a) \vartheta(v+z_j) }\notag\\
&-\frac{\vartheta(2a)\vartheta(u-v- a -\lambda_j) \vartheta(u+z_j+\lambda_j+a   )}
{\vartheta(u-v- a)\vartheta(u+z_j-a ) }=0.\label{eq:identity in the proof}
\end{align}
We now make the following change of variables. Let
\begin{align*}
&x_1-x_2=v+z_j,  &&x_2-x_3=-(-u+v+a),&& x_1-x_3=u+z_j-a\\
&x_1-y_1=v+z_j+\lambda, &&x_1-y_2=u+z_j+a, && x_1-y_3=-\lambda-2a
\end{align*}
\Omit{
Then
\begin{align*}
&x_2-y_1=\lambda, &&x_2-y_2=u-v+a, 
&&x_2-y_3=-\lambda-z_j-v-2a\\
&x_3-y_1=-u+v+\lambda+a, 
&&x_3-y_2=2a, &&x_3-y_3=-\lambda-z_j-u-a.
\end{align*}}
Clearly, we have $\sum_{i=1}^3 x_i=\sum_{i=1}^3 y_i$. 
Plugging the change of variable into \eqref{eq:identity in the proof}, the desired equality becomes
\[
\frac{\vartheta(x_1-y_1)\vartheta(x_1-y_2)\vartheta(x_1-y_3)}{\vartheta(x_1-x_2)\vartheta(x_1-x_3)}
+
\frac{\vartheta(x_2-y_1)\vartheta(x_2-y_2)\vartheta(x_2-y_3)}{\vartheta(x_2-x_1)\vartheta(x_2-x_3)}
+\frac{\vartheta(x_3-y_1)\vartheta(x_3-y_2)\vartheta(x_3-y_3)}{\vartheta(x_3-x_1)\vartheta(x_3-x_2)}=0.
\]
This follows from a similar identity as in Lemma \ref{gufang:Fact3}. 
This completes the proof. 
\end{proof}

\subsubsection{The Drinfeld double}
Recall that for a bialgebra  $(A, \star, \Delta)$  with multiplication $\star$, and coproduct $\Delta$, 
the \textit{Drinfeld double} of the bialgebra $A$ is $DA = A \otimes A^{\coop}$ as a vector space endowed with a  suitable multiplication. Here $A^{\coop}$ is $A$ as an algebra but with the opposite comultiplication. If $\dim(A)$ is infinite, in order to define $DA$ as a bialgebra, we need a non-degenerate bialgebra pairing
$( \cdot , \cdot) : A \otimes A \to R$,  i.e., an $R$-bilinear non-degenerate pairing such that \[
(a \star b, c) = (a \otimes b, \Delta(c))\text{ and }(c,a \star b) = ( \Delta(c), a \otimes b)\text{ for all $a, b, c \in A $}. 
\]

For a bialgebra $(A, \star, \Delta)$ together with a non-degenerate bialgebra pairing $( \cdot , \cdot)$, the bialgebra structure of $DA=A^-\otimes A^+$, still denoted by $(\star,\Delta)$, is uniquely determined by the following two properties (see, e.g., \cite[\S~2.4]{X}).
\begin{enumerate}
\item $A^-= A^{\coop} \otimes 1$ and $A^+ = 1 \otimes A$ are both sub-bialgebras of $DA$. 
\item For any $a,b\in A$, write $a^- = a \otimes 1\in A^-$ and $b^+ = 1 \otimes b\in A^+$. Then
\begin{equation}\label{eq: commut rel}
\sum a^-_1 \star b^+_2 \cdot (a_2, b_1) = 
\sum b^+_1 \star a^-_2 \cdot (b_2, a_1), \,\ \text{for all $a, b \in A$,}
\end{equation}
where we follow Sweedler's notation and write  $\Delta(a^-)=\sum a^{-}_1 \otimes a^{-}_2$, 
$\Delta(b^+)=\sum b^+_1 \otimes b^+_2$. 
\end{enumerate}
We now take $A$ to be $\mathbf{SH}^{\sph}\otimes \mathbf{SH}^{0}$. Recall that the  reduced Drinfeld double \[
D(\mathbf{SH}^{\sph})=\mathbf{SH}^{\sph}\otimes \mathbf{SH}^{0}\otimes \mathbf{SH}^{\sph,\coop}\] is the Drinfeld double 
$(\mathbf{SH}^{\sph}\otimes \mathbf{SH}^{0})\otimes (\mathbf{SH}^{0}\otimes \mathbf{SH}^{\sph})^{\coop}$ with the following additional relation imposed: For $\Phi_k^+(u)\in \mathbf{SH}^{0}$, and 
$\Phi_k^-(-u)\in \mathbf{SH}^{0, \coop}$, we have $\Phi_k^+(u)=\Phi_k^-(-u), \,\ \text{for any $k\in I$}.$ 

The series $\mathfrak{X}_{k}^-(u, \lambda) \in \mathbf{SH}^{\coop}[\![u]\!]$, by definiton,  corresponds to the series $-\mathfrak{X}_{k}^+(-u, -\lambda)$ in $\mathbf{SH}[\![u]\!]$.

By Proposition \ref{conj:ell} and \ref{prop:rel ext}, we have the following. 
\begin{prop}
The series $\{\mathfrak{X}_{i}^-(u, \lambda)\}_{i\in I}, \Phi_{k}(u)$ satisfy the relations \eqref{EQ3} and  \eqref{EQ4}.\end{prop}

We now prove the cross relation between $\mathfrak{X}_{i}^+(u, \lambda)\}_{i\in I}$, and $\mathfrak{X}_{i}^-(u, \lambda)\}_{i\in I}$. 
\begin{prop}\label{thm: rel DSH}
The series $\{\mathfrak{X}_{i}^\pm(u, \lambda)\}_{i\in I}, \Phi_{k}(u)$ satisfy the relations \eqref{EQ5} of the elliptic quantum group.
\end{prop}

\begin{proof}
Let $k, l\in I$, such that $k\neq l$. 
The relation $[\mathfrak{X}_k^{+}(u, \lambda_1) , \mathfrak{X}_l^{-}(v, \lambda_2)]=0$
follows from the relation \eqref{eq: commut rel} with $a=\mathfrak{X}_l^{-}(v, \lambda_2)$, $b=\mathfrak{X}_k^{+}(u, \lambda_1)$, and $(\mathfrak{X}_k^{+}(u, \lambda_1), \mathfrak{X}_l^{-}(v, \lambda_2))=0$. 

We consider the case when $k=l$. In $\mathbf{SH}^{\ext}=\mathbf{SH}\otimes \mathbf{SH}^0$, we have $\Delta(\mathfrak{X}_k^{+}(u, \lambda_1))=\Phi_k(x^{(k)})\otimes \mathfrak{X}_k^{+}(u, \lambda_1)+\mathfrak{X}_k^{+}(u, \lambda_1)\otimes 1$ by \eqref{eq:coprod}. 
We use the relation \eqref{eq: commut rel} with $a=\mathfrak{X}_k^{-}(u, \lambda)
$, $b=\mathfrak{X}_k^{+}(v, \lambda)$ and the fact
\[
(\mathfrak{X}_k^{+}(u, \lambda), \Phi_l(x))=0, \,\ (1, \mathfrak{X}_k^{+}(u, \lambda))=0. 
\] It gives the following relation in $\mathbf{SH}^{\ext, \coop}\otimes \mathbf{SH}^{\ext}$: 
\begin{align*}
&\Phi_k^-(x^{(k)})\star1 (\mathfrak{X}_k^{-}(u, \lambda), \mathfrak{X}_k^{+}(v, \lambda))+ \mathfrak{X}_k^{-}(u, \lambda)
\star \mathfrak{X}_k^{+}(v, \lambda) (1, \Phi_k(x^{(k)}))\\
=& \mathfrak{X}_k^{+}(v, \lambda)\star \mathfrak{X}_k^{-}(u, \lambda)
 (1, \Phi_k(x^{(k)}))+\Phi_k^+(x^{(k)})\star1 (\mathfrak{X}_k^{+}(v, \lambda), \mathfrak{X}_k^{-}(u, \lambda) ). 
\end{align*}
Using the equality $
\mathfrak{X}_{k}^+(u, \lambda)=
\frac{\vartheta(u+z^{(k)}+\lambda_k)}{\vartheta(u+z^{(k)})\vartheta(\lambda_k)}$, we have
\begin{align*}
&[ \mathfrak{X}_k^{-}(u,  \lambda_2), 
\mathfrak{X}_k^{+}(v,  \lambda_1)]\\
=&\Phi_k^+(z^{(k)}) (\mathfrak{X}_k^{+}(v,  \lambda_1), \mathfrak{X}_k^{-}(u,  \lambda_2))
-\Phi_k^-(z^{(k)}) (\mathfrak{X}_k^{-}(u,  \lambda_2), \mathfrak{X}_k^{+}(v,  \lambda_1))\\
=&\Phi_k^+(z^{(k)}) \left(\frac{\vartheta(v+z^{(k)}+\lambda_{1, k})}{\vartheta(v+z^{(k)})\vartheta(\lambda_{1, k})}, \,\
\frac{\vartheta(-u+z^{(k)}-\lambda_{2, k})}{\vartheta(-u+z^{(k)})\vartheta(\lambda_{2, k})} \right)
-\Phi_k^-(z^{(k)})
 \left(\frac{\vartheta(-u+z^{(k)}-\lambda_{2, k})}{\vartheta(-u+z^{(k)})\vartheta(\lambda_{2, k})} , \,\ 
 \frac{\vartheta(v+z^{(k)}+\lambda_{1, k})}{\vartheta(v+z^{(k)})\vartheta(\lambda_{1, k})}
\right)\\
=&\sum_{z\in E}\Res_{z^{(k)}=z} 
\Phi_k^+(z^{(k)})\frac{\vartheta(v+z^{(k)}+\lambda_{1, k})}{\vartheta(v+z^{(k)})\vartheta(\lambda_{1, k})} \cdot \frac{\vartheta(-u-z^{(k)}-\lambda_{2, k})}{\vartheta(-u-z^{(k)})\vartheta(\lambda_{2, k})}d z^{(k)}\\
&-\sum_{z\in E}\Res_{z^{(k)}=z}  \Phi_k^-(z^{(k)}) \frac{\vartheta(-u+z^{(k)}-\lambda_{2, k})}{\vartheta(-u+z^{(k)})\vartheta(\lambda_{2, k})} \cdot
 \frac{\vartheta(v-z^{(k)}+\lambda_{1, k})}{\vartheta(v-z^{(k)})\vartheta(\lambda_{1, k})} dz^{(k)}
\end{align*}
\begin{align*}
 =&\Phi_k^+(-u)\frac{\vartheta(v-u+\lambda_{1, k})}{\vartheta(v-u)\vartheta(\lambda_{1, k})} -\Phi_k^+(-v)\frac{\vartheta(u-v+\lambda_{2, k})}{\vartheta(u-v)\vartheta(\lambda_{2, k})}
 +\Phi_k^-(u)\frac{\vartheta(u-v-\lambda_{1, k})}{\vartheta(u-v)\vartheta(\lambda_{1, k})} 
 -\Phi_k^-(v)\frac{\vartheta(-u+v-\lambda_{2, k})}{\vartheta(-u+v)\vartheta(\lambda_{2, k})} \\
 =&
\Phi_k^+(-u)\frac{\vartheta(u-v-\lambda_{1, k})}{\vartheta(u-v)\vartheta(\lambda_{1, k})} -\Phi_k^+(-v)\frac{\vartheta(u-v+\lambda_{2, k})}{\vartheta(u-v)\vartheta(\lambda_{2, k})}
 +\Phi_k^-(u)\frac{\vartheta(u-v-\lambda_{1, k})}{\vartheta(u-v)\vartheta(\lambda_{1, k})} 
 -\Phi_k^-(v)\frac{\vartheta(u-v+\lambda_{2, k})}{\vartheta(u-v)\vartheta(\lambda_{2, k})} 
 \\
 =& \Phi_k(u) \frac{\vartheta(u-v-\lambda_{1, k})}{\vartheta(u-v)\vartheta(\lambda_{1, k})} 
 -\Phi_k(v)\frac{\vartheta(u-v+\lambda_{2, k})}{\vartheta(u-v)\vartheta(\lambda_{2, k})} , 
\end{align*}
where $\Phi_k(u):=\vartheta(\hbar) \Phi_k^+(u)+\vartheta(\hbar) \Phi_k^-(-u)$. 

This completes the proof. 
\end{proof}

\subsection{The Manin pair}
As before in \S\ref{subsec:universal elliptic curve}, let $E$ be the elliptic curve over $\calM_{1, 2}$. The algebra $\mathbf{SH}^{\sph}\rtimes\mathbf{SH}^0$ is the quantization of the Manin pair coming from the elliptic curve in the sense of Drinfeld \cite{Dr}. 

More precisely, let $\fg=\fn^+\oplus \fh\oplus \fn^-$ be the Kac-Moody Lie algebra associated to quiver $Q$.
Let $L_\lambda$ be the set of rational sections of  $\bbL_\lambda$ regular away from the origin. 
Then, the Drinfeld double $\mathbf{SH}^{\sph}\otimes\mathbf{SH}^0 \otimes \mathbf{SH}^{\sph, \coop}$
quantizes the sum \[(\fn^+\otimes L_\lambda)\oplus (\fh\otimes L_0)\oplus(\fn^-\otimes L_{-\lambda}).\]

\section{Preliminaries on equivariant elliptic cohomology}
In this section, we briefly review the equivariant elliptic cohomology theory. The details can be found in \cite{AO, Lur, GKV95, ZZ15}.
\subsection{Elliptic cohomology valued in a line bundle}\label{subsec:ellCohDyn}
Let $G$ be an algebraic reductive group with a maximal torus $T$. Let $X$ be a smooth quasi-projective variety endowed with an action of $G$. 
For an elliptic curve $E$ over the base scheme $S$. Recall that $\catA_G$ is the moduli scheme of semistable principal $G$-bundles over $E$. It is canonically isomorphic to the variety $E\otimes \bbX(T)/W$. The $G$-equivariant elliptic cohomology $\Ell_{G}(X)$ of $X$ is a quasi-coherent sheaf of $\mathcal{O}_{\catA_G}$-modules, satisfying certain axioms. In particular, for smooth morphisms, we have pullback in elliptic cohomology, and for proper morphisms, we have pushforward in elliptic cohomology theory. 

Let $\det:G\to Z$ be the universal character of $G$.  In other words, denote by $\bbX(G)=\Hom(G,\Gm)$ the character lattice of $G$. We have a canonical isomorphism of abelian algebraic groups $Z\cong \Hom(\bbX(G),\Gm)$, and the map  $\det$ is isomorphic to the tautological map $G\to \Hom(\bbX(G),\Gm)$. 
For example, 
when $G=\prod_{i=1}^n\GL_{v^i}$ for a sequence of positive integers $(v^1,\dots,v^n)$, we have $Z=\Gm^n$. The universal character $\det: \prod_{i=1}^n \GL_{v^i}\to Z$ sends $(g_1,\dots, g_n)$ to $(\det(g_1),\dots,\det(g_n))$.

The map $\det$ induces the following map of varieties $\catA_{\det}:\catA_G\to \catA_Z$. 
Consider the following maps $p_1: \catA_G\times\catA_Z^\vee \to  \catA_G$, and 
$\det\times \id:   \catA_G\times\catA_Z^\vee  \to  \catA_Z\times\catA_Z^\vee$. 
Let $\bbL$ be the universal line bundle on $\catA_Z\times_S\catA_Z^\vee$. 
For any finite $G$-variety $X$, the elliptic cohomology of $X$ valued in the line bundle $\bbL$ is defined to be
\[\Ell_G^\lambda(X):=p_1^*\Ell_G(X)\otimes((\catA_{\det}\times\id)^*\bbL), \] 
as a sheaf on $\catA_G\times\catA_Z^\vee$. Identifying $\catA_Z$ with $\catA_Z^\vee$, we will also consider $\Ell_G^\lambda(X)$ as a sheaf on $\catA_G\times\catA_Z$.

We have the following examples. 
\begin{example}\label{ex:lambda}
In the case when $Q$ is the quiver of $\fs\fl_2$, which has one vertex and no arrows, the line bundle $\bbL$ has a simple description. 
The open subset $\calM_{1,2}\times_{\calM_{1,1}}\calM_{1,2}$ of $\overline{\calM_{1,2}}$ can be considered as an elliptic curve on $\calM_{1,2}$ via the second projection.
The zero-section of this elliptic curve is given by $\calM_{1,2}\to \calM_{1,2}\times_{\calM_{1,1}}\calM_{1,2}$, $z\mapsto (0,z)$.
On $\calM_{1,2}\times_{\calM_{1,1}}\calM_{1,2}$, there is a universal Poincare line bundle $\bbL$, endowed with a natural section $s$ that vanishes on the zero-section of each $\calM_{1,2}$-factor. 
Let $\tau$ be the coordinate of $\calM_{1,1}$, and let $z, \lambda$ be the  fiber-wise coordinates of the two copies of $\calM_{1,2}$ respectively. Then, $s$ can be expressed as the function  
\begin{equation}\label{eq:theta parameter}
g_\lambda(z;\tau):=\frac{\vartheta(z;\tau)\vartheta(\lambda;\tau)}{\vartheta(z+\lambda;\tau)}, 
\end{equation} where $\vartheta(z; \tau)$ is the Jacobi-theta function defined in 
Example \ref{example:theta}. The elliptic cohomology associated to this local parameter was studied in \cite{Tot,BL,Chen10}. It is well-known that this elliptic cohomology theory has coefficients in the ring of Jacobi forms. 
\end{example}
\begin{example}
Let $E\to \calM_{1, 2}$ be the universal elliptic curve, and $G=\prod_{i\in I} \GL_{v^i}$. Then,  we have
$\catA_Z^{\vee}=E_{\lambda}=\prod_{i\in I,\calM_{1,1}} \calM_{1,2}$, and $ \catA_G\times \catA_Z^{\vee}= E^{(v)'}=\prod_{i\in I} (\mathfrak{E}^{(v^i)}\times_{\calM_{1, 1}} \calM_{1, 2})$  from \S \ref{subsec:universal elliptic curve}.
\end{example}
\subsection{Thom bundle}
Let $X$ be a $G$-variety, and $V\to X$ an equivariant $G$-vector bundle. Let $\Th(V)$ be the Thom space of $V$. Recall that $\Th(V)$ is the quotient of disk bundle $\mathbb{P}(V\oplus \C)$ by the sphere bundle $\mathbb{P}(V)$. Denote $ \Theta_{G}(V):=\Ell_{G}(\Th(V))$  the equivariant elliptic cohomology of $\Th(V)$. Such an assignment $V\mapsto  \Theta_{G}(V)$ can be extended to the Grothendieck group of $X$. That is, we have an abelian group homomorphism (see \cite{ZZ15})
\[
 \Theta_G: K^0(X)\to \Pic(\Ell_G(X) ), 
\]
where $\Pic(\Ell_G(X) )$ is the abelian group of  rank 1 locally free modules over $\Ell_G(X) $. 

For any morphism $g: X' \to X$ between smooth $G$-varieties, denote by $ \Theta(g) \in \Pic(\Ell_G(X) )$ the image of the virtual vector bundle $g^*TX-TX'$ on $X'$ under $ \Theta_G$. Note that when $g$ is a closed embedding, $g^*TX-TX'$ is the normal bundle of the embedding. 
\subsection{Refined pullback in elliptic cohomology}
\label{sec:pullback in ell}
For a closed embedding $i_{Y}: Y\inj X$, let $U_Y$  be an open neighborhood of $Y$ in  $X$ which contracts to $Y$. The Thom space, denoted by $\Th_{Y}(X)$,  of the embedding $Y\inj X$ is by definition $\Th_{Y}(X)=X/(X\backslash U_Y)$. When $Y$ is singular, we define the equivariant elliptic cohomology of $Y$ to be $\Ell_{G}(Y):=\Ell_{G}(\Th_Y(X))$. When $Y$ is a smooth variety, and the  embedding $i_{Y}: Y\inj X$ is a regular embedding, we then have $\Ell_G(\Th_Y(X))= \Theta(i_Y). $

Let $Y$ and $Y'$ are two singular varieties. Assume there are two closed embeddings $i_Y: Y\inj X$ and $i_{Y'}: Y'\inj X'$, such that $X$ and  $X'$ are smooth. The reductive group $G$ acts on those varieties, and the actions are compatible with the embeddings $i_Y$, $i_{Y'}$. 
Assume furthermore, we have the following Cartesian diagram of $G$-varieties. 
\begin{equation}\label{pullback diag}
\xymatrix@R=1.5em{
Y' \ar@{^{(}->}[r]^{i_{Y'}} \ar[d]_{f} & X' \ar[d]^{g}\\
Y\ar@{^{(}->}[r]^{i_Y} & X
}
\end{equation}
In diagram \eqref{pullback diag}, $g: X' \to X$ is a smooth morphism between two smooth varieties. The pullback $g^*: \Ell_{G}(X) \to \Ell_{G}(X')$ is well-defined. We define the pullback $f^{\sharp}: \Ell_G(Y')\to \Ell_{G}(Y)$ in diagram \eqref{pullback diag} as the pullback on Thom spaces 
\[
f^{\sharp}:=\Th(g)^*: \Ell_{G}(\Th_{Y}(X)) \to \Ell_{G}(\Th_{Y'}(X')), 
\]
where $\Th(g): \Th_{Y'} X'\to \Th_{Y} X$ is the map induced from $g$, using the following diagram
\[
\xymatrix@R=1.3em{
X'\backslash U_{Y'}  \ar[r]\ar[d]^g&X' \ar[d]^{g} \ar[r]^{\pi'} &\Th_{Y'} X' \ar[d]^{\Th(g)}\\
X\backslash U_{Y}  \ar[r]&X \ar[r]^{\pi} &\Th_{Y} X
}\] 
By definition, we have the following commutative diagram with all maps given by pullback. 
\begin{equation}\label{Thom pullback}
\xymatrix@R=1.3em{
\Ell_G(\Th_Y X) \ar[r]^{\pi^*}  \ar[d]_{f^{\sharp}} & \Ell_G(X) \ar[d]^{g^*}\\
\Ell_G(\Th_{Y'} X') \ar[r]^{\pi'^*} & \Ell_G(X')
}\end{equation}

\subsection{Pushforward in elliptic cohomology}
\label{sec:pushforward in ell}
The pushforward in equivariant elliptic cohomology theory is delicate. It involves a twist coming from the Thom bundle. 

For any morphism $g: X' \to X$ between smooth $G$-varieties, there is a well defined pushforward (see, e.g., \cite{GKV95}) $g_*:  \Theta(g)\to \Ell_{G}(X)$. 

For the singular case, the setup is the same as in diagram \eqref{pullback diag}. 
Recall in \cite[2.5.2]{GKV95}, $ \Theta(f):= \Theta(i_{Y'})\otimes \mathcal{H}om(f^* \Theta(i_Y), i_{Y'}^* \Theta(g))$. 
We have the well-defined pushforward map $f_*:  \Theta(f) \to \Ell_{G}(Y)$. 
Equivalently, let $g_*:  \Theta(g)\otimes\Ell_G(X') \to \Ell_{G}(X)$ be the pushforward. 
By restriction on the open subsets, we have $g_*:  \Theta(g)\otimes_{\Ell_G(X')}\Ell_G(X'\setminus U_{Y'})  \to \Ell_{G}(X\setminus U_Y)$. It induces the following map
\[
 \Theta(g)\otimes_{\Ell_G(X')}\Ell_G(\Th_{Y'}(X')) \to \Ell_{G}(\Th_{Y}(X)), 
\] which is equivalent to $f_*$. 

\subsection{Characteristic classes}\label{subsec:char_class}
For a $G$-equivariant virtual vector bundle $V$ on $X$, the Euler class, denoted by $e(V)$, is a natural rational section of $\Theta(V)^\vee$. When $V$ is a vector bundle and $i:X\to V$ is the zero-section, then the map $i_*:\Theta(i)=\Theta(V)\to \Ell_G(V)\cong\Ell_G(X)$ is given by $e(V)$.

For any $G$-equivariant virtual vector bundle $V$ on $X$, we define the total Chern polynomial of $V$, denoted by $\lambda_z(V)$, to be the function on $E\times\catA_G$ which is $e(k^{-1}V)$ where $k$ is the natural representation of $\Gm$ and $E$ is the $\catA_{\Gm}$ with coordinate on it denoted by $z$.

For any rational section $s$ of a line bundle  $\bbL$ on $E$, there is a notion of $s$-Chern classes introduced in \cite{GKV95} and recalled in detail in \cite[\S~3.4]{ZZ15}.

\begin{example}\label{ex:rat_sec_sl2}
Let $\Ell_G^\lambda$ be as in \S~\ref{subsec:ellCohDyn} and the function $g_\lambda(z)$ be as in Example~\ref{ex:lambda}. Let $\mathcal{O}(1)\to \PP^1$ be the tautological line bundle on $\PP^1$. 
Let $\pi: E\to E/ \mathfrak{S}_2$ be the projection. 
Then, we have an isomorphism $\Ell_{\SL_2}^{\lambda}(\PP^1)=\pi_*(\mathbb{L})$. 
The first $g_\lambda$-Chern class of $\mathcal{O}(1)$ is 
\[ c_1^{\lambda}(\mathcal{O}(1) )=g_\lambda(z)=\frac{ \vartheta(z+\lambda)}{ \vartheta(z) \vartheta(\lambda)},\] which is a section of $\mathbb{L}$. 
\end{example}

\subsection{Lagrangian correspondence formalism}
\label{subsec: Lag corr}
We recall the Lagrangian correspondence formalism following the exposition in \cite{SV2}.

Let $X$ be a smooth quasi-projective variety endowed with an action of a reductive algebraic group $G$. The cotangent bundle $T^*X$ is a symplectic variety, with an induced Hamiltonian $G$-action. Let $\mu:T^*X\to (\Lie G)^*$ be the moment map. Denote $T^*_GX:=\mu^{-1}(0)\subseteq T^*X$. 

Let $P\subset G$ be a parabolic subgroup and $L\subset P$ be a Levi subgroup. Let $Y$ be a smooth quasi-projective variety equipped an action of $L$, and $X'$ smooth quasi-projective with a $G$-action. Let $\calV\subseteq Y\times X'$ be a smooth subvariety. We have the two projection$\xymatrix{
Y& \calV \ar[l]_{\pr_1}\ar[r]^{\pr_2} &X'}.$
Assume the first projection $\pr_1$ is a vector bundle, and the second projection $\pr_2$ is a closed embedding. 

Let $X:=G\times_PY$ be the twisted product. Set $W:=G\times_P\calV$ and consider the following maps
\[\xymatrix{
X & W \ar[l]_{f_1}\ar[r]^{f_2} &X'}, \,\
f_1: [(g, v)] \mapsto [(g, \pr_1(v))], \,\ f_2: [(g, v)] \mapsto g\pr_2(v), \] 
where $[(g, v)]$ is the pair $(g, v)\mod P$. Note that the natural map $T^*X\to G\times_PT^*Y$ is a vector bundle. 

Let  $Z:=T^*_W(X\times X')$ be the conormal bundle of $W$ in $X\times X'$. Let $Z_G\subseteq T^*_GX\times T^*_GX'$ be the intersection $Z\cap (T^*_GX\times T^*_GX')$. 
Then we have the following diagram.
\begin{equation*}
\xymatrix@R=1.5em{
G\times_PT^*_LY\ar[r]^(0.6){\cong}\ar@{^{(}->}[d]&T^*_GX\ar@{^{(}->}[d]&Z_G\ar[l]_{\overline\phi}\ar[r]^{\overline\psi}\ar@{^{(}->}[d]&T^*_GX'\ar@{^{(}->}[d]\\
G\times_PT^*Y\ar@{^{(}->}[r]^{\iota}&T^*X&Z\ar[l]_{\phi}\ar[r]^{\psi}&T^*X'
}\end{equation*}
where $\phi:Z\to T^*X$ and $\psi:Z\to T^*X'$ are respectively the first and second projections of $T^*X\times T^*X'$ restricted to $Z$. The map $\iota: G\times_PT^*Y\inj T^*X$ is the zero-section of the vector bundle $T^*X\to G\times_PT^*Y$.
The following lemma is proved in \cite{SV2}.
\begin{lemma} \cite{SV2}
\label{lem:fiber_lag}
\begin{enumerate}
\item
There is an isomorphism $G\times_PT^*_LY\cong T^*_GX$  such that the above diagram commutes. 

\item
The morphism $\psi:Z\to T^*X'$ is proper. We have $\psi^{-1}(T^*_GX')=Z_G$ and $\phi^{-1}(T^*_GX)=Z_G$.
\end{enumerate}
\end{lemma}

\section{The elliptic cohomological Hall algebra}
\label{sec:CoHA}
In this section, we define the elliptic cohomological Hall algebra (CoHA) as an algebra object in the category $\calC$. As in \S\ref{sec:quantum group}, the elliptic CoHA is the positive part of the sheafified elliptic quantum group. 
\subsection{The geometric meaning of the elliptic shuffle algebra}
In this section, we give a geometric interpretation of the elliptic shuffle algebra $\SH$ in \S\ref{sec:quantum group}. 

We first fix the notations. As before, let $Q=(I, H)$ be a quiver, and let  $v, v_1, v_2 \in\bbN^I$ be dimension vectors such that $v=v_1+v_2$. Fix an $I$-tuple of vector spaces $V=\{V^i\}_{i\in I}$ of $Q$ such that $\dim(V)=v$. The representation space of $Q$ with dimension vector $v$ is denoted by $\Rep(Q,v)$. That is,\[
\Rep(Q, v):=\bigoplus_{h\in H}\Hom_{k}(V^{\out(h)},V^{\inc(h)}).
\]
Fix an $I$-tuple of subvector spaces $V_1\subset V$ such that $\dim(V_1)=v_1$. 

In the Lagrangian correspondence formalism in Section \S\ref{subsec: Lag corr}, we take $Y$ to be $\Rep(Q,v_1)\times \Rep(Q,v_2)$, $X'$ to be $\Rep(Q,v_1+v_2)$, and 
\begin{equation}\label{eq:the space V}
\calV:=\{x\in \Rep({Q},v)\mid x(V_1)\subset V_1\}\subset\Rep(Q,v).
\end{equation}

We write $G:=G_{v}=\prod_{i\in I}\GL_{v^i}$, and $P \subset G_v$, the parabolic subgroup preserving the subspace $V_1$. Let $L:=G_{v_1}\times G_{v_2}$ be the Levi subgroup of $P$. The group $G$ acts on the cotangent space $T^*\Rep(Q, v)$ via conjugation

As in  \S\ref{subsec: Lag corr}, we have the following Lagrangian correspondence of $G\times T$-varieties:
\begin{equation}
\label{equ:Lagr corresp}
\xymatrix{
G\times_PT^*Y\ar@{^{(}->}[r]^{\iota}&T^*X&Z \ar[l]_(0.3){\phi} \ar[r]^(0.3){\psi} & T^*\Rep(Q, v),
}
\end{equation}
where the torus $T=\Gm^2$ acts on the symplectic varieties with the first $\Gm$ factor scaling the base, and the second one scaling the fiber of the symplectic varieties. We assume the weights of this $T$-action satisfies Assumption~\ref{Assu:WeghtsGeneral}.

Following the Lagrangian correspondence \eqref{equ:Lagr corresp}, we have the following three maps:
\begin{align*}
& \iota_*:  \Theta(\iota) \to \Ell_G(T^*X), 
& \phi^*: \Ell_G(T^*X) \to \Ell_G(Z), 
&& \psi_*:  \Theta(\psi) \to \Ell_G(T^* X').
\end{align*}

By \cite[Lemma 5.1]{YZ}, we have $Z=G\times_P \calV$, and $T^*X=G\times_P \widetilde{Y}$, for some variety $\widetilde{Y}$, with

$\widetilde{Y}=
\{
(c, x, x^*)\mid c \in \fp_v , x \in \Rep( Q, v_1)\times \Rep( Q, v_2),
x^*\in \Rep( Q^{\op}, v_1)\times \Rep( Q^{\op}, v_2), [x, x^*]=\pr(c)
\}$. Let $\mathbb{S}_{v_1, v_2}: E^{(v_1)}\times E^{(v_2)}\to E^{(v_1+v_2)}$ be the symmetrization map. 
We then have the following equalities.
\begin{align*}
&\Ell_{P}(T^*Y)=\calO_{E^{(v_1)}\times E^{(v_2)} }, &&
\Ell_{G}(G\times_P T^*Y)=\mathbb{S}_{v_1, v_2, *}\calO_{E^{(v_1)}\times E^{(v_2)} }\\
&\Ell_{G}(T^*X)=\mathbb{S}_{v_1, v_2, *}(\Ell_{P} \widetilde{Y}), &&
\Ell_{G}(Z)=\mathbb{S}_{v_1, v_2, *}( \Ell_{P}(\calV)). 
\\
& \Theta(\iota)=\mathbb{S}_{v_1, v_2, *}(\calL^{\iota}),&& 
 \Theta(\psi)=\mathbb{S}_{v_1, v_2, *}(\calL^{\psi}). 
\end{align*}
We will abuse of notation and still denote the embedding $Y\inj \widetilde{Y}$ by $\iota$. Then, 
$\calL^{\iota}$ is the line bundle $\Ell_{P}(N(\iota))$ on $E^{(v_1)}\times E^{(v_2)}$, where $N(\iota)$ is the normal bundle of $\iota: Y\inj \widetilde{Y}$, and $N(\iota)\cong T^*G/P$ . \Omit{
The map $\psi$ is the composition $p' \circ \psi'$, see \cite[The proof of Proposition 3.4]{YZ}, we have $\calL^{\psi}=\calL^{\psi'}\otimes \calL^{p'}$.  }

The symmetrization map $\mathbb{S}_{v_1, v_2}$ is a finite map. Hence, there is a one to one correspondence between the category of coherent sheaves on $E^{(v_1)}\times E^{(v_2)}$ and the category of $\mathbb{S}_{v_1, v_2, *}(\calO_{E^{(v_1)}\times E^{(v_2)} })$--modules, via $\calF\mapsto \mathbb{S}_{v_1, v_2, *}\calF$. Using this correspondence, the composition $\phi^* \circ \iota_*$ is the same as $\mathbb{S}_{v_1, v_2, *}$ applied to the composition  of sheaves on $E^{(v_1)}\times E^{(v_2)}$: 
\begin{equation}
\label{equ:comp2}
\Ell_{P}(T^*Y)\otimes \calL^{\iota} \to \Ell_{P}(\widetilde{Y}) \to \Ell_{P}(\calV).
\end{equation}
Tensoring the composition \eqref{equ:comp2} with $\calL^{\psi}$, applying the functor $\mathbb{S}_{v_1, v_2}$, and then composing with $\psi_{*}$, we have 
\begin{equation}
\label{equ:comp}
\psi_{*}\circ \phi^*: \mathbb{S}_*(\Ell_{P}(T^*Y)\otimes \calL^{\iota}\otimes \calL^{\psi})  \to \mathbb{S}_*(\Ell_{P}(\calV)\otimes\calL^{\psi})= \Theta(\psi)\to  \Ell_G(T^* X').
\end{equation}

Recall that $\Ell_{G_v\times  \G_m^2}(T^*\Rep(Q, v))$ is a sheaf over $E^{(v)}\times E^2$, for any $v\in \N^I$. Using the tensor structure $\otimes_{t_1,t_2}$ defined in \eqref{equ:def tensor}, and the fact $\calL^{\iota}\otimes  \calL^{\psi}=\calL^{\fac}_{v_1, v_2}$, 
the composition \eqref{equ:comp} gives the map
\begin{equation}\label{eqn:15}
\Ell_{G_{v_1}}(T^*\Rep(Q, v_1))\otimes_{t_1,t_2} \Ell_{G_{v_2}}(T^*\Rep(Q, v_2))\to \Ell_{G_{v}}(T^*\Rep(Q, v)),  \,\ v=v_1+v_2.
\end{equation}
The associativity of \eqref{eqn:15} follows from the same argument as in the proof of \cite[Proposition 3.3]{YZ}. 
\begin{prop}\label{prop:SH_ellCoh}
Under the isomorphism $\Ell_{G_{v}}(T^*\Rep(Q, v))\cong \calO_{E^{(v)}}$, the map \eqref{eqn:15} is the same as $\star$ in \eqref{shuffle formula}.
\end{prop}
\begin{proof}
The claim follows from the proof of \cite[Proposition 3.4]{YZ}. Roughly speaking, the pushforward of $\iota: G\times_PT^*Y\inj T^*X$ 
is giving by multiplication of 
\[
e_{v_1,v_2}^{\iota}=\prod_{i\in I}\prod_{s=1}^{v_1^i}
\prod_{t=1}^{v_2^i}\vartheta(z^i_s-z^i_{t+v_1^i}+t_1+t_2), 
\]
where $\{z^i_s-z^i_{t+v_1^i}+t_1+t_2\}$ are the Chern roots of the equivariant normal bundle of $\iota$, which is isomorphic to $T^*G/P\cong \bigoplus_{i\in I}(\calR(v_1^i)\otimes \calR(v_2^i)^*)$ over the Grassmannian $G/P$. 

Let $EG$ be the total space of the universal bundle over $BG$. Then, as explained in the proof of \cite[Proposition 3.4]{YZ}, the map 
$\psi: EG\times_{G}Z \to  EG\times_{G} T^*\Rep(Q, v)$ is the composition of 
$\psi_1: EG\times_{G}Z \to p^*T^*\Rep(Q, v) $, and $p': p^*T^*\Rep(Q, v)\to  EG\times_{G} T^*\Rep(Q, v)$, where $p: BL\to BG$ is the map between classifying spaces. 

The pushforward $\psi_{1*}$ is multiplication by $\fac_2$ \eqref{equ:fac2}. 
The map $p'$ is a Grassmannian bundle, and consequently $p'_*$ is given by a shuffle formula. 
Putting all the above together, the map \eqref{eqn:15} is given by exactly the same formula as  \eqref{shuffle formula}.\end{proof}

\subsection{The elliptic CoHA}\label{subsec:ellCoHA}
Notations as before, let $\fg_{v}$ be the Lie algebra of $G_v=\prod_{i\in I}\GL_{v^i}$. Let 
\[\mu_{v}: T^*\Rep(Q,v)\to \fg^*_{v}, \,\ (x, x^{*})\mapsto [x, x^{*}]\] be the moment map. 
Note that the closed subvariety $\mu_v^{-1}(0)\subset T^*\Rep(Q,v)$ could be singular in general.

As before, we consider the Lagrangian correspondence formalism in Section \S\ref{subsec: Lag corr}, with the following specializations: Take $Y$ to be $\Rep(Q,v_1)\times \Rep(Q,v_2)$, $X'$ to be $\Rep(Q,v)$ and  $\calV$ is the same as \eqref{eq:the space V}. Recall in Section \S\ref{subsec: Lag corr}, we have the following correspondence of $G\times T$-varieties:
\begin{equation}\label{diag:lagrangian}
\xymatrix@R=1.5em{
G\times_P\left(\mu_{v_1}^{-1}(0)\times\mu_{v_2}^{-1}(0)\right)\ar@{=}[r]\ar@{^{(}->}[d]&T^*_GX\ar@{^{(}->}[d]&Z_G\ar[l]_{\overline\phi}\ar[r]^{\overline\psi}\ar@{^{(}->}[d]&\mu_{v}^{-1}(0)\ar@{^{(}->}[d]\\
G\times_PT^*Y\ar@{^{(}->}[r]^{\iota}&T^*X &Z \ar[l]_(0.3){\phi} \ar[r]^(0.3){\psi} & T^*\Rep(Q, v).
}
\end{equation}
\begin{definition}
The elliptic CoHA of any semisimple Lie algebra $\g_Q$ associated to $Q$ is 
\[
\calP_{\Ell}(Q):=\bigoplus_{v\in \N^I} \Ell_{G_{v}}(\mu_{v}^{-1}(0))
=\bigoplus_{v\in \N^I} \Ell_{G_{v}}\Big(\Th_{\mu_{v}^{-1}(0)}(T^*\Rep(Q, v))\Big)
.\]  
More precisely, it consists: 
\begin{itemize}
\item  A system of coherent sheaves of
$\mathcal{O}_{E^{(v)}}$-modules $\calP_{\Ell, v}:=\Ell_{G_{v}}(\mu_{v}^{-1}(0))$, for $v\in \N^I$. 
\item For any $v_1, v_2\in \N^I$, with $v=v_1+v_2$, an $\mathcal{O}_{E^{(v)}}$-module homomorphism
\begin{equation}
\label{equ:mul CoHA}
(\mathbb{S}_{v_1, v_2})_{*}\big( \Ell_{G_{v_1}}(\mu_{v_1}^{-1}(0)) \boxtimes 
\Ell_{G_{v_2}}(\mu_{v_2}^{-1}(0)) \otimes \mathcal{L}_{v_1, v_2}\big)
\to \Ell_{G_{v_1+v_2}}(\mu_{v_1+v_2}^{-1}(0)), \end{equation}
where $\mathcal{L}_{v_1, v_2}$ is the same line bundle as in \S\ref{sec:category C}, and its the dual $\mathcal{L}_{v_1, v_2}^{\vee}$ has rational section \[\fac(z_{[1, v_1]}|z_{[v_1+1, v_1+v_2]})=\fac_1\fac_2.\] 
\end{itemize}
\end{definition}
We describe the morphism  \eqref{equ:mul CoHA}  of sheaves  on $E^{(v_1+v_2)}$ using the pullback $\overline{\phi}^{\sharp}$ and pushforward $\overline\psi_{*}$ in diagram \eqref{diag:lagrangian}. 

Note the square with maps $\overline\phi, \phi$ is a Cartesian square.
As described in \S\ref{sec:pullback in ell}, \S\ref{sec:pushforward in ell}, we have the following well-defined morphisms
\begin{align*}
&\overline{\phi}^\sharp: \Ell_{G}(T^*_{G}X) \to \Ell_{G}(Z_G), &&\overline{\psi}_*:\Ell_{G}(Z_G)\otimes  \Theta(\psi)
\to \Ell_{G} (\mu_{v}^{-1}(0)). 
\end{align*}
We have the following isomorphisms.
\[
\Ell_{G}(\Th_{T^*_{G}X} (T^*X)) \cong 
\Ell_{G}(\Th_{T^*_{G}X} (G\times_P T^*Y))\otimes  \Theta(\iota)
=\Ell_{G_{v_1}\times G_{v_2}} (\mu_{v_1}^{-1}(0) \times \mu_{v_2}^{-1}(0)) \otimes  \Theta(\iota)
\]
This gives the composition
\[
\overline{\psi}_* \circ \overline{\phi}^\sharp: 
\Ell_{G_{v_1}\times G_{v_2}} (\mu_{v_1}^{-1}(0) \times \mu_{v_2}^{-1}(0)) \otimes  \Theta(\iota)\otimes  \Theta(\psi)
\to \Ell_{G} (\mu_{v}^{-1}(0)), 
\]
which is the morphism \eqref{equ:mul CoHA}, since
$ \Theta(\iota)\otimes  \Theta(\psi)=\mathbb{S}_{v_1, v_2, *}(\calL^{\iota} \otimes \calL^{\psi} )
=\mathbb{S}_{v_1, v_2, *}(\calL^{\fac}_{v_1, v_2})$. 
\begin{theorem}\label{conj:ellCoHA}
The object $(\calP_{\Ell}, \star)$ is an algebra object in $\calC$. 
\end{theorem}
\begin{proof}
It follows from the same argument as in \cite[Proposition 4.1]{YZ}. 
\end{proof}
\subsection{Relation with the elliptic shuffle algebra}
Recall we have the functor $\Gamma_{\rat}$ of taking certain rational section, and $\mathbf{SH}=\Gamma_{\rat}(\SH)$ is an algebra. Similarly, denote $\mathbf{P}:=\Gamma_{\rat}(\calP)$. Then, $\mathbf{P}$ is an associative algebra by Theorem \ref{conj:ellCoHA}. 
\begin{theorem}
\label{thm:relation with shuffle}
There is an algebra homomorphism $
\mathbf{P}\to \mathbf{SH}$ induced from the embedding
$i_v: \mu^{-1}_{v}(0) \inj T^*\Rep(Q, v)$.
\end{theorem}
\begin{proof}
For any $v\in \N^I$, the pushforward $(i_v)_*:  \Ell_{G}(\mu_{v}^{-1}(0))\to\Ell_G( T^*\Rep(Q, v))$ is by definition the pullback $p^*$, where 
\[
p: T^*\Rep({Q}, v) \to \Th_{ \mu_{v}^{-1}(0)} (T^*\Rep({Q}, v) ).
\] is the natural projection. 
The desired map is giving by the pushforward $(i_v)_*$. 
\end{proof}
\begin{remark}

The algebra homomorphism in Theorem \ref{thm:relation with shuffle} becomes an isomorphism after suitable localization, see \cite[Remark 4.4]{YZ}. Indeed,  $\mu_v^{-1}(0)$ has only one $T$-fixed point. 
It follows from the Thomason localization theorem \cite[Theorem~6.2]{GKM}, which in the present setting can be found in \cite{Kr}, that $i_{v*}$ is an isomorphism when passing to a  localization. This localization of $R[\![t_1,t_2,\lambda^i_s]\!]^{\fS_v}$ is at the prime ideal  generated by all the symmetric functions in $\lambda^i_s$ without constant terms. For the power series ring, this is the same as passing to  $R(\!(t_1,t_2)\!)$. 
\end{remark}
The \textit{spherical subalgebra}, denoted by $\mathbf{P}^{\sph}$, is the subalgebra of $\mathbf{P}$ generated by $\mathbf{P}_{e_k}=\Gamma_{\rat}( \calP_{e_k})$ as $k$ varies in $I$. 
Up to certain torsion elements, we have the isomorphism $\mathbf{P}^{\sph}\cong \mathbf{SH}^{\sph}$. 
As a result of Theorem \ref{thm:generating series relation}, the Drinfeld double $D(\mathbf{P}^{\sph})$ satisfies the commuting relations of elliptic quantum group of Felder and Gautam-Toledano Laredo. Motivated by this result, we define the sheafified elliptic quantum group to be the Drinfeld double of $\calP^{\sph}$, and its algebra of rational sections $D(\mathbf{P}^{\sph})$ is the elliptic quantum group.

\section{Representations of the sheafified elliptic quantum group}
From now on we study the representations of the sheafified elliptic quantum groups. 
In this section, we introduce a new category $\calD$, which is a module category over the monoidal category $\calC$ in \S\ref{sec:category C}. We show that the elliptic cohomology of Nakajima quiver varieties are objects in $\calD$, and  furthermore they are modules of the elliptic CoHA. 
\subsection{The category $\calD$}\label{subsec:catD}
 We first recall the general definition of a module category over a monoidal category. 
\begin{definition}
A \textit{module category} over a monoidal category $\calC$ is a category $\calM$ together with an exact bifunctor $\otimes: \calC\otimes \calM\to \calM$ and functorial associativity $m_{X, Y, M}$ and unit isomorphism $l_{M}$: 
\[
m_{X, Y, M}: (X\otimes Y)\otimes M\to  X\otimes (Y\otimes M), 
\,\
l_{M}: 1\otimes M \to M
\]
for any $X, Y \in \calC$, $M\in \calM$, such that the following two diagrams commute. 
\[
\xymatrix@C=0.8em @R=1em
{
&((X\otimes Y)\otimes Z)\otimes M\ar[ld]_{a_{X,Y, Z}\otimes \id}\ar[rd]^
{m_{X\otimes Y, Z, M}}&\\
(X\otimes (Y\otimes Z))\otimes M\ar[d]_{m_{X, Y\otimes Z, M}}&& (X\otimes Y)\otimes (Z\otimes M)\ar[d]^{m_{X, Y, Z\otimes M}}\\
X\otimes ((Y\otimes Z)\otimes M)\ar[rr]^{\id \otimes m_{Y, Z, M}}&& X\otimes (Y\otimes (Z\otimes M))
}
\,\  \text{and }\,\ 
\xymatrix@C=0.8em 
{
& X\otimes M &\\
(X\otimes 1) \otimes M \ar[rr]^{m_{X, 1, M}} \ar[ru]^{r_{X}\otimes \id}
&& X\otimes (1 \otimes M ) \ar[lu]_{\id \otimes l_{M}}
}
\Omit{
\xymatrix{
(X\otimes 1) \otimes M \ar[rr]^{m_{X, 1, M}} \ar[rd]^{r_{X}\otimes \id}
&& X\otimes (1 \otimes M ) \ar[ld]_{\id \otimes l_{M}}\\
& X\otimes M &
}}
\]
\end{definition}

For simplicity, from now on we assume $Q$ has no edge-loops. For any $w\in \N^I$, we define a category $\calD_w$. 
Roughly speaking, an object of $\calD_w$ is an integrable module with weights $\leq w$. 
\begin{definition}
An object in $\calD_w$  is a quasi-coherent sheaf $\mathcal{V}$ on $\calH_{E\times I}\times_SE^{(w)}$.
A morphism from $\mathcal{V}$ to $\mathcal{W}$ is a morphism of sheaves on $\calH_{E\times I}\times_SE^{(w)}$.
\end{definition}

We now define the tensor product $ \otimes_{t_1,t_2}^+:\calC \otimes \calD_w \to \calD_w$. 
We have the following two morphisms
\[\xymatrix{
E^{(v_2)}\times_SE^{(w)} &E^{(v_1)}\times_SE^{(v_2)}\times_SE^{(w)}\ar[l]_{pr}\ar[r]^{\bbS}& E^{(v_1+v_2)}\times_SE^{(w)}
}\]
Let $\mathcal{F}\in \calC$, and $\mathcal{G}\in \calD_w$, the tensor product $\calF\otimes_{t_1, t_2}^+\calG$ is defined as follows. 
We define the $v$-component of $\calF\otimes_{t_1, t_2}^+\calG$ to be 
 \[\sum_{v_1+v_2=v}\bbS_*\left((\calF_{v_1}\boxtimes \calG_{v_2,w})\otimes_{E^{(v_1)}\times_SE^{(v_2)}\times_SE^{(w)}}\calL_{v_1,v_2}\right)\]
By definition, it is clear that $\calF\otimes_{t_1,t_2}^+\calG$ is an object in $\calD_w$. 
Similarly, we also have $ \otimes_{t_1,t_2}^-:\calD_w \otimes\calC \to \calD_w$

To distinguish the tensor product in $\calC$ and the two actions of $\calC$ on $\calD_w$, we write the former as $\otimes^\calC$ and the later two as $\otimes^+$ and $\otimes^-$ whenever appropriate. 
\begin{lemma}\label{lem:ass of tensor}
We have the natural isomorphisms
\[
(\calF\otimes^{\calC} \calG)\otimes^{+} \calV
\cong \calF\otimes^{+} (\calG\otimes^{+} \calV), 
\,\  \text{and} \,\ 
\calV\otimes^{-} (\calG\otimes^{\calC} \calF)
\cong (\calV\otimes^{-}\calG )\otimes^{-}\calF , 
\]
for $\calF, \calG\in \calC$ and $\calV \in \calD_w$. 
\end{lemma}
\begin{proof}
This follows from Lemma~\ref{lem:relation of L}.
\end{proof}
\begin{prop}
The category $\calD_w$ is a module category over the monoidal category $\calC$. 
\end{prop}
\begin{proof}
This is a routine check, similar to the proof of Proposition~\ref{prop:C_monoidal}.
\end{proof}

\subsection{The elliptic cohomology of Nakajima quiver varieties}

For a quiver $Q$, let $Q^\heartsuit$ be the \textit{ framed quiver}, whose set of vertices is $I \sqcup I'$, where $I'$ is another copy of the set $I$, equipped with the bijection $I\to I'$, $i\mapsto i'$. 
The set of arrows of $Q^\heartsuit$ is, by definition, the disjoint union of $H$ and a set of additional edges $j_i : i \to i'$, one for each vertex $i\in I$.
We follow the tradition that $v\in \bbN^I$ is the notation for the dimension vector at $I$, and $w\in \bbN^I$ is the dimension vector at $I'$. 

Let $\mu_{v, w}:T^*\Rep(Q^\heartsuit,v,w)\to \fg\fl_v^*\cong \fg\fl_v$ be the moment map
 \[
\mu_{v, w}: (x, x^*, i, j)\mapsto 
\sum[x, x^*]+i\circ j \in \fg\fl_v.
\]
Let $\chi: G_v\to \Gm$ be the character $g=(g_i)_{i\in I} \mapsto \prod_{i\in I} \det (g_i)^{-1}.$ 
The set of $\chi$-semistable points in $T^*\Rep(Q^\heartsuit, v, w)$ is denoted by $\Rep(\overline{Q^\heartsuit},v,w)^{ss}$.
The Nakajima quiver variety is defined to be the Hamiltonian reduction 
\[\fM(v, w):=\mu_{v, w}^{-1}(0)/\!/_\chi G_v.\] 
\Omit{Let $\fM_{0}(v, w):=\mu_{v, w}^{-1}(0)/\!/ G_v$ be the affine quotient. We then have a natural projection $\fM(v, w) \to \fM_{0}(v, w)$. 
For any $x\in  \fM_{0}(v, w)$, let $\fM_{x}(v, w)$ be the fiber of this natural map. }

For each pair $v,w\in \bbN^I$, $\Ell_{G_v\times G_w \times \G_m^2} (\mu_{v, w}^{-1}(0)^{ss})$ is a sheaf on $E^{(v)}\times E^{(w)}\times E^2$.
Let $\pi: E^{(v)}\times E^{(w)}\times E^2\to E^{(w)}\times E^2$ be the natural projection. 
Since $G_v$ acts on $\mu_{v, w}^{-1}(0)^{ss}$ freely, we have an isomorphism 
$
\Ell_{G_w\times \G_m^2} (\fM(v, w))\cong \pi_*\Ell_{G_v\times G_w \times \G_m^2} (\mu_{v, w}^{-1}(0)^{ss}).
$ In this sense, we will consider $\Ell_{G_w\times \G_m^2} (\fM(v, w))$ as a coherent sheaf of algebras on $E^{(v)}\times E^{(w)}\times E^2$. Therefore, $\Ell_{G_w\times \G_m^2} ( \sqcup_{v\in \N^I} \fM(v, w))$ is an object in $\calD_{w}$. 

\begin{prop}\label{lem:M_PMod}
The object $\Ell_{G_w\times \G_m^2} ( \sqcup_{v\in \N^I} \fM(v, w))$ has a natural structure as a right module over $\calP$. 
\end{prop}

\begin{proof}
Fix the framing $w\in \N^I$, we construct a map 
\begin{equation}\label{eq:action}
(\mathbb{S}_{12}\times \id_{E^2})_* \Big( \Ell_{G_w\times \G_m^2}(\fM(v_1, w))\boxtimes\calP_{v_2} \otimes \mathcal{L}_{v_1, v_2}\Big)
\to \Ell_{G_w\times \G_m^2}(\fM(v_1+v_2, w)), 
\end{equation}
where $\mathbb{S}_{12}: E^{(v_1)}\times E^{(v_2)} \to E^{(v_1+v_2)}$ is the symmetrization map. The map \eqref{eq:action} gives the claimed action. 

We start with the the Lagrangian correspondence formalism in \S~\ref{subsec: Lag corr},
specialized as follows:
We take $X'$ to be $\Rep(Q^\heartsuit, v, w)$ and $Y$ to be $\Rep(Q^\heartsuit, v_1, w)\times\Rep(Q,v_2)$.  Define $\calV$ to be
\[
\calV:=\{(x, j)\in \Rep(Q^\heartsuit, v_1+v_2, w) \mid x(V_1) \subset V_1\} \subset X'. 
\]
As in Section \S\ref{subsec: Lag corr}, set $X:=G\times_{P} Y$, $W:=G\times_P \calV$, and $Z:=T^*_W(X\times X')$ the conormal bundle of $W$. 
We then have the correspondence, see \cite[Lemma~5.2]{YZ}: 
\begin{equation} \label{corr for action}
\xymatrix@R=1.5em{
T^*_{G}X^{s}\ar@{^{(}->}[d]
&&Z_G^s\ar@{^{(}->}[d]\ar[ll]_{\overline\phi}\ar[r]^(0.4){\overline\psi}&
\mu_{v, w}^{-1}(0)^{ss} \ar@{^{(}->}[d]\\
G\times_PT^*Y^s\ar@{^{(}->}[r]^{\iota} & T^*X^s &
Z^s \ar[l]_(.4)\phi \ar[r]^(.4)\psi & T^*X'^s.
}\end{equation}
Where the $G$-varieties are given by
\begin{align*}
&
T^*_{G}X^{s}=G\times_{P}(\mu_{v_1, w}^{-1}(0)^{ss}\times \mu^{-1}_{v_2}(0)),\,\ \,\
Z_{G}^s=G\times_P\{ (x, x^*, i, j )\in \mu^{-1}_{v, w}(0)^{ss} \mid (x, x^*)(V_1)\subset V_1,   \text{Im}(i)\subset V_1\}.
\end{align*}
The left square of diagram \eqref{corr for action} is a pullback diagram. 
We have the following maps
\begin{align*}
&
\overline{\phi}^\sharp: 
\Ell_{G}( \Th_{T^*_{G}X^{s} } T^*X^{s})
\to \Ell_{G}( \Th_{Z_{G^s}} Z^{s}), \,\ \,\
\overline{\psi}_*: 
 \Theta(\psi) 
\to \Ell_{G}( \Th_{\mu_{v, w}^{-1}(0)^{ss} } T^*X'^s)\,\
\end{align*}
By taking the composition of $\overline{\psi}_* \circ \overline{\phi}^\sharp$, we get
\begin{align*}
\overline{\psi}_* \circ \overline{\phi}^\sharp: &
(\mathbb{S}_{12}\times \id_{E^2})_*\Big(\Ell_{G_{v_1}\times G_{v_2}}( \Th_{\mu_{v_1, w}^{-1}(0)^{ss}\times \mu^{-1}_{v_2}(0) }T^*Y^s) \otimes   \Theta(\iota)\otimes  \Theta(\psi) \Big)
\to \Ell_{G_v} ( \Th_{\mu_{v, w}^{-1}(0)^{ss} } T^*X'^s).
\end{align*}This line bundle $ \Theta(\iota)\otimes  \Theta(\psi) $ the same as in \S~\ref{subsec:ellCoHA}.
This gives the morphism \eqref{eq:action}.
\end{proof}

Recall that we also have the quiver variety $\fM^-(v, w)$ given by the GIT quotient $\mu_{v, w}^{-1}(0)/\!/_{-\chi} G_v$ corresponding to the opposite stability $-\chi$.
Similarly, we have the following. 
\begin{prop}\label{prop:left_act}
The object $\Ell_{G_w\times \G_m^2} ( \sqcup_{v\in N^I} \fM^-(v, w))$ has a natural structure as a left module over $\calP$. 
\end{prop}
\begin{proof}
The proof is similar to that of Lemma~\ref{lem:M_PMod}. See also \cite[Theorem~5.4]{YZ}.
\end{proof}

\subsection{The negative chamber of the root lattice}\label{subsec:negative_chamber}
For any object $\calV$ in $\calD_w$, the vector $w=\sum_{i\in I}w_i\overline{\omega}_i$ is the highest weight of $\calV$, where $\{\overline{\omega}_1, \overline{\omega}_2, \cdots, \overline{\omega}_n\}$ are the fundamental weights. 
The restriction $\calV_{v,w}:=\calV|_{E^{(v)}\times E^{(w)}}$, called the $v$-component of $\calV$,  has weight $\sum_{i\in I}w_i\overline{\omega}_i-\sum_{i\in I}v_i\alpha_i$.

Let $(s_k)_{i\in I}$ be the simple reflections of the Weyl group and $w_0$ be the longest element. 
The action of $s_k$ on the weight lattice is given by $s_k(\mu)=\mu-\angl{\alpha_k^\vee,\mu}\alpha_k$. Hence, for $\mu=\sum_{i\in I}w_i\overline{\omega}_i-\sum_{i\in I}v_i\alpha_i$, we have
\[s_k(\mu)=\sum_{i\in I}w_i\overline{\omega}_i-\sum_{i\in I}v'_i\alpha_i,\]
where $v'\in\bbN^I$ is such that $v'_j=v_j$ if $j\neq k$ and $v'_k=w_k-v_k-\sum_{\{j\mid j\neq k\}} c_{kj} v_j$.

The Weyl group action on cohomology of Nakajima quiver varieties is constructed in \cite{Nak03, Lus00, Maff}. In particular, assuming $v$ and $v'\in\bbN^I$ are such that $w_0(v)=v'$.
There is a correspondence \[\fM(v,w)\leftarrow F\rightarrow\fM^-(v',w),\] with the two arrows being principle $\GL_v$ and $\GL_{v'}$-bundles respectively. Denote  $\underline{\Spec}\Ell_{\GL_v\times \GL_{v'}\times\GL_w\times\Gm}(F)$ by $\fF$. We then have the following diagram
\[
\xymatrix{
E^{(v)}\times E^{(w)}\times E   &\fF\ar[l]_(0.3)p\ar[r]^(0.3)q& E^{(v')}\times E^{(w)}\times E
}
\]
and $p_*\calO_{\fF}\cong \Ell_{\GL_v\times\Gm}(\fM(v,w))$ and $q_*\calO_{\fF}\cong \Ell_{\GL_{v'}\times\Gm}(\fM^-(v',w))$.

Let $\calE_w$ be a category fibered over $\calD_{w}$, consisting of coherent sheaves on 
$\left\{\underline{\Spec}_{E^{(v)}\times E^{(w)}\times E}\Ell_{\GL_w\times\Gm}(\fM(v,w))\right\}_{v\in \N^I}$. In other words, $\calE_w$ consists of the objects in $\calD_{w}$ whose supports are contained in the support of $\Ell_{\GL_w\times\Gm}(\fM(v,w))$. Similarly, define $\calE_w^-\subset \calD_{w}$ to be a subcategory consisting of those sheaves whose supports are in the support of  $\{\Ell_{\GL_w\times\Gm}(\fM^-(v',w))\}_{v' \in \N^I}$. 
\begin{lemma}\label{lem:equivalence_w0}
There is an equivalence of categories $w_0:\calE_w\cong\calE_w^-$ , fibered over $\calD_w$. 
\[
\xymatrix@R=0.5em{\calE_w\ar[dr]\ar[rr]_{\cong}^{w_0}&&\calE_w^-\ar[dl]\\&\calD_w}.
\]
\end{lemma}
\begin{proof}
For each $v\in\bbN^I$, let $\calE_{v, w}$ be the category of coherent sheaves of modules over $\Ell_{\GL_w\times\Gm}(\fM(v,w))$; similarly for $\calE^-_{v, w}$.
The statement follows from the following sequence of isomorphisms
\begin{eqnarray*}
\calE_{v, w}&=&\Mod \Ell_{\GL_w\times\Gm}(\fM(v,w))\\
&\cong& \Coh\fF\\
&\cong& \Mod \Ell_{\GL_w\times\Gm}(\fM^-(v',w))=\calE_{v', w}.
\end{eqnarray*}
\end{proof}

\begin{lemma}\label{lem:left_right_tensor}
\begin{enumerate}
\item There are well-defined functors $\calP\otimes^+-:\calE_w^-\to \calE_w^-$ and $-\otimes^-\calP:\calE_w\to \calE_w$, which are compatible with the $\bbN^I$-gradings.
\item The equivalence $w_0$ from Lemma~\ref{lem:equivalence_w0} intertwines the two functors above. 
\end{enumerate}
\end{lemma}
\begin{proof}
Let $\calM_{v, w}:=\Ell_{G_w\times \Gm}(\fM(v, w))$. For any object $\calV\in \calE_{w}$, by restricting to the $v_2$-component, we have $\calV_{v_2}$ is a $\calM_{v_2,w}$-module. Let 
$
(\calP\otimes^+\calV)_v:=\bigoplus_{v_1+v_2=v} \mathbb{S}_*( \calP_{v_1}\boxtimes\calV_{v_2}\otimes\calL_{v_1,v_2} ). 
$
The fact that $\calP\otimes^+\calV$ lies in $\calE_w$ amounts to the $\calM_{v_1+v_2,w}$-module structure on $(\calP\otimes^+\calV)_{v_1+v_2}$, which in turn is induced from the map of sheaves 
\[
 \mathbb{S}_*(\calP_{v_1}\boxtimes\calM_{v_2,w}\otimes\calL_{v_1,v_2})\to \calM_{v_1+v_2,w}
\] on $E^{(v_1+v_2)}\times E^{(w)}$ coming from Proposition~\ref{prop:left_act}. 
The statement for $-\otimes^-\calP$ is proved similarly. 

The assertion (2) is clear by construction.
\end{proof}

Let $\sEnd(\calE_w)$ be the monoidal category of endofunctors of the abelian category $\calE_w$. 
The monoidal structure is given by composition of functors. 
By Lemma \ref{lem:left_right_tensor}, we have two objects $A_l(\calP):=\calP\otimes^+-$ and $A_r(\calP):=\omega_0(\omega_0(-)\otimes^-\calP)$ in the category $\sEnd(\calE_w)$.
Explicitly, for each $v_1$ and $v_2\in \bbN^I$, $w_0((w_0\calF_{v_1,w})\otimes^-\calP_{v_2})$ is a sheaf on $E^{(v_1+w_0v_2)}\times_SE^{(w)}$. Note that $w_0v_2$ is negative, hence the action $A_r(\calP)$ on $\calE_w$ decreases the weights. In particular, when $v_1+w_0v_2$ is not positive, then  the $v_1+w_0v_2$-component of $w_0((w_0\calF)\otimes^-\calG)$ is zero.

\begin{remark}
\begin{enumerate}
\item We expect there is a braided tensor category $\calE$, which contains $\calE_w$. The objects of $\calE$ are sheaves on $\coprod_{v\in \bbN^I}\coprod_{w\in \bbN^I} E^{(v)}\times E^{(w)}$, and we impose an equivalence relation between sheaves on $E^{(v)}\times E^{(w)}$  and $E^{(v')}\times E^{(w')}$  for different $(v,w)$ and $(v',w')$ as long as they have the same weight $\mu$. 
\item We further expect the category $\calE$ in (1) would also make sense even when $v$ lies in $\mathbb{Z}^I$. There should also be a (limit of) subcategory $\calE^{\infty}\subset \calE$. The objects of $\calE^{\infty}$ are concentrated on $E^{(v)}\times E^{\infty \rho}$, where $v\in \bbZ^I$, and $\rho\in (\mathbb{Z}_{>0})^I$. $\calE^{\infty}$ acts faithfully on $\calE$ via the tensor structure, therefore, one has an embedding $\calE^{\infty}\subset \sEnd(\calE_w)$. The braided tensor category $A_l(\calC)$ (resp. $A_r(\calC)$) coincides with $(\calE^{\infty})^+$ (resp. $(\calE^{\infty})^-$), whose objects are concentrated on $E^{(v)}\times E^{\infty \rho}$ for $v\in (\mathbb{Z}_{>0})^I$ (resp. $E^{(v)}\times E^{\infty \rho}$ for  $v \in  (\mathbb{Z}_{<0})^I$). 
 However, for the moment being, this is beyond the scope of the present paper.
\item  We will construct the Drinfeld double of $\SH^{\sph}$ via a pairing between $\SH^{\sph}$ and $\SH^{\sph,\coop}$. In order to do so, we need to view $\SH^{\sph}$ and $\SH^{\sph,\coop}$ as bialgebra objects in the same braided tensor category, and at the same time consider the dimension vectors on $\SH^{\sph,\coop}$ as negative root vectors. This would be true if (1) and (2) above were achieved.   In \S~\ref{subsec:DrinfeldDouble} we get around this issue by taking the image of $\SH^{\sph}$  in $\sEnd(\calE_w)$ via $A_l$, and the image of $\SH^{\sph,\coop}$ in $\sEnd(\calE_w)$ via $A_r$. 
\end{enumerate}
\end{remark}

\subsection{Drinfeld double of the Borel subalgebra}\label{subsec:DrinfeldDouble}
In this section, using a similar construction as in \cite{YZ2}, we construct the elliptic quantum group as the Drinfeld double of $\SH^{\sph}$. For this purpose, we need to construct a non-degenerate bialgebra pairing on the ad\`ele version of $\SH^{\sph}$. 

\Omit{We now recall the ad\`ele version of the shuffle algebra, constructed for 1-dimensional affine algebraic groups in \cite[\S~3.1]{YZ2}.} We fix a  non-vanishing,  translation-invariant section $\omega$ of $T^* E$. 
Denote by $F$ the field of rational functions of $ E$, $O_z$ the local ring of $ E$ at a closed point $z\in  E$, and $K_z$ the  local field. 
Let $\bbA$ be the restricted product $\prod_{z\in  E}'K_z$, i.e., an element in $\bbA$ is a collection of rational functions $\{r_z\}_{z\in  E}$, so that $r_z\in\calO_z$ except for finitely many $z$. 
We will call $\bbA$ the ring of ad\`eles.

More generally, for any $n\in \bbN$, and any closed point $(z_1,\dots,z_n)\in E^n$, consider the local ring $\calO_{(z_1,\dots,z_n)}$. Let $K_{(z_1,\dots,z_n)}$ be the localization of $\calO_{(z_1,\dots,z_n)}$ at the divisor $D$ consisting of points in $E^n$ whose $i$-th coordinate is equal to $z_i\in E$ for some $i\in[1,n]$. Let $\omega^n=\wedge_{i=1}^n\omega_i$ be the $n$-form on $E^n$ induced from $\omega$. Recall that $\Res_{(z_1,\dots,z_n)}f\cdot \omega^n$ is well-defined for any $f\in K_{(z_1,\dots,z_n)}$. Furthermore, for any $\sigma\in\fS_n$, we have
\begin{equation}\label{eqn:res_inv}
\Res_{(z_1,\dots,z_n)}f\cdot \omega=\Res_{\sigma(z_1,\dots,z_n)}(\sigma^*f)\cdot\omega,
\end{equation}
(see, e.g., \cite[Lemma~3.1]{YZ2}).

Define $\bbA^n(U)$ to be the restricted product of $K_z$ over all $z\in U\subseteq E^n$. Clearly this is a sheaf of $\calO_{E^n}$-modules on $E^n$. Let $\bbA^{(n)}$ be $(\bbS_*\bbA^n)^{\fS_n}$, which is a sheaf of $\calO_{E^{(n)}}$-modules on $E^{(n)}$, where $\bbS:E^n\to E^{(n)}$ is the natural projection. Similarly, for $v=(v^i)_{i\in I}\in\bbN^I$, we have $\bbA^{(v)}:=\prod_{i\in I} \bbA^{(v^i)}$ on $E^{(v)}$. 
We consider a braided monoidal category $\calC_\bbA$, where objects are sheaves of $\calO_{\calH_{E\times I}}$-modules, not necessarily quasi-coherent, endowed with the same monoidal structure as in $\calC$. 
We define the  ad\`ele version of the shuffle algebra to be $\SH_{\bbA}:=(\SH_{\bbA,n})_{v\in\bbN^I}$, where $\SH_{\bbA,n}:=\bbA^{(v)}$,  which is an object in $\calC_\bbA$. The morphism $\star$ \eqref{shuffle formula} makes $\SH_{\bbA}$ an algebra object, and the rational morphism $\Delta$ \eqref{eq:coprod} is a coproduct with poles. We also have the spherical subalgebra $\SH_{\bbA}^{\sph}$ of $\SH_{\bbA}$, generated by $\SH_{\bbA,e_k}$ for $k\in I$, which is a bialgebra object in $\calC_\bbA$.

Similar to \S~\ref{subsec:negative_chamber}, we have the two categories $\calD_{w,\bbA}$ and $\calE_{w,\bbA}$, for $w\in \N^I$, as well as the two algebra objects
$\SH_{\bbA}$ and $\SH^{\coop}_{\bbA}$ in $\sEnd(\calE_{w,\bbA})$.
Let $\iota: \SH_{\bbA,v}^{\sph}\inj\SH_{\bbA,v}$ be the natural embedding. For any local section $f$ of $ \SH_{\bbA,v}$ which lies in the image of $\SH_{\bbA,v}^{\sph}$, we will denote the corresponding section of $\SH_{\bbA,v}^{\sph}$ by $\iota^{-1}f$.
\footnote{
This $\iota^{-1}$ plays the same role as the factor $\frac{1}{\fac(z_A)}$ in \cite{YZ2}. }

We now construct  the Drinfeld double of $\SH^{\sph}_{\bbA}$ via defining a bialgebra pairing between  $A_l(\SH_{\bbA}^{\sph})$ and $A_r(\SH_{\bbA}^{\sph})^{\coop}$ in $\sEnd(\calE_{w,\bbA})$. Set
\[
D(\SH^{\sph}_{\bbA}):=A_l(\SH_{\bbA}^{\sph})\otimes_{\sEnd(\calE_{w,\bbA})}A_r(\SH_{\bbA}^{\sph})^{\coop}. \]
There is a unique algebra structure on $D(\SH^{\sph}_{\bbA})$, such that $A_l(\SH_{\bbA}^{\sph})$ and 
$A_r(\SH_{\bbA}^{\sph})^{\coop}$ are two subalgebras, and the commutation relation \cite[(3.3)]{Laug} (spelled out in the present setting as \eqref{eqn:double}). 

Recall that a bialgebra pairing is a morphism 
\[
(-,-): A_l(\SH_{\bbA}^{\sph})\otimes_{\sEnd(\calE_{w,\bbA})}A_r(\SH_{\bbA}^{\sph})^{\coop}\to \epsilon.
\] 
By definition of $\epsilon$ from \S~\ref{subsec:cat C},  the pairing between $\SH_{v,\bbA}^{\sph}$ and $(\SH_{v',\bbA}^{\sph})^{\coop}$ is zero if $v\neq v'$. 
On $(\SH_{\bbA}^{\sph })_v\otimes (\SH_{\bbA}^{\sph })_v$,
we define the pairing $(\cdot, \cdot)$ as 
\[
(f, g):= \sum_{z\in  E}\Res_{z } \Big( \frac{\iota^{-1}f(z_A)\cdot \iota^{-1}g(-z_A)}{ { |A|!} } \omega\Big), 
\] 
for any local sections $f, g$ of $ \SH_{\bbA,v}$ which lie in the image of $\SH_{\bbA,v}^{\sph}$, where $-z_A$ is the group inverse of $E$ and $|A|!=\prod_{i\in I} |A^{i}|!$.

\begin{theorem}
\label{thm:pairing}
The pairing $(\cdot,\cdot )$ between $\SH^{\sph}_{\bbA}$ and $(\SH^{\sph}_{\bbA})^{coop}$ is a  bialgebra pairing, which is non-degenerate on both factors. 
\end{theorem} 
For the meaning of bialgebra pairing, see e.g., \cite[\S1.7, Figure (1.1)]{Laug}.
\begin{proof}
The theorem follows essentially from the same proof as in \cite{YZ2}. For completeness, we include the proof here.

Let $v_1, v_2\in \N^I$ be two dimension vectors with $v=v_1+v_2$. For  $f_1\in \SH_{\bbA, v_1}, f_2\in \SH_{\bbA, v_2}, P\in \SH_{\bbA, v}$,
 we need to show that
\begin{equation}\label{eq:proof bialgebra}
\sum_{(A, B)\in\bfP(v_1,v_2)}\Big(f_{1}(z_A) \cdot f_2(z_B) \cdot \fac(z_A|z_B), P\Big)
=\Big(f_1(z_{A_o})\otimes f_2(z_{B_o}),  \Phi_{A_o,B_o} \Delta P(z)\Big),
\end{equation}
where $(A_o,B_o)=([1,v_1],[v_1+1,v])$ is the standard element in $\bfP(v_1,v_2)$. 
Using \eqref{eqn:res_inv}, the left hand side of \eqref{eq:proof bialgebra}  is the same as
\begin{align*}
&\sum_{z\in \overline{\bbG}}\Res_{z}\iota_{A_o\cup B_o}^{-1}\sum_{\{(A, B)\}}\Big(\frac{f_{1}(z_A) \cdot f_2(z_B) \cdot \fac(z_A|z_B)\cdot }{  |A\cup B|! }\Big)\iota_{A_o\cup B_o}^{-1}P(-z_A\otimes -z_B) w_{A\cup B}\\
=&
\sum_{z\in \overline{\bbG}}\Res_{z}\iota_{A_o\cup B_o}^{-1}\Big(\frac{f_{1}(z_{A_o}) \cdot f_2(z_{B_o}) \cdot \fac(z_{A_o}|z_{B_o})\cdot P(-z_{A_o}, -z_{B_o})}{ |{A_o}|! |{B_o}|! )}\Big) w_{A_o\cup B_o}
\end{align*}
The right hand side of \eqref{eq:proof bialgebra} is the same as
\begin{align*}
&\Big(f_1(z_{A_o})\otimes f_2(z_{B_o}),  \frac{ P(z_{A_o}\otimes z_{B_o})  }{\fac(z_{A_o}| z_{B_o})}\Big)\\
=&\sum_{z\in \overline{\bbG}}\Res_{z} 
\iota_{A_o}^{-1}\frac{f_{1}(z_{A_o})}{ |{A_o}|!  } \cdot  \iota_{B_o}^{-1}\frac{f_{2}(z_{B_o})}{ |{B_o}|!   } \cdot (\iota_{B_o}^{-1}\otimes\iota_{A_o}^{-1}\frac{ P(-z_{A_o}\otimes -z_{B_o})}{ \fac(-z_{A_o}| -z_{B_o}) } )
w_{A_o}w_{B_o}.
\end{align*}
Note that $\fac(-z_{A_o}| -z_{B_o})=\fac(z_{B_o}| z_{A_o})$ under assumption of Remark~\ref{rmk:weights}(2).
\begin{align*}
&\sum_{z\in \overline{\bbG}}\Res_{z} 
\iota_{A_o}^{-1}\frac{f_{1}(z_{A_o})}{ |{A_o}|!  } \cdot  \iota_{B_o}^{-1}\frac{f_{2}(z_{B_o})}{ |{B_o}|!   } \cdot (\iota_{B_o}^{-1}\otimes\iota_{A_o}^{-1}\frac{ P(-z_{A_o}\otimes -z_{B_o})}{ \fac(-z_{A_o}| -z_{B_o}) } )
w_{A_o}w_{B_o}\\
=&\sum_{z\in \overline{\bbG}}\Res_{z}\iota_{A_o\cup B_o}^{-1}\Big(\frac{f_{1}(z_{A_o}) \cdot f_2(z_{B_o}) \cdot \fac(z_{A_o}|z_{B_o})\cdot P(-z_{A_o}, -z_{B_o}) }{ |{A_o}|! |{B_o}|!   }\Big) w_{{A_o}\cup {B_o}}.\end{align*}
Here we have a morphism $\SH_{v_1}\boxtimes \SH_{v_2}\to \bbS^*\SH_{v_1+v_2}\otimes\calL^\vee_{v_1+v_2}$ of sheaves on $E^{A_o}\times E^{B_o}$.  
The subsheaf $\SH_{v_1}^{\sph}\boxtimes\SH_{v_2}^{\sph}\subset \SH_{v_1}\boxtimes \SH_{v_2}$ factors through $\bbS^*\SH_{v_1+v_2}^{\sph}$. 
\[
\xymatrix{
\SH_{v_1}\boxtimes \SH_{v_2}\ar[r]& \bbS^*\SH^{\sph}_{v_1+v_2}\otimes\calL^\vee_{v_1+v_2}\\
\SH_{v_1}^{\sph}\boxtimes\SH^{\sph}_{v_2}\ar@{^{(}->}[u]^{\iota_{v_1}
\otimes\iota_{v_2}}\ar[r]&\bbS^*\SH_{v_1+v_2}^{\sph}\ar@{^{(}->}[u]^{\iota_{v_1+v_2}}
}\]
Therefore, the equality \eqref{eq:proof bialgebra} holds. 
\end{proof}

This gives a bialgebra structure on $D(\SH^{\sph}_{\bbA})$. 

We define the {\it sheafified elliptic quantum group} $\Ell_{\hbar, \tau}(Q)$ to be the subalgebra of   $D(\SH^{\sph}_{\bbA})$, generated by $\SH_{e_k}\subseteq \SH_{\bbA, e_k}$ and $(\SH_{e_k})^{\coop}\subseteq (\SH_{\bbA,e_k})^{\coop}$ for $k$ varies in $I$.  
Note that $\Ell_{\hbar, \tau}(Q) \in \sEnd(\calE_{w,\bbA})$ lies in $\sEnd(\calD_w)$.

\begin{definition}
\begin{enumerate}
\item A representation of $\Ell_{\hbar, \tau}(Q)$ is an object in $\calE_w$ for some $w\in \N^I$, which is endowed with a structure of modules over $\Ell_{\hbar, \tau}(Q) \in \sEnd(\calE_w)$. The category of representations of the sheafified elliptic quantum group is denoted by $\Mod \Ell_{\hbar, \tau}(Q)$. 
\item The category of finite dimensional representations of the sheafified elliptic quantum $\Mod_f \Ell_{\hbar, \tau}(Q)$ is the full subcategory of $\Mod \Ell_{\hbar, \tau}(Q)$ consisting of objects that are supported on zero-dimensional subschemes as coherent sheaves on $\calH_{E\times I}\times_SE^{(w)}$.
\end{enumerate}
\end{definition}

\subsection{Action of the Drinfeld double on quiver varieties}
To give some examples of objects in $\Mod \Ell_{\hbar, \tau}(Q)$, we prove here that $\Ell_{G_w}(\mathfrak{M}(w))$ is an object in $\Mod \Ell_{\hbar, \tau}(Q)$ for any $w\in \bbN^I$.

Recall that there is a right action of $\calP$ on $\calM^-$:
 \[\calM^-_{v_1,w}\otimes^-\calP_{v_2}\to \calM^-_{v_1+v_2,w}.\] 
Using notations defined after Lemma~\ref{lem:left_right_tensor}, 
we get a morphism 
\[A_r(\calP_{w_0v_2})(\calM_{v_1,w})\to\calM_{v_1+w_0v_2,w}.\]
Note that $w_0v_2$ is a negative root vector. Therefore we have an action of $A_r(\calP)$ on $\calM_w$, which will be referred to as the action by lowering operators. 

\begin{theorem}\label{thm:DoubAct}
The Drinfeld double $D(\calP^{\sph})$ acts on $\Ell_{G_w}(\mathfrak{M}(w))$, for any framing $w\in \N^I$. 
\end{theorem}
\begin{proof}
Recall that it suffices to show that the pairing $(\cdot,\cdot)$ is compatible with the actions of $\SH^{\sph}$ and $\SH^{\sph,\coop}$ in the sense of, e.g., \cite[(3.3)]{Laug}.
By definition, it suffices to show that for any local section $f$, $e$, and $p$ of $\SH^{\sph}_{e_k}$, $\SH^{\sph,\coop}_{e_k}$, and $\calM_{v,w}$ respectively,  we have
\begin{equation}\label{eqn:double}
(f_\bullet p)_\bullet e+\angl{f,e}\Phi_{v,e_k}(p)=\Phi_{e_k,v}(p)\angl{f,e}+f_\bullet(p_\bullet e).
\end{equation}
By the same argument as in \cite[\S~11.3]{Nak01}, we only need to verify this in the case when $\fg_Q$ is $\fs\fl_2$. The latter is done in \S~\ref{subsec:sl2}.
\end{proof}

\section{Reconstruction of the Cartan action}

In this section, we provide explicit formulas of the action of the Cartan subalgebra of elliptic quantum group on the elliptic cohomology of quiver varieties. We also show the corresponding action of $D(\calP)$ coincides with the action of the Hecke operators of Nakajima. 

We also define and show the existence of Drinfeld polynomials for highest weight modules in the elliptic setting.

\subsection{Reconstruction of the Cartan action in $\calE_w$}\label{subsec:CartanModule}
Similar to \S~\ref{Cartan subalgebra}, objects in $\calE_w$ admit a natural action of the Cartan subalgebra $\SH^0$, which comes from the braiding switching the two actions $A_l(\calP)$
 and $A_r(\calP)$ on $\calE_w$. We spell out the formula for this braiding.
 
The element $w_0$ induces a bijection on the set of simple roots, which is denoted by $w_0(k)=k'$. For any $k\in I$,  recall here that $\calP_{e_k}\cong\calO_{e_k}$.
For $\calF\in \calE_w$ and $v\in \bbN^I$, on $E^{e_k}\times E^{(v)}\times E^{(w)}\times E$ we have $A_l(\calO_{e_k})\calF_v=\calO_{E^{e_k}}\boxtimes \calF_v\otimes \calL_{e_k,v}$. Similarly on $E^{e_{k'}}\times E^{(v')}\times E^{(w)}\times E$ we have $A_r(\calO_{e_{k'}})\calF_v=\calO_{e_{k'}}\boxtimes w_0(\calF_v)\otimes \calL_{e_{k'},v'}$.
Recall that $w_0$ is defined using the  diagram
\[
\xymatrix{
E^{(v)}\times E^{(w)}\times E   &\fF\ar[l]_(0.3)p\ar[r]^(0.3)q& E^{(v')}\times E^{(w)}\times E
.}
\]
We denote the coordinates of $E^v$ by 
$(z_i^{(j)})_{i\in I,j=1,\dots,v^i}$,  the coordinates of $E^{v'}$ by 
$(z_i'^{(j)})_{i\in I,j=1,\dots,v'^{i}}$, and the coordinates of $E^{w}$ by $(t_i^{(j)})_{i\in I,j=1,\dots,w^{i}}$.
Consider both $A_l(\calO_{e_k})\calF_v$ and $A_r(\calO_{e_{k'}})\calF_v$ as sheaves on $\fF$, we have sections
$\fac(z_{[1,v]}|z)$ of $\calL_{e_k,v} $ where $z$ is the coordinate of $E^{e_k}$, and the section $\fac(z'_{[1,v']}|z)$ of  $\calL_{e_{k'},v'}$ where $z$ is the coordinate of $E^{e_{k'}}$.
Pushing forward to $E^{e_k}\times E^{(v)}\times E^{(w)}\times E$, the natural section $\frac{\fac(z_{[1,v]}|z)}{w_0^{-1}\fac(z'_{[1,v']}|z)}$ of the line bundle $\calL_{e_k,v}^\vee\otimes w_0^*\calL_{w+w_0v,e_{k'}}$ gives a map \[A_l(\calO_{e_k})\calF_v\to A_r(\calO_{e_{k'}})\calF_v.\]
This is the braiding, denoted by $\Phi^-_{e_k,v}$. Similarly, we also have $\Phi^+_{v',e_{k'}}:A_r(\calO_{e_{k'}})\calF_v\to A_l(\calO_{e_k})\calF_v$. Similar to \S~\ref{subsec:rational sec}, we define $\Phi_k$ to be $\vartheta(\hbar)(\Phi^-_{e_k,v}-\Phi^+_{v',e_{k'}})$.

\begin{example}\label{ex:sl2_1}
We consider the Nakajima quiver variety of type $A_1$. The quiver $Q$ has only one vertex without any arrows. Therefore, the dimension vectors $v$ and $w$ are natural numbers. Let $\Gr_v(w)$ be the Grassmannian of $v$-dimensional subspaces in the $w$-dimensional vector space $W$. Then, the quiver variety $\mathfrak{M}(v, w)$ of type $A_1$ is the cotangent bundle $T^*\Gr_v(w)$ of the Grassmannian $\Gr_v(w)$. The general linear group $G_w=\GL_w$ acts on $T^*\Gr_v(w)$ naturally, and the torus $\mathbb{G}_m$ acts by scalar multiplication on the fibers of the cotangent bundle.
Its equivariant elliptic cohomology is the structure sheaf of $E^{(v)}\times E^{(w-v)}\times E$, which is the quotient of $E^w\times E$ by $\fS_v\times\fS_{w-v}$.
The coordinates of $E^{w}$ will be denoted by $t_1,\dots,t_w$. Then, $\fac(z_{[1,v]}|z)=\prod_{j=1}^{v}\frac{\vartheta(z-t_j-\hbar)}{\vartheta(z-t_j)}$.

In this case, $I$ has one single element, denoted by $k$.
Let the parameter of $E^{e_k}$ be denoted by $z$. 
For any $v<w$, we have $w_0(v)=w-v$.
Recall from \cite{Lus00} that the correspondence defining $w_0$ consists of short exact sequences \[0\to V'\to qW\to q^2V\to 0,\]
hence, $\fF$ from \S~\ref{subsec:negative_chamber} is the sub-variety of $E^{(v)}\times E^{(w)}\times E^{(w-v)}$ consisting of points such that $z_i=t_i$ for $i=1,\dots,v$ and $z_i'=t_{w-i}-\hbar$, which is abstractly isomorphic to $E^{(v)}\times E^{(w-v)}$.
In particular, we have $w_0^{-1}(\fac(z_{[1,w-v]}|z))=\fac((t_{v+1}+\hbar,\dots,t_{w}+\hbar)|z)$. Therefore, \begin{equation}\label{eqn:Phi_sl2}
\Phi_k=\frac{\prod_{i=1}^v\vartheta(z-t_i+\hbar)\prod_{i=v+1}^w\vartheta(z-t_i-\hbar)}{\prod_{i=1}^w\vartheta(z-t_i)}
\end{equation}

We also have the following geometric interpretation. 
Here $W$ is the trivial bundle  on $T^*\Gr_v(w)$ with the natural $\GL_w$-action,  and $V$ is the tautological subbundle which is $\GL_w$-equivariant. 
Then, formula \eqref{eqn:Phi_sl2} can be interpreted as $\lambda_z((q-q^{-1})W+(q^2-q^{-2})V)$, where $\lambda_z$ is the equivariant Chern polynomial defined in \S~\ref{subsec:char_class}. 
\end{example}
In  general, we also have the following.
\begin{prop}\label{prop:cartan_H}
On $\calE_w$ the operator $\Phi_k$ is also equal to
\[H_k=\lambda_z((q-q^{-1})W_k+(q^2-q^{-2})V_k+(q^{-1}-q)\sum_{l\neq k}[-c_{kl}]_qV_l).\]
\end{prop}
\begin{proof}
By the same argument as in \cite[\S~11.3]{Nak01}, we only need to verify this in the case when $\fg_Q$ is $\fs\fl_2$, which is proved in Example~\ref{ex:sl2_1} and \S~\ref{subsec:sl2}.
\end{proof}

In view of \eqref{eqn:double}, $\Phi_k$ is the formula of the Cartan-action on objects in $\calE_w$.

\begin{definition}
We say that $\calF$ is a highest weight module of highest weight $w$ if  on $E^{(w)}$ the component $\calF_{0,w}$ is a cyclic $\calO_{E^{w}}$-module, and $\calF_{v,w}$ is generated by $\calF_{0,w}$ and the $\calP$-action.
\end{definition}

For $w\in \bbN^I$, the coordinates of $E^{w}$ are denoted by $(t_i^{(j)})_{i\in I,j=1,\dots,w^i}$.
\begin{definition} 
For a highest weight module, for any $k\in I$, the action $\Phi_k$ on $\calF_{0,w}$ is by the section 
\[\frac{1}{\omega_0^{-1}\fac(z_{[1,w]}|z)}=\prod_{j=1}^{w^k}\frac{\vartheta(z-t_k^{(j)}-\hbar)}{\vartheta(z-t_k^{(j)})}.\]
This collection of rational functions on $E^{(w)}$ is the elliptic Drinfeld polynomials.
\end{definition}

\subsection{Recovery of the dynamical parameters}
Similar to \S~\ref{subsec:dyn_cartan}, we have a version of the category $\calD_w'$ involving the dynamical parameters, which is equivalent to $\calD_w$. For the convenience of the readers, we spell out the details here. 

We take the elliptic curve $E$ as in \S~\ref{subsec:universal elliptic curve}, notations $\calH_{E\times I}'$ and $E^{(w)'}$ are as in \S~\ref{subsec:universal elliptic curve}. On $E^{(v)'}\times_{E_{\lambda}} E^{(w)'}$, we have the line bundle $\bbL^{v,w}$ induced by the natural Poincar\'e line bundle on $E$. 
Note the difference between the projections $E^{(v)'}\to E_{\lambda}$ and $E^{(w)'}\to E_{\lambda}$, due to the difference between the two basis $\{\overline{\omega}_i\}$ and $\{\alpha_i\}$ as has been seen in \S~\ref{subsec:negative_chamber}.

By Proposition~\ref{prop:cartan_H} and the definition of the total Chern polynomial in \S~\ref{subsec:char_class}, $\Phi_k$ is a section of $\bbL_{k,\hbar}^{v,w}:=\Theta(k^{-1}\otimes ((q-q^{-1})W_k+(q^2-q^{-2})V_k+(q^{-1}-q)\sum_{l\neq k}[-c_{kl}]_qV_l))$.  
Hence, similar to  Lemma~\ref{lem:5.3}, we have 
\[\bbL_{k,\hbar}^{v,w}\cong \calL_{e_k,v}\calL_{v',e_k}^\vee\]
Here we identify the equivariant parameter of the $\Gm$ acting on $k$ in the definition of $\lambda_z$ with the natural parameter of $E^{e_k'}$.

Let $q:E^{e_k'}\times_{E_{\lambda}} E^{(v)'}\times_{E_\lambda}E^{(w)'}\times_{\calM_{1,1}}\mathfrak{E}\to E^{(v)'}\times_{E_\lambda}E^{(w)'}\times_{\calM_{1,1}}\mathfrak{E}$ be the subtraction map.
We now define a category $\calD'$. 
An object of $\calD'$ consists of a pair $(\mathcal{F},  \varphi_{\mathcal{F}})$, 
where 
\begin{itemize}
\item  $\calF$ is a coherent sheaf on $\calH_{E\times I}'\times_{E_{\lambda}} E^{(w)'}$, with certain support condition, and $\phi_{\calF}$ is 
an isomorphism
\[
\phi_{\calF}: w_0(w_0\calF_{v,w}\boxtimes\calO_{E^{e_{k'}}})\otimes\bbL_{k,\hbar}^{v,w}\to q^*\rho_k^{(v,w)*}\calF_v. 
\]
\item The morphism from $(\calF, \phi_{\calF})$ to $(\calG, \phi_{\calG})$ is a morphism of sheaves $f: \calF \to \calG$, such that, the following diagram commutes:
\[
\xymatrix{
 w_0(w_0\calF_{v,w}\boxtimes\calO_{E^{e_{k'}}})\otimes\bbL_{k,\hbar}^{v,w} \ar[d]\ar[r]^(0.75){\phi_{\calF}}&q^*\rho_k^{(v,w)*}\calF_{v,w}\ar[d]\\
   w_0(w_0\calG_{v,w}\boxtimes\calO_{E^{e_{k'}}})\otimes\bbL_{k,\hbar}^{v,w}\ar[r]^(0.75){\phi_{\calG}}&q^*\rho_k^{(v,w)*}\calG_{v,w}
}
\]
\end{itemize}
Similarly we have the category $\calE_w'$ consisting of objects in $\calD_w'$ with the required support condition. 

For any object $\calF$ in $\calE_w$, composing $\phi_k$ with the braidings $\Phi_k$, we get $\calO_{E^{e_k}}\boxtimes\calF_{v,w}\to \rho_k^{(v,w)*}\calF_{v,w}$; any morphism $f:\calF\to\calG$ in $\calE_w$ gives the commutative diagram
\[
\xymatrix{
 \calO_{E^{e_k}}\boxtimes\calF_{v,w} \ar[r]^(0.25){\Phi_{k} }\ar[d]_{f_v}& w_0(w_0\calF_{v,w}\boxtimes\calO_{E^{e_{k'}}})\otimes\calL_{e_k,v}\calL_{v',e_{k'}}^\vee \ar[d]\ar[r]^(0.75){\phi_{\calF}}&q^*\rho_k^{(v,w)*}\calF_{v,w}\ar[d]\\
 \calO_{E^{e_k}}\boxtimes\calG_{v,w}  \ar[r]^(0.25){\Phi_{k}}&   w_0(w_0\calG_{v,w}\boxtimes\calO_{E^{e_{k'}}})\otimes\calL_{e_k,v}\calL_{v',e_{k'}}^\vee\ar[r]^(0.75){\phi_{\calG}}&q^*\rho_k^{(v,w)*}\calG_{v,w}
}
\]

Similar to Lemmas~\ref{lem:trans_twist}, we have
\[q^*\rho^{(v,w)*}_k\bbL^{v,w}\cong (q^*\bbL^{v,w})\otimes\bbL_{k,\hbar}^{v,w}.\] Therefore, we have the following.
\begin{theorem}\label{thm:DynParam}
\begin{enumerate}
\item The category $\calD_w$ is equivalent to $\calD_w'$ as abelian categories. 
\item The equivalence $\calE_w \cong \calE_w'$ is compatible with the isomorphism $\calC\cong \calC'$ and their respective actions.
\item Under these equivalences above, $\Mod D(\SH^{\sph})\cong \Mod D(\SH^{\sph})'$.
\end{enumerate}
\end{theorem}
Note that (3) is a formal consequence of the first two. 
\begin{proof}The proof is similar to that of Proposition~\ref{prop:dyn_param}.
\end{proof}

\begin{remark}
\begin{enumerate}
\item This gives a conceptual interpretation that the irreducible representations of the dynamical elliptic quantum group does not depend on the dynamical parameter, a phenomenon discovered by Felder-Enriquez in the case of $\fs\fl_2$, and by Gautam-Toledano Laredo in general.

\item We have the Drinfeld coproduct, the definition of which does not involve the dynamical parameters. However, the ordinary coproduct may depend on the dynamical parameters in a non-trivial way, which we will study in future investigations. 
\end{enumerate}
\end{remark}

\subsection{Quiver varieties in the $\fs\fl_2$-case}\label{subsec:sl2}
Consider $O(v, w):=\{(V^1, V^2)\in \Gr_{v-1}(w)\times \Gr_v(w) \mid V^1\subseteq V^2\}$.
The Hecke correspondence $C^+(v, w)$ is the conormal bundle to $O(v, w)$.

Let $Z(v, v-1, w):= T^*(\Gr_{v-1}(w))\times_{\mathfrak{M}_0(w)} T^*(\Gr_{v}(w))$ be the Steinberg variety, where $\mathfrak{M}_0(w)$ is the quiver variety with trivial stability condition.
Then, $C^+(v, w)\subset  Z(v, v-1, w)$.  We have the following diagram. 
\[\xymatrix@R=1em @C=0.1em{
&Z(v, v-1, w)\ar@{^{(}->}[rr]\ar[d]^{\pi} && T^*(\Gr_{v-1}(w))\times T^*(\Gr_v(w))\ar[d]^{\pi_1\times \pi_2}&\\
&O(v, w)\ar[dl]_{p_1}\ar[dr]^{p_2}\ar@{^{(}->}[rr]&& \Gr_{v-1}(w)\times \Gr_v(w)&\\
\Gr_{v-1}(w)& &\Gr_v(w)
}\]
Let $0\to R_1 \to W\to Q_1\to 0$ and  $0\to R_2 \to W\to Q_2\to 0$ 
 be the tautological sequences over $\Gr_{v-1}(w)$, and $\Gr_v(w)$ respectively, with $\dim(R_1)=v-1$ and $\dim(Q_1)=w-v+1$, and $\dim(R_2)=v$. The first projection $p_1: O(v, w)\to \Gr_{v-1}(w)$ can be identified with 
the projective bundle $\mathbb{P}(Q)\to \Gr_{v-1}(w)$, and the second projection $p_2$ can be identified with $\PP(R_2^{\vee})$. 
\[
\xymatrix @R=1em{
 & \PP(Q_1)\cong \PP(R_2^{\vee}) \ar[ld]_{p_1}\ar[rd]^{p_2} &\\
\Gr_{v-1}(w) && \Gr_v(w) 
}\]
Let $\mathbb{L}$ be the Poincar\'e line bundle on $E$ which has rational section $g_\lambda(z)=\frac{ \vartheta(z+\lambda)}{ \vartheta(z) \vartheta(\lambda)}$. 
For a vector bundle $V\to X$ on algebraic variety $X$, the Chern classes $c_{i}^{\lambda}(V)$ of $V$ in the elliptic cohomology are certain rational sections of $\Ell_{G}^{\lambda}(X)$.

Let $\calL:=\sO_{\mathbb{P}(Q)}(-1)$. 
Let $\varphi: T^*(\Gr_{v-1}(w))\times T^*(\Gr_{v}(w))\to T^*(\Gr_{v}(w))\times T^*(\Gr_{v-1}(w))$ be the map that permutes the components.
We have the following natural embeddings $\Delta^+\,:\,C^+(v, w)\inj Z(v-1, v; w)$, and 
$\Delta^-\,:\,C^-(v, w)\inj Z(v, v-1; w)$,  
where $C^-(v, w):=\varphi(C^+(v, w))$.

There are very few interesting global section, therefore, we study actions of the rational sections from Example~\ref{ex:rat_sec_sl2}.
If $r\geq 0$, put
$x_{r}^{\pm}=\sum_{v}(-1)^{v\pm 1}
\Delta^{\pm}_*(c_1^\lambda(q^2\mathcal{L}^\pm))^{r}\in \Ell_{G_w\times \G_m}(Z (w)). $
\Omit{
The  function $f(z)$ is asymptotic to the following
\[
f(z) \sim \sum_{n\in \Z}\oint_a \Big(\frac{f(\zeta)}{\zeta^n} d\zeta\Big)z^n
\]
by definition of the Laurent expansion.
Generally, we have $ f(z+x) \sim \sum_{n\in \Z}\oint_a \Big(\frac{f(x+\zeta)}{\zeta^n} d\zeta\Big)z^n$
\begin{align*}
&\hat{b}(\lambda):=\sum_{r=0}^{\infty} \star g^r_{\lambda+v\hbar}(\zeta)z^r\,\ \text{and}\,\ 
\hat{c}(\lambda):=\sum_{r=0}^{\infty} \star g^r_{-\lambda-(w-v)\hbar}(\zeta)z^r\\
&\hat{\Phi}:=\star H_{v, w}(-z), \text{where} \,\ H_{v, w}:=\frac{\prod_{i=1}^{v} \vartheta(z-t_i+\hbar) \prod_{i=v+1}^w \vartheta(z-t_i-\hbar)}{\prod_{i=1}^w \vartheta(z-t_i)}
\end{align*}}
Then, we have maps:
\begin{align*}
&x_r^+: \Ell^{\lambda}_{G_w\times \C^*}(T^*\Gr(v, w)) \to \Ell^{\lambda}_{G_w\times \C^*}(T^*\Gr(v-1, w))\otimes\Theta(\pi_1) , \\
&x_r^-: \Ell^{\lambda}_{G_w\times \C^*}(T^*\Gr(v, w))  \to \Ell^{\lambda}_{G_w\times \C^*}(T^*\Gr(v+1, w))\otimes\Theta(\pi_2).
\end{align*}
Taking $f$ to be a rational section of the domain, which by Example~\ref{ex:sl2_1} can be thought of as a function on 
$E^{(v)'}\times_{\calM_{1, 2}} E^{(w-v)'}\times_{\calM_{1, 1}} \mathfrak{E}$.  The actions from Propositions~\ref{lem:M_PMod} and \ref{prop:left_act} have the following formulas:
\begin{align}
&x_r^+(f)=\sum_{k=v}^{w} f(t_{[1, v-1]\cup k}; t_{[v, w]\backslash \{k\}}) g^r_{\lambda+v\hbar}(t_k) \prod_{m\in [v, w] \backslash k} \frac{\vartheta(t_k -t_m +\hbar)}{\vartheta(t_{k}-t_m)}, \label{eq:x^+}\\
&x_r^-(f)=\sum_{k=1}^{v+1}f(t_{[1, v+1] \backslash \{k\} }; t_{[v+2, w]\cup \{k\}}) g^r_{-\lambda-(w-v)\hbar}(t_k) 
 \prod_{m\in [1, v+1] \backslash \{k\}} \frac{\vartheta(t_m -t_k +\hbar)}{\vartheta(t_{m}-t_k)}.\label{eq:x^-}
\end{align}
The formulas are obtained using the identities $
c^\lambda_{1}(q^2 \calL)=g_{\lambda+v\hbar}(z)=\frac{\vartheta(z+v\hbar+\lambda)}{\vartheta(z)\vartheta(\lambda+v\hbar)}
=g_{\lambda}(z+v\hbar) \frac{\vartheta(z+v\hbar)\vartheta(\lambda)}{\vartheta(z)\vartheta(\lambda+v\hbar)},
$
and $c^\lambda_1(q^2\mathcal{L}^\pm))^{(r)}$ is determined as 
$c^\lambda_1(q^2\mathcal{L}^\pm))^{(r)}=g_{\lambda +v\hbar} ^{(r)}(z):=\frac{\partial^r g_{\lambda +v\hbar}(z)}{ r!\partial^rz}$. 
Therefore, 
\[\sum_{r\geq 0}g_{\lambda +v\hbar} ^{(r)}(z)u=
\sum_{r\geq 0}\frac{\partial^rg_{\lambda +v\hbar} ^{(r)} }{r!\partial^rz}(z)u^r=
g_{\lambda +v\hbar} (z+u).
\]

Let ${\Phi}_{\lambda}$ be the convolution with $H_{v, w}(-z)$, where
\[
H_{v, w}(z):=\frac{\prod_{i=1}^{v} \vartheta(z-t_i+\hbar) \prod_{i=v+1}^w \vartheta(z-t_i-\hbar)}{\prod_{i=1}^w \vartheta(z-t_i)} \in \Ell^{\lambda}_{G_w\times \C^*}(T^*\Gr(v, w)).
\] 
\begin{prop} 
\begin{enumerate}
\item The actions of the rising and lowering operators satisfy the relation \eqref{eqn:double}; In particular, they induce an action of the Drinfeld double. 
\item On a weight space $\mathbb{V}_\mu$, assume $\lambda_1+\lambda_2=\hbar \mu$. 
Then the operators $\mathfrak{X}^+(u, \lambda_1) , 
\mathfrak{X}^{-}(u, \lambda_2), \Phi(u)$  satisfy the relation \eqref{EQ5}:
\[
\vartheta(\hbar) [\mathfrak{X}^+(u, \lambda_1) , 
\mathfrak{X}^{-}(v, \lambda_2)]=
\frac{\vartheta(u-v+\lambda_{1}) }{\vartheta(u-v) \vartheta(\lambda_{1})} \Phi(v)
+\frac{\vartheta(u-v-\lambda_{2}) }{\vartheta(u-v) \vartheta(\lambda_{2})} \Phi(u). 
\]
\end{enumerate}
\end{prop}
\begin{proof}
Using the formulas \eqref{eq:x^+}, \eqref{eq:x^-}, we have the following
\begin{align*}
&x_s^{-} \circ x_{r}^+ (f) (t_{[1, v]}; t_{[v+1, w]})=\sum_{l=1}^{v} \sum_{k\in [v+1, w]\cup \{l\}} f(t_{[1, v]\backslash \{l, k\}}; t_{[v+1, w]\cup l \backslash \{k\}}) g^{(s)}_{-\lambda -(w-v)\hbar}(t_l) g^{(r)}_{\lambda +v\hbar}(t_k) X_{kl}, \\
&x_{r}^+ \circ x_s^{-} (f) (t_{[1, v]}; t_{[v+1, w]})=
\sum_{l\in [1, v] \cup k}\sum_{k=v+1}^w
 f(t_{[1, v] \cup k\backslash \{l\}}; t_{[v+1, w]\backslash \{k, l\}}) g^{(s)}_{-\lambda -(w-v)\hbar}(t_l) g^{(r)}_{\lambda +v\hbar}(t_k) Y_{kl}, 
\end{align*}
where
\begin{align*}
X_{k, l}:=\prod_{m\in [1, v]\backslash l} \frac{\vartheta(t_m -t_l +\hbar)}{\vartheta(t_m-t_l)}
\prod_{n\in [v+1, w] \cup l\backslash \{k\}} 
\frac{\vartheta(t_k -t_n +\hbar)}{\vartheta(t_k-t_n)}, \,\ \\
Y_{k, l}:=\prod_{m\in [1, v]\cup k \backslash \{l\}} \frac{\vartheta(t_m -t_l +\hbar)}{\vartheta(t_m-t_l)}
\prod_{n\in [v+1, w] \backslash \{k\}} 
\frac{\vartheta(t_k -t_n +\hbar)}{\vartheta(t_k-t_n)}. 
\end{align*}
Define the functions  $A(z):=\prod_{m=1}^w \vartheta(z-t_m)$ and $B(z):=\prod_{m=1}^v \vartheta(z-t_m-\hbar) \prod_{m=v+1}^w \vartheta(z-t_m +\hbar)$. By $\vartheta'(0)=1$, we have $A'(z)|_{z=t_k} = \prod_{m\in [1, w] } \widehat{\vartheta(t_k-t_m)}$. We have
\begin{align*}
&
 x_r^{+} x_s ^{-} (f) 
-  x_s^{-} x_r ^{+} (f) \\
=&\sum_{k=1}^w \frac{f }{\vartheta(\hbar)} g^{(s)}_{-\lambda-(w-v)\hbar}(t_k) g^{(r)}_{\lambda+v\hbar} (t_k)\frac{
\prod_{m\in [1, v]} \vartheta(t_k-t_m-\hbar) \prod_{m\in [v+1, w]} \vartheta(t_k-t_m+\hbar)}{\prod_{m\in [1, w] } \widehat{\vartheta(t_k-t_m)}}\\
=&\sum_{k=1}^w  \frac{f }{\vartheta(\hbar)}g^{(s)}_{-\lambda-(w-v)\hbar}(t_k) g^{(r)}_{\lambda+v\hbar} (t_k)  \Res_{z=t_k}\frac{B(z)}{A(z)}\\
=&\Res_{z\in E}x_r^{+} x_s ^{-} (f)\frac{B(z)}{A(z)},
\end{align*}
where the notation $\hat{}$ on the denominator means omitting the term when $m=k$.  By the formula \eqref{eqn:Phi_sl2} of $\Phi$ and the definition of the pairing \S\ref{subsec:DrinfeldDouble}, this proves (1).

Now we prove (2).
Writing in terms of generating series, and using the fact that $\sum_{r\geq 0}g_{\lambda +v\hbar} ^{(r)} (t)u^r=
g_{\lambda +v\hbar} (t+u)$, we have
\begin{align*}
&
\sum_{r, s\geq 0} x_r^{+} x_s ^{-} (f) w_1^sw_2^r 
- \sum_{r, s\geq 0} x_s^{-} x_r ^{+} (f) w_1^s w_2^r\\
=&\sum_{k=1}^w  \frac{f }{\vartheta(\hbar)}g_{-\lambda-(w-v)\hbar}(w_1+t_k) g_{\lambda+v\hbar} (w_2+t_k)  \Res_{z=t_k}\frac{B(z)}{A(z)}.
\end{align*}

Note that $g_{-\lambda-(w-v)\hbar}(z+w_1) g_{\lambda+v\hbar} (z+w_2)\frac{B(z)}{A(z)}$ 
as a function in $z$ is a double periodic meromorphic function. 
Let $\hat\lambda_1:=-\lambda-(w-v)\hbar$, $\hat{\lambda}_2:=\lambda+v\hbar$, 
and $w:=w_2-w_1$. 
Then, we have $\hat{\lambda}_1+\hat{\lambda}_2=(2v-w)\hbar=\mu\hbar$. 
Set $\mathfrak{X}^+(w_2, \hat\lambda_2):=\sum_{r\geq 0} x_r^+w_2^r$, and 
$\mathfrak{X}^-(w_1, \hat\lambda_1):=\sum_{s\geq 0} x_s^+w_1^s$. By the same argument as in the proof of Proposition~\ref{thm: rel DSH}, the following identity holds.
\begin{align*}
[\mathfrak{X}^+(w_2, \hat\lambda_2), \mathfrak{X}^-(w_1, \hat\lambda_1)]
=&\frac{1}{\vartheta(\hbar)}\left(
\frac{\vartheta(w+ \hat{\lambda}_2)}{\vartheta(w) \vartheta(\hat{\lambda}_2)}H_{v, w}(-w_1)
+
\frac{\vartheta(-w+ \hat{\lambda}_1)}{\vartheta(-w) \vartheta(\hat{\lambda}_1)}H_{v, w}(-w_2)
\right)\\
=&\frac{1}{\vartheta(\hbar)}\left(
\frac{\vartheta(w+ \hat{\lambda}_2)}{\vartheta(w) \vartheta(\hat{\lambda}_2)}\Phi(w_1)
+
\frac{\vartheta(w- \hat{\lambda}_1)}{\vartheta(w) \vartheta(\hat{\lambda}_1)}\Phi(w_2)
\right).
\end{align*}
This completes the proof. 
\end{proof}
\Omit{
\section{Representations of the elliptic Drinfeld currents}
This section is aim to compare the sheafified elliptic quantum group in the current paper with the category of \cite{GTL15}. 
We explain how to get the operators $\Phi_i(u), \mathfrak{X}_i^{\pm}(u, \lambda)$ 
using the current setup.

On the notes: 
$\mathfrak{X}_i^{+}(u):=g_{\lambda}(z+u): E_u \to \Hom(\mathcal{F}_{v, w},  \mathcal{F}_{v+1, w})$
is the following operator:

Let $\pi: E\times E\to E$ be the map $(z, u)\mapsto z+u$. 
For $\calP=\calP_{e_k}$ a sheaf on $E^{(e_k)}=E$, and $\mathcal{F}_{v, w}$ a sheaf on 
$E^v\times E^w$. 
We have $\mathcal{F}_{v, w}$ has support on $E^{(v)}\times E^{w-v}$. 
We have the following diagram:

\[
\xymatrix@R=1em{
& E^{(e_k)}\times E^{(v)} \times E^{(w-v-e_k)}
\ar[ld]\ar[rd]&\\
 E^{(v)}\times E^{w-v} \ar@{^{(}->}[d]&& E^{(v+e_k)}\times E^{w-v-e_k}\ar@{^{(}->}[d]\\
  E^{(v)}\times E^{w}&& E^{(v+1)}\times E^{w}
}\]

We have, $\pi^*(\calP)\boxtimes \mathcal{F}_{v, w}$ is a sheaf on $E\times E\times E^{v+w}$. 
There is a map $\pi^*(\calP)\boxtimes \mathcal{F}_{v, w}\to (\pi\times \id)^* \mathcal{F}_{v+1, w}$. 
Let $q: E\times E\times E^{v+w}\to E$ be the projection $(z, u, a)\mapsto u$. 
After taking the rational section, the action induces the following map
\begin{align*}
&q_*\Big(  \pi^*(\calP) \to 
\Hom(p_2^*\mathcal{F}_{v, w},  (\pi\times \id)^* \mathcal{F}_{v+1, w})  \Big)\\
\text{Equivalently,} \,\  &
\calP_{E_u} \to \Hom_{E_u}(\pi_{\pt}^*(\mathcal{F}_{v, w}),  q_*(\pi\times \id)^* \mathcal{F}_{v+1, w}) 
\end{align*}
}

\newcommand{\arxiv}[1]
{\texttt{\href{http://arxiv.org/abs/#1}{arXiv:#1}}}
\newcommand{\doi}[1]
{\texttt{\href{http://dx.doi.org/#1}{doi:#1}}}
\renewcommand{\MR}[1]
{\href{http://www.ams.org/mathscinet-getitem?mr=#1}{MR#1}}

\end{document}